\documentclass{article}

\usepackage[a4paper, textwidth=14cm, top=3cm]{geometry}

\title{Effective interface laws for fluid flow and solute transport through thin reactive porous  layers}
\author{M. Gahn$^{\ast}$ and M. Neuss-Radu$^{\ast\ast}$}

\date{}

\usepackage{amsmath}
\usepackage{amsthm}
\usepackage[utf8]{inputenc}
\usepackage{amsfonts}
\usepackage{amssymb}
\usepackage{enumitem}
\usepackage{color}
\usepackage{hyperref}
\usepackage{MnSymbol}
\usepackage{graphicx}

\usepackage{stmaryrd}

\usepackage{comment}

\newcommand{\xieps}{\xi_{\varepsilon}}

\newcommand{\phieps}{\phi_{\varepsilon}}
\newcommand{\vareps}{\varepsilon}
\newcommand{\oeps}{\Omega_{\varepsilon}}

\newcommand{\tceps}{\widetilde{c}_{\varepsilon}}
\newcommand{\bceps}{\bar{c}_{\varepsilon}}
\newcommand{\bueps}{\bar{u}_{\varepsilon}}

\newcommand{\x}{\bar{x}}
\newcommand{\y}{\bar{y}}
\newcommand{\rats}{\overset{2}{\rightharpoonup}}
\newcommand{\rasts}{\overset{2}{\rightarrow}}

\newcommand{\tbxfxe}{\left(t,\bar{x},\frac{x}{\varepsilon}\right)}
\newcommand{\fxe}{\frac{x}{\varepsilon}}
\newcommand{\hxi}{\widehat{\xi}}

\newcommand{\Div}{\mathrm{Div}}

\newcommand{\tZ}{\widetilde{Z}}

\newcommand{\teps}{\mathcal{T}_{\varepsilon}}

\newcommand{\R}{\mathbb{R}}
\newcommand{\Z}{\mathbb{Z}}
\newcommand{\N}{\mathbb{N}}
\renewcommand{\oe}{\Omega_{\varepsilon}}
\newcommand{\geps}{\Gamma_{\varepsilon}}

\newcommand{\veps}{v_{\varepsilon}}

\newcommand{\ueps}{u_{\varepsilon}}
\newcommand{\peps}{p_{\varepsilon}}
\newcommand{\oem}{\Omega_{\varepsilon}^M}

\newcommand{\ceps}{c_{\varepsilon}}

\newcommand{\weps}{w_{\varepsilon}}

\newcommand{\oef}{\Omega_{\varepsilon}^f}

\newcommand{\sepm}{S_{\varepsilon}^{\pm}}
\renewcommand{\x}{\bar{x}}

\newcommand{\oems}{\Omega_{\varepsilon}^{M,s}}
\newcommand{\oemf}{\Omega_{\varepsilon}^{M,f}}

\newcommand{\foe}{\dfrac{1}{\varepsilon}}
\newcommand{\bxfxe}{\left(\bar{x},\dfrac{x}{\varepsilon}\right)}

\newtheorem{definition}{Definition}
\newtheorem{remark}{Remark}
\newtheorem{theorem}{Theorem}
\newtheorem{proposition}{Proposition}
\newtheorem{lemma}{Lemma}
\newtheorem{corollary}{Corollary}

\begin{document}

\maketitle

\vspace{2em}
\begin{abstract}
  We consider a coupled model for fluid flow and transport in a domain consisting of two bulk regions separated by a thin porous layer. The thickness of the layer is of order $\vareps$ and the microscopic structure of the layer is periodic in the tangential direction also with period $\vareps$. The fluid flow is described by an instationary Stokes system, properly scaled in the fluid part of the thin layer. The evolution of the solute concentrations is described by a reaction-diffusion-advection equation in the fluid part of the domain and a diffusion equation (allowing different scaling in the diffusion coefficients) in the solid part of the layer. At the microscopic fluid-solid interface inside the layer nonlinear reactions take place. This system is rigorously homogenized in the limit $\vareps \to 0$, based on weak and strong (two-scale) compactness results for the solutions. These are based on new embedding inequalities for thin perforated layers including coupling to bulk domains. In the limit, effective interface laws for flow and transport are derived at the interface separating the two bulk regions. These interface laws enable effective mass transport through the membrane, which is also an important  feature from an application point of view.
\end{abstract}

\let\thefootnote\relax\footnotetext{$^{*}$Institute for Mathematics, University Heidelberg, \\
Im Neuenheimer Feld 205, 69120 Heidelberg, Germany, markus.gahn@iwr.uni-heidelberg.de.
	
	\vspace{.5mm}
	$^{**}$Department Mathematik, Friedrich-Alexander-Universität Erlangen-Nürnberg, Cauerstr. 11, 91058 Erlangen, Germany, maria.neuss-radu@math.fau.de }

\section{Introduction}

Thin porous layers with a complex microscopic structure, also denoted as membranes, occur in many applications from engineering, material sciences, biology and medicine. E.g., in living organisms, the endothelial layer inside blood vessels, made up of endothelial cells linked together by transmembrane proteins, forms a selective barrier that controls the exchange of solutes and water between the intra- and extravascular space. Modern experimental assays
provide more and more detailed information about structures and processes at micro and nano-scale, which has to be incorporated into microscopic mathematical models. These are the basis for the derivation of effective models which can be used to give quantitative descriptions (via numerical computations) for processes in environments involving membranes. Let us mention that transport of lipids and immune cells through the endothelial layer plays a crucial role in the formation and growth of atherosclerotic plaques in blood vessels, see e.g., \cite{Telma} where, however, effective mathematical models are derived phenomenologically.

In this paper, we rigorously derive effective mathematical models for fluid flow and reactive transport of solutes in domains consisting of two bulk regions separated by a porous membrane made up of a solid part and a pore space filled with fluid. The thickness of the membrane is of order $\vareps$ and the pore structure is periodic in tangential direction also with period $\vareps$. The microscopic interface between the two phases of the membrane, denoted by  $\Gamma_\vareps$ also has an $\vareps$-periodic structure.
The fluid flow in the bulk regions and in the pores of the membrane is described by an instationary Stokes system. A solute is transported by diffusion and advection within the fluid part of the domain whereas in the solid part of the membrane a diffusion-equation is considered. In the solid phase, different scalings of the diffusion coefficient with $\vareps^\gamma$ are considered, ranging from small diffusion ($\gamma = 1$) to fast diffusion ($\gamma =-1$). At the microscopic fluid-solid interface $\Gamma_\vareps$, nonlinear transmission conditions are formulated in terms of normal fluxes which are given by nonlinear coupling functions of the solute concentrations on both sides of the interface. In case of the endothelial membrane, the solid phase represents the region occupied by endothelial cells, which are surrounded by extracellular fluid forming the fluid part of the membrane. The nonlinear transmission conditions describe bio-chemical reactions at the microscopic  interface between the two phases. Let us mention, that the subsystem describing the fluid flow is independent from that of transport equations, whereas the velocity of the fluid contributes to the advective transport in the fluid part of the domain.  

The main contribution of this paper is the homogenization of this so called  \textit{microscopic model} in the limit $\vareps \to 0$ by rigorous methods of two scale analysis and dimension reduction.  
In the limit we obtain an \textit{effective} or \textit{macroscopic model}, where the porous membrane is reduced to an effective interface $\Sigma$ (separating the two bulk regions) for which effective interface laws are derived. It has the advantage of reduced (computational) complexity, but still maintains key features about the microgeometry and processes in the membrane. The effective flow model consists of the instationary Stokes equations in the bulk regions coupled by effective transmission conditions for the normal stresses at the interface $\Sigma$. More precisely, both, the jump in the normal component of the normal stress, and the tangent forces on each side of the interface are given in terms of the velocities and involve effective permeability coefficients associated to a Darcy-type problem in the membrane. Additionally, the normal component of the effective velocity at the interface $\Sigma$ is continuous. The effective transport problem in the fluid part of the domain consists in a reaction-diffusion-advection equation within the two bulk regions coupled by effective transmission conditions at $\Sigma$. The latter consist of the continuity of the effective concentration and a jump in the normal flux given by the homogenized nonlinear kinetics from the microscopic interface $\Gamma_\vareps$. The effective approximation for the diffusion process inside the solid phase strongly depends on the scaling of the diffusion coefficient by $\varepsilon^\gamma$. It consists in an effective reaction-diffusion equation along the interface $\Sigma $ for $\gamma=-1$, an ordinary differential equation for $\gamma \in (-1,1)$, and a so called \textit{cell problem} formulated on the standard periodicity cell in the membrane for $\gamma=1$. Hereby, the nonlinear reaction kinetics at the microscopic interface $\Gamma_\vareps$ gives rise to source terms in the homogenized equations for $\gamma \in [-1,1)$, and to a nonlinear Neumann boundary condition for $\gamma=1$. 

To the best of our knowledge, the effective model derived in this paper is the first rigorous homogenization result for advective-diffusive transport through reactive porous membranes, starting from a microscopic model involving a system of transport equations in the solid and fluid part of the medium, coupled to an instationary Stokes system in the fluid domain. A recent contribution related to our model is given in \cite{freudenberg2024homogenization}, where heat transfer through a thin grain layer separating two bulk domains was homogenized. There, only transport of solutes was considered, and the fluid velocity causing the advective transport was \textit{a priori} given and assumed to be uniformly essential bounded. 
In our situation, where the fluid flow is described by a Stokes system, we have much less integrability and the bad scaling for the gradient of the fluid velocity leads to additional difficulties in the control of the advective term inside the fluid part of the thin layer.
Furthermore, in \cite{freudenberg2024homogenization} the coupling condition at the microscopic interface $\Gamma_\vareps$ was linear, the normal fluxes being proportional to the jump of the solution, and in case of a connected layer, only the scaling $\gamma=-1$ was considered.  In \cite{gahn2022derivation}, an effective model describing fluid flow through elastic porous membranes was derived. A drawback of this model was, that no mass transport through the effective interface $\Sigma$ occured. In \cite{Orlik2023}, where an application to thin filters consisting of slender yarns in contact 
was considered, the authors extended the effective model from \cite{gahn2022derivation} by incorporating an additional (phenomenlogical) interface condition obeying Darcy's law, in order to include mass transport through the filter. 
We emphasize, that the scaling of the flow equations in the present paper (which is different from \cite{gahn2022derivation}) allows the rigorous derviation of  interface conditions which allow mass transport through $\Sigma$. 

For the derivation of the effective model, we use two-scale convergence for thin domains \cite{NeussJaeger_EffectiveTransmission} and for rapidly oscillating surfaces \cite{bhattacharya2022homogenization}, which allow to deal simultaneously with the homogenization of the periodic structures in the membrane and the reduction of the thin membrane to a (lower dimensional) interface. In a first step, we provide \textit{a priori} estimates for the solutions to the microscopic problems with an explicit dependence on the scale parameter $\vareps$. The most critical and new parts are the \textit{a priori} estimates for the concentrations in the thin perforated layer, in particular the treatment of the nonlinear coupling term on the microscopic interface $\geps$ and the control of the convective term in the thin layer. To deal with the nonlinear boundary term, we prove a trace inequality giving control of the traces on $\geps$ via the bulk domains, see Lemma \ref{lem:Trace_inequality_Bulk}. This lemma generalizes to Lipschitz domains the results from \cite[Proposition 2]{donato2019asymptotic} and \cite[Lemma 2]{freudenberg2024homogenization} proven for a rough oscillating interface given as a graph. Furthermore, we improve these results by using the $H^\beta$-norm (instead of the $H^1$-norm) in the trace inequality. To estimate the convective term in the thin layer we prove an embedding result with explicit $\vareps$-dependence of the embedding constant, for $H^1$-functions into $L^p$-functions on the fluid domain in the layer, involving an additional boundary term, see Lemma \ref{lem:embedding_thin_layer_H1_Lp}. 
Based on the \textit{a priori} estimates, we prove compactness results with respect to weak and strong  (two-scale) convergence, in the bulk domains, in the porous layer and on the interface $\geps$. Here, the strong two-scale convergence of the concentrations on $\geps$ and the convergence of the time derivative for the concentrations represent the crucial part. To cope with the nonlinear term, we extend to thin porous layers  Kolmogorov-Simon-type arguments coupled (where necessary) with estimates for shifted functions. To pass to the limit in the terms involving the time derivative, especially that of the concentration in the solid part of the layer, we prove new two-scale compactness results just based on the \textit{a priori} estimates for the sequence, see Section \ref{sec:Ts_compactness_time_derivative}. 

In summary, the novel contributions of our paper are:
\begin{itemize}
    \item Existence of a microscopic solution with $\vareps$-uniform \textit{a priori} estimates for coupled Stokes-transport equations through thin porous layers in the presence of nonlinear kinetics on rapidly oscillating microscopic interfaces
    \item Strong two-scale compactness results for sequences in thin perforated layers and on rapidly oscillating surfaces, including different scalings for the gradient and/or coupling to bulk domains.
    \item Rigorous derivation of interface laws for fluid flow through thin porous membranes, in particular expressing the normal and tangential component of the normal stress at the interface using the velocity on both sides of the interface 
    \item Derivation of macroscopic model for a coupled fluid and transport model with effective transport conditions across the interface

    \item General embedding inequalities, like trace inequality and Sobolev inequality, in thin perforated layers including coupling to bulk domains with explicit dependence on the scaling parameter $\vareps$
    \item Identification of a proper scaling of the microscopic Stokes-transport system that enables an effective mass transport through the membrane
\end{itemize}

While the rigorous derivation of a macroscopic model for the coupled flow-transport problem considered in this paper is completely new, there are several results on  subproblems included in our model, like pure reactive-diffusive transport or only consider fluid flow. Additionally, there are results on Stokes flow in thin perforated layers. From a mathematical point of view the methods developed in this paper are related to treatment of such subproblems and in the following we give an overview about such previous results. In \cite{GahnEffectiveTransmissionContinuous,GahnNeussRaduKnabner2018a,NeussJaeger_EffectiveTransmission} the pure reactive-diffusive transport through a heterogeneous layer was considered for the different scalings $\gamma \in [-1,1]$ and different interface conditions between the bulk domains and the layer. In particular, the heterogeneity in the layer was obtained by oscillating coefficients in the equations.  
Here (for the transport problem) we consider perforated thin layers, we additionally take into account advection, and also our scaling for the time-derivative in the fluid phase of the thin layer does not match these previous results, which has a significant impact on the strong compactness results in the layer. 
In \cite{donato2018homogenization} a transport problem of parabolic type in bulk domains separated by a lower dimensional rough interface (with different scaled apmplitude) given as a graph of a function was considered. Even if the geometric conditions differ from ours, some methods are closely related, like the trace inequality mentioned above.  There is a large literature on the derivation of effective interface conditions for fluid flow through sieves and filters, and here we only mention some pioneering works. The case of stationary Stokes flow through an $\vareps$-periodic filter consisting of an array of (disconnected) obstacles of size $\vareps$ is treated in \cite{sanchez1982boundary} and \cite{ConcaI1987,ConcaII1987}, where the filter was already assumed to be lower-dimensional (perforated interface). Non-Newtonian flow in a similar geometry was considered in \cite{BourgeatGipoulouxMarusicPaloka2001}. The case of tiny holes of size asymptotically behaving like order $\vareps^2$ (for $n=3$) is treated in \cite{AllaireII1991} and $\vareps^{\alpha}$ with $\alpha \in (1,2)$ in \cite{Marusic1998}. A first result regarding the derivation of the Darcy-law for Stokes flow through thin perforated layers was given in \cite{bayada1989homogenization}, where a thin layer with a rough surface, given as a graph, was considered. The two-scale convergence for thin layers without heterogeneity was later used in \cite{MarusicMarusicPalokaTwoScaleConvergenceThinDomains} for the whole thin layer. The treatment of Stokes flow through thin perforated layers, also including connected perforations, can be found in \cite{fabricius2023homogenization} and \cite{fabricius2020pressure}. We emphasize that the limit behavior is completely different in our situation, due to the coupling to the bulk domains. In fact, in our case it figures out, that the fluid pressure in the layer vanishes in the limit $\vareps\to 0$, and therefore compared to the previously cited literature on Stokes flow through thin layers, the Darcy-pressure is zero.
Finally, we mention the recent work \cite{freudenberg2024analysis}, where fluid flow coupled to heat transport in a thin layer with rough surface is considered, which is coupled over the oscillating surface to heat transport in a bulk domain. Compared to our situation, they use a different scaling for the transport equation (same as $\gamma = -1$ in our solid phase). To deal with the advective term in the thin layer, they also need the strong two-scale convergence of the temperature. In their proof for this strong convergence they refer to \cite{GahnJaegerTwoScaleTools}, where such a result was shown under better regularity conditions for the time-derivative, which is not given in \cite{freudenberg2024analysis}. Our compactness result for the solid phase in the case $\gamma  = -1$ fills this gap.

The paper is organized as follows. The microscopic model including the assumptions on the data is formulated in Section \ref{sec:micro_model}. In Section \ref{sec:main_results} the main results of the paper including the macroscopic model are stated. This section also aims to give a rough overview of the different parts of the paper and provides the important results in our work. Further, it should  help the reader to better identify the different steps in the derivation of the effective models and the required techniques. In Section \ref{sec:existence_apriori_estimates} the existence and uniqueness result for the microscopic solutions is formulated and $\vareps$-uniform \textit{a priori} estimates for the microscopic solutions are proved. Compactness results for the microscopic solutions are proved in Section \ref{sec:comptness_results}. These are based on the \textit{a priori} estimates, and for $\gamma \in (-1,1]$, on bounds for differences of shifts of the concentration in the solid part of the layer. In Section \ref{sec:derivation_macro_model}, macroscopic models including effective interface laws are derived for the fluid flow and for the transport problems. Furthermore, in Section \ref{sec:Case_Sspm_not_empty} the case when the solid phase touches the bulk domains is discussed. Finally, some auxiliary results and estimates are given in the appendix A. Some of these are standard, however, there are also new results which are of independent interest. Furthermore, in appendix B, we briefly recall the definition of two-scale convergence for thin (perforated) layers together with known compactness results and also give some new two-scale compactness results for the time derivative.

\section{The microscopic model}
\label{sec:micro_model}
We consider fluid flow and reactive transport of solutes within two bulk domains $\oeps^{\pm}$ separated by a thin porous layer. The solute is transported by diffusion and advection in the fluid domain, and by diffusion in the solid phase. At the fluid-solid interface we assume continuity of the fluxes, which are given by nonlinear reaction-kinetics modelling the transport across the interface. The fluid flow is described by the incompressible Stokes equations and is not influenced by the concentration of the solute. Before we formulate the equations for these processes in detail we start with the precise definition of the underlying microscopic geometry.

\subsection{The microscopic geometry}
\label{sec:micro_geometry}

We consider the domain $\Omega := \Sigma \times (-H, H) \subset \R^3$ with $H > 0$, and 
 $\Sigma = (a,b)\subset \R^2$ with $a,b \in \Z^2$ and $a_i < b_i$ for $i=1,2$. Further, we assume that $\vareps^{-1} \in \N$. The domain $\Omega$  consists of two bulk domains 
\begin{align*}
\oe^+ := \Sigma \times (\vareps,H),  \qquad \oe^- := \Sigma \times (-H,-\vareps),
\end{align*}
which are separated by the thin layer
\begin{align*}
\oem := \Sigma \times (-\vareps,\vareps).
\end{align*}
Within the  thin layer we have a fluid part $\oemf$ and a solid part $\oems$, which have a periodical microscopic structure. More precisely, we define  the reference cell
\begin{align*}
Z := Y\times (-1,1) := (0,1)^2 \times (-1,1),
\end{align*}
with top and bottom
\begin{align*}
S^{\pm} := Y \times \{\pm 1\}.
\end{align*}
For $n\in \N$, let us denote the interior of a set $M \subset \mathbb{R}^n$ by $\mathrm{int}(M)$.
The cell $Z$ consists of a solid part $Z_s\subset Z $, see Figure \ref{fig:FigureMicroDomain}, and a fluid part $Z_f \subset Z$ with common interface $\Gamma = \mathrm{int}\left( \overline{Z_s} \cap \overline{Z_f} \right)$. Hence, we have 
\begin{align*}
Z = Z_f \cup Z_s \cup \Gamma.
\end{align*}
We denote the top/bottom of $Z_f$ respectively $Z_s$ by
\begin{align*}
S_f^\pm = \mathrm{int}\left( \overline{Z_f} \cap S^\pm\right), \quad S_s^\pm=\mathrm{int}\left( \overline{Z_s} \cap S^\pm \right).
\end{align*}
We have to distinguish between the cases $|S_s^{\pm}| = 0$ and $|S_s^{\pm}|>0$. In our analysis we will focus on the first case and therefore we assume $Z_s$ does not touch the boundary parts $S^\pm$. In the other case we will see in Section \ref{sec:Case_Sspm_not_empty}  that the macroscopic model for the fluid flow gets somehow trivial. Further, at the crucial points in the proofs where this assumption has an impact, we will provide additional comments.
We request that $Z_f$ and $Z_s$ are open, connected   with Lipschitz-boundary, and the lateral boundary is $Y$-periodic which means that for $i=1,2$ and $\ast \in \{s,f\}$
\begin{align*}
\left(\partial Z_{\ast} \cap \{y_i = 0\}\right) + e_i = \partial Z_{\ast} \cap \{y_i=1\}.
\end{align*}
We introduce the set $K_{\vareps}:= \{k \in \Z^2 \times \{0\} \, : \, \vareps(Z + k) \subset \oem\}$.  Clearly, we have $$\oem = \mathrm{int}\left(\bigcup_{k\in K_{\vareps}} \vareps(\overline{Z} + k)\right).
$$
Now, we define the fluid and solid part of the membrane, see Figure \ref{fig:FigureMicroDomain}, by
\begin{align*}
\oemf := \mathrm{int} \left( \bigcup_{k \in K_{\vareps}} \vareps \left(\overline{Z_f} + k \right) \right),
\quad 
\oems := \mathrm{int} \left( \bigcup_{k \in K_{\vareps}} \vareps \left(\overline{Z_s} + k \right) \right).
\end{align*}
The fluid-solid interface between the solid and the fluid part in the membrane is denoted by 
\begin{align*}
\geps := \mathrm{int}\left( \overline{\oems} \cap \overline{\oemf}\right).
\end{align*}
The interface between the fluid part in the membrane and the bulk domains is defined by
\begin{align*}
\sepm := \Sigma \times \{\pm \vareps\}.
\end{align*}
Again, if  we allow the solid part of the thin layer  to touch the bulk domains, the interface $\sepm$ is the union of the sets 
\begin{align*}
   S_{\vareps,s}^{\pm} := S_{\vareps}^{\pm} \setminus \partial \oemf, \qquad S_{\vareps,f}^{\pm}:= S_{\vareps}^{\pm} \setminus \partial \oems,  
\end{align*}
where $S_{\vareps,s}^{\pm}$ is the interface between the solid part in the membrane and the bulk domains, and similarly for the index $f$.
Altogether, we have the following decomposition of the domain $\Omega$
\begin{align*}
\Omega &= \oe^+ \cup \oe^- \cup \oem \cup  S_{\vareps}^+ \cup S_{\vareps}^-
\\
&= \oe^+ \cup \oe^- \cup \oems \cup \oemf \cup \geps \cup  S_{\vareps}^+ \cup S_{\vareps}^-.
\end{align*}
The whole fluid part is defined by
\begin{align*}
\oef := \Omega \setminus \overline{\oems}.
\end{align*}
\begin{figure}
\hspace{1cm}\begin{minipage}{4.6cm}
\centering
\includegraphics[scale=0.1025]{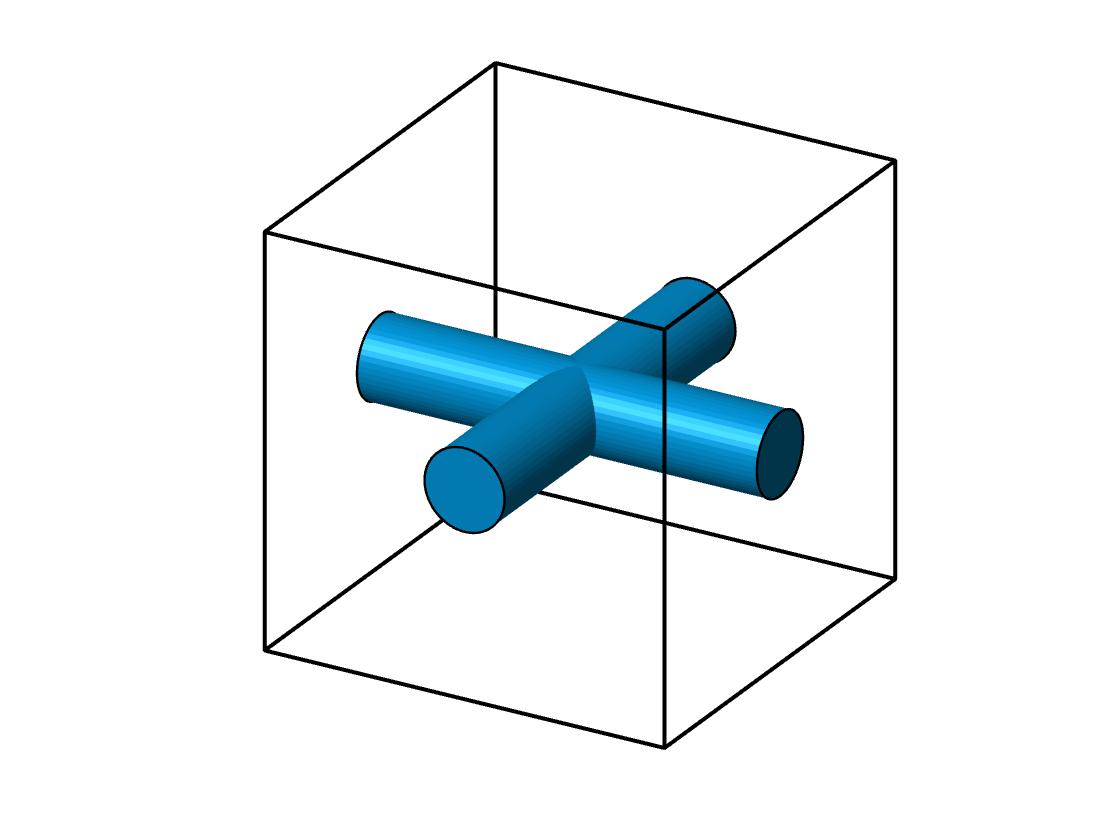}
\end{minipage} 
\hspace{0.3cm}
\begin{minipage}{4.7cm}
\centering
\includegraphics[scale=0.1725]{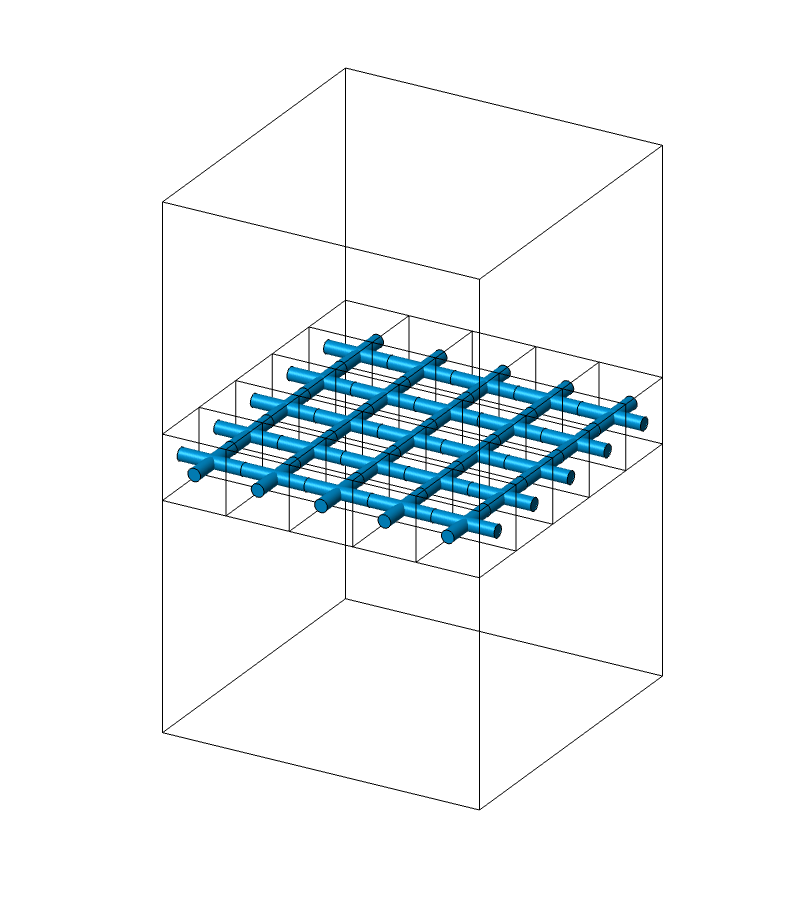}
\end{minipage}
\caption{Left: A reference cell $Z$ for the porous layer with the solid part $Z_s$ highlighted by the coloring. 
Right: The microscopic domain $\Omega$ with the porous layer $\oem$ consisting of the fluid part $\oemf$ and the solid part $\oems$.}
\label{fig:FigureMicroDomain}
\end{figure}
We further assume that the domains $\oef$, $\oemf$, and $\oems$ are connected and Lipschitz.
The upper and lower boundary of $\Omega$ is denoted by
\begin{align*}
    \partial_N \Omega &:= \bigcup_{\pm}\Sigma \times \{\pm H\}.
\end{align*}
Here a stress boundary condition is assumed for the fluid flow and a flux boundary condition for the transport equation.

In the limit $\vareps \to 0 $ the thin layer $\oem$ is reduced to the interface $\Sigma$ and the domains $\oeps^{\pm}$ converge to the  macroscopic bulk domains $\Omega^{\pm}$ defined by
\begin{align*}
\Omega^+ := \Sigma \times (0,H),
\quad 
\Omega^- := \Sigma \times (-H,0).
\end{align*}
The outer unit normal to $\Omega^\pm$ at $\Sigma$ is given by $\nu^\pm=\mp e_3$, where $e_3$ is the third standard unit vector in $\R^3$. Further, we denote the top respectively the bottom of $\Omega$ by
\begin{align*}
    \partial_N \Omega^\pm = \Sigma \times \{\pm H\}.
\end{align*}

\subsection{Notations } 

For $U\subset \R^m$ with $m \in \N$ we denote by $L^p(U)$ with $p \in [1,\infty]$ the usual Lebesgue spaces, and for $m \in \N_0$ the Sobolev space is defined by $W^{1,p}(U)$. For $p=2$ we shortly write $H^1(U):= W^{1,2}(U)$. For $\beta \in (0,1)$ we denoty by $H^{\beta}(U)$ the Sobolev-Slobodeckii space (for $p=2$). For norms of vector valued functions with values in $\R^n$ for $n\in \N$, we usually skip the exponent, for example we write $\|\cdot\|_{L^2(U)} := \|\cdot\|_{L^2(U)^n}$.

We use the notation $\bar{x} = (x_1, x_2) \in \Sigma$ for a vector $x= (x_1, x_2, x_3)\in \mathbb{R}^3$. For an arbitrary function $\phieps : \oef \rightarrow \R^m$ for $m\in \N$ we define the restrictions to the bulk domains and the fluid part of the membrane by
\begin{align*}
\phieps^{\pm} := \phieps\vert_{\oe^{\pm}}, \qquad \phieps^M := \phieps\vert_{\oemf}.
\end{align*}
For function spaces we use the index $\#$ to indicate functions which are periodic with respect to the first two variables. More precisely, we have 
\begin{align*}
    C_{\#}^{\infty}(\Omega):= \left\{ \phi \in C^{\infty}(\R^2 \times [-H,H] ) \, : \, \phi \mbox{ is } \Sigma\mbox{-periodic}\right\}.
\end{align*}
Here, a $\Sigma$-periodic function $\phi: \R^2 \times [-H,H]\rightarrow \R$ fulfills $\phi(k + x) = \phi(x)$ for all $x \in \R^2\times [-H,H]$ and $k:= (b-a,0) \in \Z^2 \times \{0\}$ (see Section \ref{sec:micro_geometry} for definition of $\Sigma$ and $a,b$). Then $H_{\#}^1(\Omega)$ is the closure of $C_{\#}^{\infty}(\Omega)$ with respect to the $H^1(\Omega)$-norm. In a similar way, we define the space $H_{\#}^1(\oeps^M)$, and we denote by $H^1_{\#}(\oeps^{M,\ast})$ for $\ast \in \{s,f\}$ the restriction of $H_{\#}^1(\oeps^M)$-functions on $\oeps^{M,\ast}$. 
Further, we define  
$$
C_{\#}^{\infty}(Z):= \left\{ \phi \in C^{\infty}(\R^2 \times [-1,1] ) \, : \, \phi \mbox{ is } Y\mbox{-periodic}\right\},
$$ 
and denote by $H_{\#}^1(Z)$ its closure with respect to the $H^1(Z)$-norm, and by $H^1_{\#}(Z_{\ast})$ for $\ast \in \{s,f\}$ the restriction of $H^1_{\#}(Z)$-functions on $Z_{\ast}$. We emphasize that for $L^p$-spaces we avoid to write $L^p_{\#}$, since these functions have no traces and if not stated otherwise we extend such functions periodically with respect to the first two components.

For a Lipschitz domain $U\subset \R^n$ with $n\in \N$ and $\omega \subset \partial U$, we define
\begin{align*}
H^1(U,\omega):= \left\{ u \in H^1(U)\, :\, u= 0 \mbox{ on } \omega\right\}.
\end{align*}
Further, $H^1_{\#}(\oeps^{M,\ast},\geps)$  for $\ast \in \{f,s\}$ is the space of functions from $H^1(\oeps^{M,\ast},\geps)$ which are $\Sigma$-periodic.

For a Banach space $B$ its dual is given by $B^{\ast}$ and we denote the duality pairing between the dual space $B^{\ast}$ and $B$ by $\langle \cdot ,\cdot \rangle_B$. Further, for an arbitrary open set $U\subset \R^m$ with $m\in \N$ we write $L^p(U,B)$ for the usual Bochner spaces (for $p\in [1,\infty]$).

\subsection{The microscopic problem}
Now, we formulate the microscopic model together with the assumptions on the data and give the definition of a weak solution of the problem.\\

\noindent \textbf{Reminder:} We only consider the case when the solid phase $\oems$ does not touch the bulk domains $\oeps^{\pm}$, more precisely we have $|S_s^{\pm}|= 0$. For the other case we refer to Section \ref{sec:Case_Sspm_not_empty}.
\\

In the fluid part $\oef$ we have the fluid velocity $\veps = (\veps^+,\veps^M,\veps^-): (0,T)\times \oef \rightarrow \R^3$ and the fluid pressure $\peps = (\peps^+,\peps^M,\peps^-) : (0,T)\times  \oef \rightarrow \R$.  Further, we consider the concentration $\ceps^f: (0,T)\times \oef \rightarrow \R$  in the fluid domain, and the concentration $\ceps^s: (0,T)\times \oems \rightarrow \R$ in the solid domain.
The evolution of the  velocity and pressure of the fluid is given by
\begin{subequations}\label{MicroscopicModel}
\begin{align}
\label{def:micro_equations_fluid_pde}
\partial_t \veps^{\pm} - \nabla \cdot D( \veps^{\pm}) + \nabla \peps^{\pm}  &= f_{\vareps}^{\pm} &\mbox{ in }& (0,T)\times \oeps^{\pm},
\\
 \partial_t \veps^M- \vareps\nabla \cdot D( \veps^M )+  \foe\nabla \peps^M  &= 0 &\mbox{ in }& (0,T)\times \oemf,
\\
\nabla \cdot \veps &= 0 &\mbox{ in }& (0,T)\times \oef,
\\
\left(-\peps I + D(\veps) \right)\cdot \nu &= 0 &\mbox{ on }& (0,T) \times \partial_N \Omega,
\\
\label{MicroModelBCDirichlet}
\veps &= 0 &\mbox{ on }& (0,T) \times \geps,
\\
\label{ContinuityVelocity}
\veps^\pm & = \veps^M & \mbox{ on }& (0,T) \times \sepm,
\\
\label{ContinuityNormalStress}
\left(-\peps^\pm I + D(\veps^\pm) \right)\cdot \nu^\pm = \left(-\frac{1}{\vareps}\peps^M I + \vareps D(\veps^M) \right)& \cdot \nu^\pm & \mbox{ on }& (0,T) \times \sepm,
\\
\veps(0) &= 0 &\mbox{ in }& \oef,
\\
\label{def:micro_equations_fluid_periodicBC}
\veps \mbox{ is } &\Sigma\mbox{-periodic}.
\end{align}
Here,  $D(\ueps):= \frac12 \left(\nabla \ueps + \nabla \ueps^T\right)$ denotes the symmetric gradient, $f_{\vareps}^\pm$ are the bulk forces, $\nu$ is the outer unit normal of $\oef$. At the interface $\sepm$ we assume the natural transmission conditions \eqref{ContinuityVelocity}-\eqref{ContinuityNormalStress} describing the continuity of the velocity and of the normal stresses. Here $\nu^\pm = \mp e_3$ is the outer unit normal to $\oeps^\pm$ on $\sepm$, where $e_3$ is the third standard unit vector in $\R^3$.
The transport equations for $\ceps^f$ and $\ceps^s$ are given by
\begin{align}
\partial_t \ceps^f - \nabla \cdot (D^f \nabla \ceps^f - \veps \ceps^f ) &= 0 &\mbox{ in }& (0,T)\times \oef,
\\
\foe \partial_t \ceps^s - \vareps^{\gamma} \nabla \cdot (D^s \nabla \ceps^s) &= 0 &\mbox{ in }& (0,T)\times \oems,
\\
-(D^f \nabla \ceps^f - \veps \ceps^f) \cdot \nu &= 0 &\mbox{ on }& (0,T)\times \partial_N \Omega,
\\
-(D^f \nabla \ceps^f - \veps \ceps^f) \cdot \nu = - \vareps^{\gamma} D^s \nabla \ceps^s \cdot \nu &= h(\ceps^f,\ceps^s) &\mbox{ on }& (0,T)\times \geps,
\\
\ceps^f(0) &= c_{\vareps,\mathrm{in}}^f &\mbox{ in }& \oef,
\\
\ceps^s(0) &= c_{\vareps,\mathrm{in}}^s &\mbox{ in }& \oems,
\\
\ceps^s,\, \ceps^f &\mbox{ are } \Sigma\mbox{-periodic},
\end{align}
\end{subequations}
with $\gamma \in [-1,1]$.

\begin{remark}\
\begin{enumerate}
[label = (\roman*)]

\item  For the sake of simplicity we put all data except the bulk forces equal to zero. This also corresponds to our aim to consider the influence of the bulk domain on the membrane in this paper, as this is of particular importance for applications. However, the results can be generalized to inhomogeneous forces in the membrane, see \cite{gahn2024derivation} for more details. It is also possible to consider non-linear reaction terms (with similar properties as $h$) in the transport equations.

\item In this paper we treat the physical relevant case $n=3$.  The case $n=2$ is excluded, since we assume that $\oemf$ and $\oems$ are both connected. From a mathematical point of view, we can not treat the case $n>3$, since we need the continuous embedding $H^1(\Omega) \hookrightarrow L^4(\partial_N \Omega)$.

\item On the lateral boundary $\partial \Sigma \times (-H,H)$ we consider periodic boundary conditions. This is for technical reasons. Similar boundary conditions as in \cite{GahnNeussRaduKnabner2018a} for the transport equations and \cite{gahn2022derivation} for the Stokes-system are possible, but lead to more technical difficulties in particular when estimating the shifts near the boundary, see Lemma \ref{lem:apriori_shifts} below.

\item For the transport problem in the fluid domain no interface conditions are needed at $\sepm$ (in comparison to \eqref{ContinuityVelocity}-\eqref{ContinuityNormalStress}), since we have the same scaling in the bulk domains and in the fluid part of the membrane.
\end{enumerate}
\end{remark}

\noindent\textbf{Assumptions on the data:}
\begin{enumerate}[label = (A\arabic*)]
\item \label{ass:rhs_feps_pm} For the bulk forces we assume $f_{\vareps}^\pm \in L^2(\oeps^\pm)$ and there exists $f_0^\pm \in L^2(\Omega^{\pm})$ such that $\chi_{\oeps^{\pm}} f_{\vareps}^\pm \rightharpoonup f_0^\pm $ in $L^2(\Omega^\pm)$. In particular we have that $\chi_{\oeps^{\pm}} f_{\vareps}^\pm$ is bounded in $L^2(\Omega^\pm)$.
\item\label{ass:diffusion_coefficients} The diffusion coefficients fulfill $D_f,\, D_s>0$.

\item \label{ass:h} The reaction-kinetics $h:\R\times \R \rightarrow \R$ is (globally) Lipschitz continous. In particular, it holds for all $(s_1,s_2) \in \R\times \R$ that
\begin{align*}
    |h(s_1,s_2)|\le C (1 + |s_1| + |s_2|).
\end{align*}

\item \label{ass:initial_values_ceps_f} We assume $c_{\vareps,\mathrm{in}}^f \in L^2(\oef)$ and there exists $c_{\mathrm{in}}^f \in L^2(\Omega)$ such that $\chi_{\oef} c_{\vareps,\mathrm{in}}^f \rightharpoonup c_{\mathrm{in}}^f $ in $L^2(\Omega)$. In particular we have that $\chi_{\oef}c_{\vareps,\mathrm{in}}^f$ is bounded in $L^2(\Omega)$.

\item \label{ass:initial_values_ceps_s} We assume $c_{\vareps,\mathrm{in}}^s\in L^2(\oems)$ and there exists $c_{\mathrm{in}}^s \in L^2(\Sigma \times Z_s)$ such that $\chi_{\oems}c_{\vareps,\mathrm{in}}^s \rats c_{\mathrm{in}}^s.$
In particular, we have
\begin{align*}
   \|c_{\vareps,\mathrm{in}}^s\|_{L^2(\oems)} \le C\sqrt{\vareps}.
\end{align*}
Further, we denote the mean value of the limit function $c_{\mathrm{in}}^s$ by 
\begin{align*}
   \bar{c}_{\mathrm{in}}^s := \frac{1}{|Z_s|} \int_{Z_s} c_{\mathrm{in}}^s dy.
\end{align*}
Additionally, for $\gamma \in (-1,1]$ we assume that for every sequence $l_{\vareps} \in \vareps \Z^2 \times \{0\}$ with $l_{\vareps} \to 0$ for $\vareps \to 0$ it holds that
\begin{align*}
    \frac{1}{\sqrt{\vareps}}\|c_{\vareps,\mathrm{in}}^s (\cdot + l_{\vareps  }) - c_{\vareps,\mathrm{in}}^s \|_{L^2(\oems)} \overset{\vareps \to 0}{\longrightarrow} 0.
\end{align*}
\end{enumerate}

The weak formulation of the microscopic model $\eqref{MicroscopicModel}$ reads as follows:

\begin{definition}\label{def:Weak_Micro_Model}
Let $\gamma \in [-1,1]$. We say that $(\veps,\peps,\ceps^f,\ceps^s)$ is a weak solution of the microscopic model $\eqref{MicroscopicModel}$, iff 
\begin{align*}
\veps &\in L^2((0,T),H^1_{\#}(\oef))^3 \cap H^1((0,T),L^2(\oef))^3
\\
\peps &\in L^2((0,T),L^2(\oef)),
\\
\ceps^f &\in L^2((0,T),H^1_{\#}(\oef)) \cap H^1((0,T),H^1_{\#}(\oef)^{\ast}),
\\
\ceps^s &\in L^2((0,T),H^1_{\#}(\oems)) \cap H^1((0,T),H^1_{\#}(\oems)^{\ast})
\end{align*} 
with $\veps = 0$ on $\geps$ and $\nabla \cdot \veps = 0$, and  for all $\phieps \in H^1_{\#}(\oef)^3$ with $\phieps = 0 $ on $\geps$ and all $\xieps^f \in H^1_{\#}(\oef)$ and $\xieps^s \in H^1_{\#}(\oems)$ it holds almost everywhere in $(0,T)$ that
\begin{align}
\begin{aligned}
\label{eq:Var_Micro_veps}
\sum_{\pm}& \left\{ \int_{\oeps^{\pm}} \partial_t \veps^{\pm} \cdot \phieps dx +  \int_{\oeps^{\pm}}  D( \veps^{\pm}) : D( \phieps) dx - \int_{\oeps^{\pm}} \peps^{\pm} \nabla \cdot \phieps dx\right\}
+  \int_{\oemf} \partial_t \veps^M \cdot \phieps dx 
\\
&+ \vareps \int_{\oemf} D( \veps^M) : D( \phieps) dx 
 -\foe  \int_{\oemf} \peps^M \nabla \cdot \phieps dx 
=\sum_{\pm} \int_{\oeps^{\pm}} f_{\vareps}^{\pm} \cdot \phieps dx,
\end{aligned}
\\
\begin{aligned}
\label{eq_Var_Micro_cepsf}
\langle \partial_t \ceps^f , \xieps^f\rangle_{H^1_{\#}(\oef)} + \int_{\oef} \left[ D^f \nabla \ceps^f - \veps \ceps^f\right] \cdot \nabla \xieps^f dx  &= - \int_{\geps} h(\ceps^f,\ceps^s) \xieps^f d\sigma,
\end{aligned}
\\
\begin{aligned}
\label{eq_Var_Micro_cepss}
\foe \langle \partial_t \ceps^s , \xieps^s\rangle_{H^1_{\#}(\oems)} +  \int_{\oems}  \vareps^{\gamma} D^s \nabla \ceps^s  \cdot \nabla \xieps^s dx  &= \int_{\geps} h(\ceps^f,\ceps^s) \xieps^s d\sigma.
\end{aligned}
\end{align}
Further, the initial conditions $\veps(0) = 0$, $\ceps^f(0) = c_{\vareps,\mathrm{in}}^f$, and $\ceps^s(0) = c_{\vareps,\mathrm{in}}^s$ are valid in the  $L^2$-sense.
\end{definition}

\section{Formulation of the main results and the macroscopic model}
\label{sec:main_results}
We will show that the microscopic solutions $(\veps,\peps,\ceps^f,\ceps^s)$ converge in a suitable sense to limit functions $(v_0,p_0,c_0^f,c_0^s)$. In the following we formulate the macroscopic problem solved by these limit functions and sketch the procedure for the derivation of the macro model with the crucial steps. Finally, we formulate the main results of the paper.
Since the fluid problem is not influenced by the concentration, we formulate the limit models for the fluid flow and the transport separately. 

\subsection{The macroscopic fluid model with effective interface conditions}

For the weak solution $(\veps,\peps)$ of the microscopic fluid problem we first consider separately the solutions $(\veps^{\pm},\peps^{\pm})$ in the bulk domains $\oeps^{\pm}$ and the solutions $(\veps^M,\peps^M)$ in the fluid part of the thin layer $\oemf$. First of all, we obtain the \textit{a priori} estimate 
\begin{align*}
   \|\partial_t \veps^{\pm}\|_{L^2((0,T)\times \oeps^{\pm})} +    \|\veps^{\pm}\|_{L^{\infty}((0,T),H^1(\oeps^{\pm}))} + \|\peps^{\pm}\|_{L^2((0,T)\times \oeps^{\pm})} &\le C.
\end{align*}
We emphasize that for the estimate of the pressure it is essential to have the pressure boundary condition at $\partial_N \Omega^{\pm}$. Assuming for a moment that $\veps^{\pm}$ and $\peps^{\pm}$ are defined on the whole fixed bulk domains $\Omega^{\pm}$ and the previous estimates are still valid,  we obtain by standard results for weak and strong convergence in $L^2$ spaces the existence of $v_0^{\pm} \in L^2((0,T),H^1_{\#}(\Omega^{\pm}))^3$ (and additional time-regularity) and $p_0^{\pm} \in L^2((0,T)\times \Omega^{\pm})$ such that up to a subsequence
\begin{align*}
    \veps^{\pm} \rightharpoonup v_0^{\pm} \quad \mbox{in } L^2((0,T),H^1 (\Omega^{\pm}))^3,\qquad \peps^{\pm} \rightharpoonup p_0^{\pm} \quad \mbox{in } L^2((0,T)\times \Omega^{\pm}).
\end{align*}
For the precise convergence results and regularity for the limit functions we refer to Proposition \ref{prop:compactness_fluid}. Hence, we can directly pass to the limit $\vareps \to 0$ in the bulk terms in the variational equation $\eqref{eq:Var_Micro_veps}$ (choosing test-functions vanishing on $\Sigma$) and obtain that $(v_0^{\pm},p_0^{\pm})$ solve in a weak sense
\begin{align*}
    \partial_t v_0^{\pm} - \nabla\cdot D(v_0^{\pm}) + \nabla p_0^{\pm} &= f_0^{\pm} &\mbox{ in }& (0,T)\times \Omega^{\pm},
\\
\nabla \cdot v_0^{\pm} &= 0 &\mbox{ in }& (0,T)\times \Omega^{\pm},
\end{align*}
together with homogeneous initial condition, periodic boundary conditions on the lateral boundary, and a pressure boundary condition on $\partial_N \Omega^{\pm}$. Now, the crucial question is the interface condition across $\Sigma$. For this, we have to consider the limit problem for $(\veps^M,\peps^M)$. We have 
\begin{align*}
     \|\partial_t \veps^M&\|_{L^2((0,T)\times \oemf)} +     \frac{1}{\sqrt{\vareps}} \|\veps^M\|_{L^{\infty}((0,T),L^2(\oemf))}
     \\
     &+ \sqrt{\vareps} \|\nabla \veps^M\|_{L^{\infty}((0,T), L^2( \oemf))} 
     + \foe  \|\peps^M\|_{L^2((0,T)\times \oemf)} \le C.
\end{align*}
Here a crucial point is to estimate the pressure, where we construct a suitable Bogovskii-operator. From the two-scale convergence theory in thin layers (see Section \ref{sec:two_scale_convergence_unfolding} for definitions, notations, and basic compactness results) we obtain 
\begin{align*}
    \peps^M \rats 0.
\end{align*}
In other words, the zeroth order approximation $p_0^M$ of $\peps^M$ is equal to zero. This is a significant difference to Stokes flow in perforated thin layers without coupling to bulk domains, where the term $p_0^M$ is the Darcy-pressure and given as the solution of the Darcy-equation, see for example \cite{fabricius2023homogenization} or \cite{fabricius2020pressure}. The bounds for $\veps^M$ imply the following convergences in the two-scale sense (up to a subsequence) 
\begin{align*}
    \veps^M \rats v_0^M,\qquad \vareps \nabla \veps^M \rats \nabla_y v_0^M
\end{align*}
for a limit function $v_0^M \in L^2((0,T)\times \Sigma,H_{\#}^1(Z_f,\Gamma))^3$. The time-derivative vanishes in the two-scale limit, and we can pass to the limit in the fluid part of the layer $\oemf$ and obtain that $v_0^M$ solves in a weak sense
\begin{align}
\begin{aligned}
\label{Stokes_Problem_Cell_Main}
    -\nabla_y \cdot D_y(v_0^M) + \nabla_y p_1^M &= 0 &\mbox{ in }& (0,T)\times \Sigma \times Z_f,
\\
\nabla_y \cdot v_0^M &= 0 &\mbox{ in }& (0,T)\times \Sigma \times Z_f,
\\
v_0^M &= 0 &\mbox{ on }& (0,T)\times \Sigma \times \Gamma,
\\
v_0^{\pm} &= v_0^M &\mbox{ on }& (0,T)\times \Sigma \times S^{\pm},
\end{aligned}
\end{align}
together with $Y$-periodic boundary conditions. The pressure term $p_1^M$ has to be constructed by using a Bogovskii-type argument, which is not straightforward, since we have to take into account the coupling to the bulk domains.  Further, we need additional boundary conditions on $S^{\pm}$, which can be derived from the continuity of the fluid velocity $\veps^{\pm} = \veps^M$ on $S_{\vareps}^{\pm}$, and we obtain
\begin{align}
\label{eq:continuity_velo_macro}
    v_0^M = v_0^{\pm} \quad\mbox{on } (0,T)\times \Sigma \times S^{\pm},
\end{align}
implying that $v_0^M$ is constant with respect to $y$ on the top and bottom of the cell $Z_f$. This interface condition gives the coupling condition between the equation for $(v_0^{\pm},p_0^{\pm})$ in the bulk domains $\Omega^{\pm}$ and the cell problems for $v_0^M$ on $\Sigma$. Further, this condition implies together with the incompressibility condition $\nabla_y \cdot v_0^M = 0$ the continuity of the normal velocity of $v_0$ across $\Sigma$, more precisely, we have 
\begin{align*}
    [v_0^+]_3 = [v_0^-]_3 \quad\mbox{ on } (0,T)\times \Sigma.
\end{align*}
To complete the system for $v_0$, we need a boundary condition for the tangential component of the normal stress and the jump of the normal component of the normal stress across $\Sigma$.  
The linearity of the Stokes problem $\eqref{Stokes_Problem_Cell_Main}$ and $\eqref{eq:continuity_velo_macro}$ allows to express $v_0^M$ and $p_1^M$ via $v_0^{\pm}|_{\Sigma}$ in the following way
\begin{align*}
    v_0^M(t,\x,y) &= \sum_{\pm} \sum_{i=1}^2 [v_0^{\pm}]_i(t,\x,0)  q_i^{\pm}(y) + [v_0]_3 q_3(y),
    \\
    p_1^M(t,\x,y) &= \sum_{\pm}\sum_{i=1}^2 [v_0^{\pm}]_i(t,\x,0) \pi_i^{\pm}(y) + [v_0]_3 \pi_3(y),
\end{align*}
where $(q_i^{\pm},\pi_i^{\pm})$ and $(q_3,\pi_3)$ are the solutions of suitable cell problems, see $\eqref{Cell_problems_fluid_12}$ and $\eqref{Cell_problems_fluid_3}$, and $[v_0]_3:= [v_0^{\pm}]_3$.
Now, formally we obtain from the continuity of the normal stresses in $\eqref{ContinuityNormalStress}$ that
\begin{align*}
    -\left[D(v_0^{\pm} - p_0^{\pm} I \right] \nu^{\pm} = -\left[D(v_0^M) - p_1^M I \right] \nu^{\pm} \quad\mbox{ on } (0,T)\times \Sigma \times S^{\pm}.
\end{align*}
We can split this equation in its normal and tangetial part. After integration with respect to $y$ over $S^{\pm}$ and using the representation of $v_0^M$ in $\eqref{id:representation_v0M_p1M}$ and the properties of the cell solution, we obtain after a long but straightforward computation
\begin{align*}
-\llbracket (D (v_0) - p_0 I ) \nu\cdot \nu\rrbracket &= K^+v_0^+ \cdot \nu^+ -K^-v_0^- \cdot \nu^- &\mbox{ on }& (0,T)\times \Sigma,
\\
[(D(v_0^{\pm}) - p_0^{\pm}I)\nu^{\pm}]_t &= -[K^{\pm}v_0^{\pm}]_t - Mv_0^{\mp} &\mbox{ on }& (0,T)\times \Sigma,
\end{align*}
where $\llbracket \phi \rrbracket:= \phi^+ - \phi^-$ on $\Sigma$ denotes the jump across $\Sigma$ and $[\cdot]_t $ the tangential part of a vector field. The effective coefficients $K^{\pm} \in \R^{3\times 3}$ and $M\in \R^{3\times 3}$ are defined in $\eqref{def:Kpm}$ and $\eqref{def:M}$ in Section \ref{sec:derivation_fluid_model}, and are given via suitable Stokes-cell problems.  To make this argument rigorous we have to work with the weak formulation for the limit functions $(v_0^+,v_0^M,v_0^-)$ and $(p_0^+,p_0^M,p_0^-)$ and we will not do explicitly the formal computation mentioned above.
\\

Altogether, we showed that the limit functions $(v_0^\pm,p_0^\pm)$ defined 
on $\Omega^{\pm}$ are the unique weak solution of the following macroscopic equation for the fluid flow:
\begin{subequations}\label{def:Macro_Stokes_model_strong}
\begin{align}
\label{def:Macro_Stokes_model_PDE}
    \partial_t v_0^{\pm} - \nabla\cdot D(v_0^{\pm}) + \nabla p_0^{\pm} &= f_0^{\pm} &\mbox{ in }& (0,T)\times \Omega^{\pm},
\\
\nabla \cdot v_0^{\pm} &= 0 &\mbox{ in }& (0,T)\times \Omega^{\pm},
\\
[v_0^+]_3 &= [v_0^-]_3 &\mbox{ on }& (0,T)\times \Sigma,
\\
-[D(v_0^{\pm}) - p_0^{\pm}I]\nu &= 0 &\mbox{ on }& (0,T)\times \Sigma \times \{\pm H\},
\\
\label{Macro_Model_Fluid_Stress_normal}
-\llbracket (D (v_0) - p_0 I ) \nu\cdot \nu\rrbracket &= K^+v_0^+ \cdot \nu^+ -K^-v_0^- \cdot \nu^- &\mbox{ on }& (0,T)\times \Sigma,
\\
\label{Macro_Model_Fluid_Stress_tangential}
[(D(v_0^{\pm}) - p_0^{\pm}I)\nu^{\pm}]_t &= -[K^{\pm}v_0^{\pm}]_t - Mv_0^{\mp} &\mbox{ on }& (0,T)\times \Sigma,
\\
v_0^{\pm}(0) &= 0 &\mbox{ in }& \Omega^{\pm},
\\
v_0^{\pm} \,\, \Sigma\mbox{-periodic}.
\end{align}
\end{subequations}
We call the tuple $(v_0^\pm,p_0^\pm)$ a weak solution of $\eqref{def:Macro_Stokes_model_strong}$, if 
\begin{align*}
    v_0^{\pm} &\in L^2((0,T),H^1_{\#}(\Omega^{\pm}))^3 \cap H^1((0,T),L^2(\Omega^{\pm}))^3, \mbox{ and }
    p_0^{\pm} \in L^2((0,T)\times \Omega^{\pm})
\end{align*}
with $[v_0^+]_3 = [v_0^-]_3$ on $(0,T)\times \Sigma$ and $\nabla \cdot v_0^{\pm} = 0$, and for all $(\phi^+,\phi^-) \in H^1_{\#}(\Omega^+)^3 \times H^1_{\#}(\Omega^-)^3$ with $\phi^+_3 = \phi^-_3$ on $\Sigma$ it holds almost everywhere in $(0,T)$
\begin{align}
\begin{aligned}\label{eq:var_macro_fluid}
\sum_{\pm} & \left\{ \int_{\Omega^{\pm}} \partial_t v_0^{\pm} \cdot \phi^{\pm} dx +\int_{\Omega^{\pm}} D(v_0^{\pm}) : D(\phi^{\pm}) dx -  \int_{\Omega^{\pm}} p_0^{\pm} \nabla \cdot \phi^{\pm} dx   \right\} 
\\
&+ \sum_{\pm} \int_{\Sigma} K^{\pm} v_0^{\pm} \cdot \phi^{\pm} d\x + \int_{\Sigma} Mv_0^- \cdot \phi^+ + Mv_0^+ \cdot \phi^- d\x = \sum_{\pm} \int_{\Omega^{\pm}} f^{\pm}_0 \cdot \phi^{\pm} dx  .
\end{aligned}
\end{align}

The weak equation $\eqref{eq:var_macro_fluid}$ is obtained with an elemental calculation by formally multiplying $\eqref{def:Macro_Stokes_model_PDE}$ with suitable test functions, integrating by parts, and decomposing the normal stress into its normal- and tangential part and using the interface conditions in $\eqref{def:Macro_Stokes_model_strong}$.
Finally, we summarize the results above in the following theorem:
\begin{theorem}\label{MainThm:Fluid}
Let $(\veps,\peps)$ be the unique weak solution of the microscopic Stokes problem $\eqref{def:micro_equations_fluid_pde}$-$\eqref{def:micro_equations_fluid_periodicBC}$. Then, there exist $v_0^{\pm} \in L^2((0,T)\times \Omega^{\pm})^3$ and $p_0^{\pm} \in L^2((0,T)\times \Omega^{\pm})$ such that 
\begin{align*}
    \chi_{\oeps^{\pm}} \veps^{\pm} \rightharpoonup v_0^{\pm} \quad\mbox{in } L^2((0,T)\times \Omega^{\pm})^3, \qquad \chi_{\oeps^{\pm}} \peps^{\pm} \rightharpoonup p_0^{\pm} \quad\mbox{in } L^2((0,T)\times \Omega^{\pm}).
\end{align*}
The tuple $(v_0^\pm,p_0^\pm)$ is the unique weak solution of the macroscopic problem $\eqref{def:Macro_Stokes_model_strong}$.

Additionally, there exists $v_0^M \in L^2((0,T)\times \Sigma \times Z_f)^3$, such that $\veps^M$ converges in the two-scale sense to $v_0^M$. The Darcy-velocity $\bar{v}_0^M \in L^2((0,T)\times \Sigma)^3$  defined as the average of $v_0^M$  fulfills
\begin{align}\label{Darcy_velocity_main_result}
\bar{v}_0^M := \frac{1}{|Z_f|}\int_{Z_f} v_0^M dy =  \sum_{\pm} Q^{\pm}v_0^{\pm}|_{\Sigma} + \frac{|Z|}{|Z_f|}[v_0]_3 e_3,
\end{align}
where $[v_0]_3:= [v_0^{\pm}]_3$ and the entries of the matrix $Q^{\pm}$ are given by \eqref{MatrixQbar}.
\end{theorem}

\subsection{The macroscopic transport equations}

While the limit equation for the solute concentration in the fluid part has the same structure for every $\gamma$, for the solid phase we have to distinguish between the cases $\gamma = -1$, $\gamma \in (-1,1)$, and $\gamma = 1$. Before we formulate the macro-model in detail, let us briefly summarize the main points. For $\gamma = -1$ we obtain a reaction-diffusion equation on $\Sigma$ including homogenized diffusion coefficients given by cell problems on the reference element $Z_s$. For the case $\gamma \in (-1,1)$ the diffusion term vanishes, and the evolution of the macroscopic concentration of the solute on $\Sigma$ is described by an ordinary differential equation. Finally, in the critical case $\gamma = 1$ the micro- and macro-variable do not decouple, leading in every macroscopic point $\x \in \Sigma$ to a partial differential equation with respect to the microscopic variable $y \in Z_s$.
\\

 We are looking for $c_0^f :(0,T)\times \Omega \rightarrow \R$  and $c_0^s:(0,T)\times \Sigma \times Z_s \rightarrow \R$, such that $c_0^f$ solves
\begin{subequations}\label{Macro_Model_Transport}
\begin{align}
\begin{aligned}\label{eq:Macro_c0f_strong}
\partial_t c_0^f - \nabla \cdot \left[ D^f \nabla c_0^f - v_0 c_0^f\right] &= H_{\Gamma}^{\Sigma}(c_0^f,c_0^s) &\mbox{ in }& (0,T)\times \Omega,
\\
-\left[D^f \nabla c_0^f - v_0 c_0^f\right]\cdot \nu &= 0 &\mbox{ on }& (0,T)\times \partial_N \Omega,
\\
c_0^f (0) &= c_{\mathrm{in}}^f &\mbox{ in }& L^2(\Omega),
\\
c_0^f \mbox{ is } \Sigma\mbox{-periodic,}
\end{aligned}
\end{align}
with $H_{\Gamma}^{\Sigma}(c_0^f,c_0^s) \in L^2((0,T),H^1_{\#}(\Omega)^{\ast})$ defined almost everywhere in $(0,T)$ by
\begin{align*}
    \langle H_{\Gamma}^{\Sigma} (c_0^f ,c_0^s),\xi^f \rangle_{H^1_{\#}(\Omega)}:= - \int_{\Sigma} \int_{\Gamma} h(c_0^f,c_0^s) d\sigma_y  \xi^f d\x. 
\end{align*}
We call $c_0^f$ a weak solution of the problem $\eqref{eq:Macro_c0f_strong}$ if 
$c_0^f \in L^2((0,T),H^1_{\#}(\Omega)) \cap H^1((0,T),H^1_{\#}(\Omega)^{\ast})$ with $c_0^f(0) = c_{\mathrm{in}}^f$ and for all $\xi^f \in H^1_{\#}(\Omega)$ it holds almost everywhere in $(0,T)$ that (for the precise regularity of $c_0^s$ see below for the different choices of $\gamma$)
\begin{align*}
 \langle \partial_t c_0^f , \xi^f \rangle_{H^1_{\#}(\Omega)} +  \int_{\Omega} [D^f \nabla c_0^f - v_0 c_0^f] \cdot \nabla \xi^f dx = -\int_{\Sigma} \int_{\Gamma} h(c_0^f,c_0^s) d\sigma_y \xi^f d\x,
\end{align*}
where $v_0$ is the limit function of $\veps$.
The macroscopic equation \eqref{eq:Macro_c0f_strong}
can equivalently be written as a homogeneous equation on $\Omega$ and a jump condition for the normal flux across $\Sigma$. In particular we see, that the normal flux across $\Sigma$ is not continuous and the jump in the normal flux is caused by the processes in the thin layer.
The function $c_0^s$ solves the following problems for the different choices of $\gamma$:\\

\noindent \underline{$\gamma = -1$}: For $\gamma = - 1$ it holds that $c_0^s: (0,T)\times \Sigma \rightarrow \R$ solves 
    \begin{align}
    \begin{aligned}
    \label{eq:Macro_c0s_gamma-1_strong}
     |Z_s|\partial_t c_0^s - \nabla_{\x} \cdot ( D^s_{\ast} \nabla_{\x} c_0^s ) &= |\Gamma|h(c_0^f,c_0^s) &\mbox{ in }& (0,T)\times \Sigma,
\\
c_0^s(0) &= \bar{c}_{\mathrm{in}}^s &\mbox{ in }& \Sigma,
\\
c_0^s \,\, \Sigma\mbox{-periodic},
    \end{aligned}
    \end{align}
and we call $c_0^s$ a weak solution of problem $\eqref{eq:Macro_c0s_gamma-1_strong}$ if 
$c_0^s \in L^2((0,T),H_{\#}^1(\Sigma)) \cap H^1((0,T),H_{\#}^1(\Sigma)^{\ast})$ with $c_0^s(0) = \bar{c}_{\mathrm{in}}^s$ and for all $\xi_0^s\in H^1_{\#}(\Sigma)$ it holds almost everywhere in $(0,T)$ that
\begin{align*}
     |Z_s|\langle \partial_t c_0^s, \xi_0^s\rangle_{H^1_{\#}(\Sigma)} + \int_{\Sigma} D_{\ast}^s \nabla_{\x} c_0^s \cdot \nabla_{\x} \xi_0^s d\x = |\Gamma|\int_{\Sigma} h(c_0^f,c_0^s) \xi_0^s d\x.
\end{align*}

\noindent\underline{$\gamma \in (-1,1)$}: For $\gamma \in (-1,1)$ it holds that $c_0^s: (0,T)\times \Sigma \rightarrow \R$ solves 
\begin{align}
\begin{aligned}
\label{eq:Macro_c0s_gamma-11_strong}
 |Z_s|\partial_t c_0^s &= |\Gamma|h(c_0^f,c_0^s) &\mbox{ in }& (0,T)\times \Sigma,
\\
c_0^s(0) &= \bar{c}_{\mathrm{in}}^s &\mbox{ in }& \Sigma,
\end{aligned}
\end{align}
and we call $c_0^s$ a weak solution of problem $\eqref{eq:Macro_c0s_gamma-11_strong}$ if $c_0^s \in H^1((0,T),L^2(\Sigma))$ with $c_0^s (0) = \bar{c}_{\mathrm{in}}^s$ and for all $\xi_0^s \in L^2(\Sigma)$ it holds almost everywhere in $(0,T)$ that
\begin{align*}
    |Z_s|\int_{\Sigma} \partial_t c_0^s \xi_0^s d\x = |\Gamma| \int_{\Sigma} h(c_0^f,c_0^s) \xi_0^s d\x. 
\end{align*}

\noindent \underline{$\gamma = 1$}:  For $\gamma =1 $ the function $c_0^s:(0,T)\times \Sigma \times Z_s \rightarrow \R$ solves 
\begin{align}
\begin{aligned}
\label{eq:Macro_c0s_gamma1_strong}
 \partial_t c_0^s - \nabla_y \cdot ( D^s \nabla_y c_0^s) &=0 &\mbox{ in }& (0,T)\times \Sigma \times Z_s,
\\
-D^s \nabla_y c_0^s \cdot \nu &= h(c_0^f,c_0^s) &\mbox{ on }& (0,T)\times \Sigma \times \Gamma,
\\
c_0^s(0) &= c_{\mathrm{\mathrm{in}}}^s &\mbox{ in }& \Sigma \times Z_s,
\\
c_0^s \,\, Y\mbox{-periodic},
\end{aligned}
\end{align}
and we call $c_0^s$ a weak solution of the problem $\eqref{eq:Macro_c0s_gamma1_strong}$ if $c_0^s \in L^2((0,T)\times \Sigma ,H_{\#}^1(Z_s)) \cap H^1((0,T),L^2(\Sigma,H_{\#}^1(Z_s))^{\ast})$ with $c_0^s(0) = c_{\mathrm{in}}^s$ and for all $\xi_1^s \in L^2(\Sigma,H_{\#}^1(Z_s))$ it holds almost everywhere in $(0,T)$ that 
\begin{align*}
    \langle \partial_t c_0^s,\xi_1^s\rangle_{L^2(\Sigma,H_{\#}^1(Z_s))} + \int_{\Sigma} \int_{Z_s} D^s \nabla_y c_0^s \cdot \nabla_y \xi_1^s dyd\x = \int_{\Sigma} \int_{\Gamma} h(c_0^f,c_0^s) \xi_1^s d\sigma d\x.
\end{align*}
\end{subequations}
We will show that the sequence $(\ceps^f,\ceps^s)$ converges in a suitable sense to the function $(c_0^f,c_0^s)$ solving the macroscopic equation $\eqref{eq:Macro_c0f_strong}$ and (\ref{Macro_Model_Transport}$\ast$) with $\ast \in\{\mbox{b,c,d}\}$ for the different choices of $\gamma$. Our compactness results are based on $\vareps$-uniform a priori estimates for the microscopic solutions. For the concentration of the solute in the solid phase $\ceps^s$ we have
\begin{align}\label{ineq:aux_main_results}
\foe \|\partial_t \ceps^s\|_{L^2((0,T),H_{\vareps,\gamma}(\oems)^{\ast})}  + \frac{1}{\sqrt{\vareps}} \|\ceps^s\|_{L^{\infty}((0,T),L^2(\oems))} + \vareps^{\frac{\gamma}{2}} \|\nabla \ceps^s\|_{L^2((0,T)\times \oems)} \le C,
\end{align}
where we refer to the beginning of Section \ref{sec:existence_apriori_estimates} for the definition of the space $H_{\vareps,\gamma}(\oems)$. The concentration of the solute in the fluid phase $\ceps^f$ fulfills
\begin{align*}
\|\partial_t \ceps^f \|_{L^2((0,T),H^1(\oef)^{\ast})} + \|\ceps^f\|_{L^{\infty}((0,T),L^2(\oef))} + \|\nabla \ceps^f\|_{L^2((0,T)\times \oef)} \le C.
\end{align*}
The proof of these estimates is given in Section \ref{sec:existence_apriori_estimates}, where we formulated all \textit{a priori} estimates for the microscopic solution in Proposition \ref{prop:existence_apriori}. The most critical parts in the proof are the control of the nonlinear boundary term on $\geps$ and the convective term in the thin layer. The crucial ingredient is the embedding result in Lemma \ref{lem:embedding_thin_layer_H1_Lp} and the trace inequality in Lemma \ref{lem:Trace_inequality_Bulk} giving control of the traces of $\ceps^f$ on $\geps$ via the bulk domains. The latter is based on an estimate for the $L^2(\oemf)$-norm via the $H^1(\oeps^{\pm})$-norm (respectively $H^{\beta}$-norm for $\beta \in (\frac12,1)$), see Lemma \ref{lem:estimate_L2_membrane_Bulk_Gradient} and \ref{lem:estimate_L2_norm_membrane_Hbeta}. In particular, we obtain 
\begin{align*}
    \frac{1}{\sqrt{\vareps}} \|\ceps^f\|_{L^2((0,T)\times \oemf)} \le C\|\ceps^f\|_{L^2((0,T),H^1(\oef))} \le C.
\end{align*}
These uniform estimates imply weak (two-scale) convergence results for $\ceps^f$ and $\ceps^s$ what is enough to pass to the limit in the linear terms. However, a crucial point remains to pass to the limit in the nonlinear boundary term on $\geps$. For $\ceps^f$ we first use a Kolmogorov-Simon-compactness argument to obtain strong $L^2$-convergence in the bulk domains. Using again Lemma \ref{lem:Trace_inequality_Bulk}, which allows to control the traces on $\geps $ via the bulk domains, we are able to obtain the strong two-scale compactness result (see Section \ref{sec:two_scale_convergence_unfolding} in the appendix for the definition of the two-scale convergence)
\begin{align*}
    \ceps^f|_{\geps} \rasts c_0^f|_{\Sigma} \qquad\mbox{on }\geps.
\end{align*}
To establish the strong two-scale convergence on $\geps$ for the concentration of the solute in the solid phase $\ceps^s$, we extend these functions for $\gamma \in [-1,1)$ to the whole layer $\oem$ using the extension operators preserving the \textit{a priori} bounds, see Lemma \ref{lem:Extension_operators}, and then average this function over the $x_3$-component. For this sequence we use Kolmogorov-Simon-type arguments, where for $\gamma \in (-1,1) $ an additional estimate for the shifts is necessary. Finally, in the critical case $\gamma = 1$ we use a general Kolmogorov-Simon-compactness result, combining two-scale convergence and dimension reduction, for sequences fulfilling $\eqref{ineq:aux_main_results}$ and an additional bound for the shifts. 
We summarize the results in the following main result:
\begin{theorem}\label{Main_theorem_transport}
Let  $(\ceps^f,\ceps^s)$ be the concentration of the solutes from the microscopic solution of problem $\eqref{MicroscopicModel}$.  There exists an extension $\tceps^f $ of $\ceps^f$ to the whole domain $\Omega$ and a limit function $c_0^f \in L^2((0,T),H_{\#}^1(\Omega)) \cap H^1((0,T),H^1_{\#}(\Omega)^{\ast})$, such that  for all $\beta \in \left(\frac12,1\right)$
\begin{align}
    \tceps^f \rightharpoonup c_0^f \quad\mbox{in } L^2((0,T),H^1(\Omega)),\qquad \tceps^f \rightarrow c_0^f \quad\mbox{in } L^2((0,T),H^{\beta}(\Omega)).
\end{align}
Further, there exists $c_0^s$ with 
\begin{align*}
    c_0^s \in \begin{cases}
L^2((0,T),H_{\#}^1(\Sigma)) \cap H^1((0,T),H_{\#}^1(\Sigma)^{\ast}) &\mbox{ for } \gamma = -1,
\\
H^1((0,T),L^2(\Sigma)) &\mbox{ for } \gamma \in (-1,1),
\\
L^2((0,T)\times \Sigma ,H_{\#}^1(Z_s)) \cap H^1((0,T),L^2(\Sigma,H_{\#}^1(Z_s))^{\ast}) &\mbox{ for } \gamma = 1,
    \end{cases}
\end{align*}
such that $\chi_{\oems} \ceps^s \rats c_0^s$. Further, it holds that
\begin{align*}
    \ceps^f|_{\geps} \rasts c_0^f|_{\Sigma} \quad\mbox{on }\geps, \qquad \ceps^s|_{\geps} \rightarrow c_0^s \quad\mbox{ on }\geps.
\end{align*}
The tuple $(c_0^f,c_0^s)$ is the unique weak solution of the macroscopic equation $\eqref{eq:Macro_c0f_strong}$ and (\ref{Macro_Model_Transport}$\ast$) with $\ast \in\{\mbox{b,c,d}\}$ for the different choices of $\gamma$. 
\end{theorem}

\begin{remark}
The precise convergence results for $\ceps^f$ and $\ceps^s$ are formulated in Section \ref{sec:compactness_concentrations_fluid} and \ref{sec:conv_ceps_s}.
\end{remark}

\section{Existence and \textit{a priori} estimates}
\label{sec:existence_apriori_estimates}

In this section we formulate the existence result for the microscopic problem $\eqref{MicroscopicModel}$ and show $\vareps$-uniform \textit{a priori} estimates for the microscopic solution. These estimates form the starting point for the homogenization and provide several (two-scale) compactness results. We start by introducing a function space on the perforated solid part $\oems$ adapted to the thin structure and the scaling in the diffusive term in the solid phase depending on $\gamma$. More precisely, we define  the space $H_{\vareps,\gamma}(\oems)$ as the space of functions from $H^1_{\#}(\oems)$ together with the norm
\begin{align*}
\|\xieps^s\|_{H_{\vareps,\gamma}(\oems)}^2:= \foe \|\xieps^s\|_{L^2(\oems)}^2 + \vareps^{\gamma} \|\nabla \xieps^s\|_{L^2(\oems)}^2.
\end{align*}
Further, we consider the Gelfand-triple
\begin{align*}
    H_{\vareps,\gamma}(\oems) \hookrightarrow L^2(\oems) \hookrightarrow H_{\vareps,\gamma}(\oems)^{\ast}
\end{align*}
with the natural embeddings, leading to the equality (for $F_{\vareps} \in H_{\vareps,\gamma}(\oems)^{\ast} \simeq H^1_{\#}(\oems)^{\ast}$)
\begin{align*}
    \langle F_{\vareps} , \xieps^s \rangle_{H_{\vareps,\gamma}(\oems)} = \langle F_{\vareps} , \xieps^s \rangle_{H^1_{\#}(\oems)}
\end{align*}
for all $\xieps^s \in H_{\vareps,\gamma}(\oems) \simeq H^1_{\#}(\oems)$. 
\begin{proposition}\label{prop:existence_apriori}
There exists a unique weak solution $(\veps,\peps,\ceps^f,\ceps^s)$ of the microscopic problem $\eqref{MicroscopicModel}$, such that the following \textit{a priori} estimates hold
\begin{align*}
 \|\partial_t \veps^{\pm}\|_{L^2((0,T)\times \oeps^{\pm})} +    \|\veps^{\pm}\|_{L^{\infty}((0,T),H^1(\oeps^{\pm}))} + \|\peps^{\pm}\|_{L^2((0,T)\times \oeps^{\pm})} &\le C,
    \\
  \|\partial_t \veps^M\|_{L^2((0,T)\times \oemf)} +     \frac{1}{\sqrt{\vareps}} \|\veps^M\|_{L^{\infty}((0,T),L^2(\oemf))} + \sqrt{\vareps} \|\nabla \veps^M\|_{L^{\infty}((0,T), L^2( \oemf))} &\le C,
    \\
\frac{1}{\sqrt{\vareps}} \|\peps^M\|_{L^2((0,T)\times \oemf)} &\le C\sqrt{\vareps},
\\
\|\partial_t \ceps^f \|_{L^2((0,T),H^1_{\#}(\oef)^{\ast})} + \|\ceps^f\|_{L^{\infty}((0,T),L^2(\oef))} + \|\nabla \ceps^f\|_{L^2((0,T)\times \oef)} &\le C,
\\
\foe \|\partial_t \ceps^s\|_{L^2((0,T),H_{\vareps,\gamma}(\oems)^{\ast})}  + \frac{1}{\sqrt{\vareps}} \|\ceps^s\|_{L^{\infty}((0,T),L^2(\oems))} + \vareps^{\frac{\gamma}{2}} \|\nabla \ceps^s\|_{L^2((0,T)\times \oems)} &\le C.
\end{align*}
Further, in the fluid part of the thin layer we obtain the estimates 
\begin{align*}
   \frac{1}{\sqrt{\vareps}} \|\ceps^f\|_{L^{\infty}((0,T),L^2(\oemf))} + \|\ceps^f\|_{L^2((0,T)\times \geps)} \le C. 
\end{align*}
\end{proposition}
\begin{proof}
The existence and uniqueness of a solution are standard and can be obtained via a Galerkin method based on similar estimates as below. So we focus on proving the $\vareps$-uniform \textit{a priori} estimates. The estimates for the fluid velocity are quite standard and can be obtained by using similar methods as for Stokes flow in perforated thin layers. However, for the sake of completeness we give some details. For the pressure we use similar arguments as in \cite{gahn2022derivation}, by constructing a suitable solution of the divergence equation in $\oef$. The most critical and new part are the \textit{a priori} estimates for the concentrations in the thin perforated layer, in particular the treatment of the nonlinear coupling term on $\geps$, and the control of the convective term to find a uniform bound for $\partial_t \ceps^f$, which are necessary to obtain strong convergence results. Compared to the existing literature, see for example \cite{GahnEffectiveTransmissionContinuous} and  \cite{GahnNeussRaduKnabner2018a}, we here have a different scaling in the diffusive term in the transport equation for $\ceps^f$ in $\oemf$ and the term including the time-derivative, which causes additional difficulties. To get control of the nonlinear boundary term, we use the trace inequality in Lemma \ref{lem:Trace_inequality_Bulk} in the appendix. To estimate the convective term we use Lemma \ref{lem:embedding_thin_layer_H1_Lp} in the appendix, which gives an explicit dependence on $\vareps$ for the operator norm of the embedding $H^1(\oemf)\hookrightarrow L^p(\oemf)$ for $2\le p < 6$.

We start with estimating the fluid velocity $\veps$ by choosing $\phieps = \veps$ as a test-function in $\eqref{eq:Var_Micro_veps}$ to obtain
\begin{align*}
    \sum_{\pm} &\left\{ \frac12 \frac{d}{dt} \|\veps^{\pm}\|_{L^2(\oeps^{\pm})}^2 + \|D(\veps^{\pm})\|^2_{L^2(\oeps^{\pm})} \right\}  + \frac{1}{2} \frac{d}{dt} \| \veps^M\|^2_{L^2(\oemf)} + \vareps \|D(\veps^M)\|_{L^2(\oemf)}^2
    \\
    &\le \sum_{\pm} \|f_{\vareps}^{\pm}\|_{L^2(\oeps^{\pm})}\|\veps^{\pm}\|_{L^2(\oeps^{\pm})}  \le C\left(1 + \sum_{\pm} \|\veps^{\pm}\|_{L^2(\oeps^{\pm})}^2\right).
\end{align*}
Integration with respect to time, Gronwall inequality, and the Korn inequality (see Lemma \ref{lem:Korn_inequality} in the appendix for the perforated thin layer) imply for almost every $t \in (0,T)$
\begin{align*}
    \sum_{\pm} \left\{ \|\veps^{\pm}(t)\|_{L^2(\oeps^{\pm})}^2 + \|\nabla \veps^{\pm}\|_{L^2((0,t) \times \oeps^{\pm})} \right\} +  \|\veps^M(t) \|_{L^2(\oemf)}^2 + \vareps \|\nabla \veps^M\|_{L^2((0,t)\times \oemf)}^2 \le C.
\end{align*}
Using the Poincar\'e inequality in the perforated domain, see Lemma \ref{lem:Korn_inequality}, we obtain
\begin{align*}
    \frac{1}{\sqrt{\vareps}}\|\veps^M\|_{L^2((0,T)\times \oemf)} \le C.
\end{align*}
To obtain an estimate for the time derivative $\partial_t \veps$ and $L^{\infty}$-estimates with respect to time for $\veps$ and its gradient we choose $\phieps = \partial_t \veps$ and use similar arguments as above. We emphasize that this is a formal argument because $\partial_t \veps$ is not regular enough. This argument can be made rigorous for example via a Galerkin approximation. We also obtain the $L^{\infty}$-regularity with respect to time for the gradients of $\veps$.
To estimate the pressure we use similar arguments as in the proof of Lemma 4.2 in \cite{gahn2022derivation}. For the sake of completeness we give some details. First of all, there exists $\phieps \in H^1_{\#}(\Omega)^3$ with $\phieps=0 $ in $\oeps^-\cup \oeps^M$ and  $\nabla \cdot \phieps = \peps^+$, such that 
\begin{align*}
    \|\phieps\|_{H^1(\oeps^+)} \le C \|\peps^+\|_{L^2(\oeps^+)}.
\end{align*}
We emphasize that the constant on the right-hand side can be chosen independently of $\vareps$. This can be shown by a transformation of $\oeps^+$ to the fixed domain $\Omega^+$ (see for example \cite[Section 5.1]{NeussJaeger_EffectiveTransmission} for a transformation).
Choosing the function $\phieps$ as a test function in $\eqref{eq:Var_Micro_veps}$ and using the estimates for $\veps^+$ obtained above we get the desired estimate for $\peps^+$. In the same way we can treat the pressure $\peps^-$. Using the Bogovskii-operator from $L^2(Z_f) \rightarrow H^1(Z_f,\partial Z_f \setminus S^{\pm}_f)^3$, we obtain a $\phieps^M \in H^1_{\#}(\oemf,\geps)^3$ such that $\nabla \cdot \phieps^M = \peps^M$ and 
\begin{align*}
    \|\nabla \phieps^M\|_{L^2(\oemf)} \le C \|\peps^M\|_{L^2(\oemf)}.
\end{align*}
Using a mirror and a cut-off (in an $\vareps$-neighborhood of $\oeps^M$) argument this function can be extended to a function $\phieps \in H^1_{\#}(\Omega)^3$ with 
\begin{align*}
    \sum_{\pm} \left\{ \vareps^{-1}\|\phieps^{\pm}\|_{L^2(\oeps^{\pm})} + \|\nabla \phieps^{\pm}\|_{L^2(\oeps^{\pm})} \right\} + \|\nabla \phieps^M\|_{L^2(\oemf)} \le C \|\peps^M\|_{L^2(\oemf)}.
\end{align*}
Testing $\eqref{eq:Var_Micro_veps}$ with the function $\phieps$ we get
\begin{align*}
\foe \|\peps^M\|^2_{L^2(\oemf)} \le& 
\sum_{\pm} \left\{ \|\partial_t \veps^{\pm}\|_{L^2(\oeps^{\pm})}+ \|D(\veps^{\pm})\|_{L^2(\oeps^{\pm})}  + \|\peps^{\pm}\|_{L^2(\oeps^{\pm})}  +  \|f_{\vareps}^{\pm}\|_{L^2(\oeps^{\pm})} \right\}\|\phieps\|_{H^1(\oeps^{\pm})} 
\\
&+  \|\partial_t \veps^M \|_{L^2(\oemf)} \|\phieps\|_{L^2(\oemf)} + \vareps \|D(\veps^M)\|_{L^2(\oemf)} \|D(\phieps)\|_{L^2(\oemf)} 
\\
\le& \bigg[ \sum_{\pm} \left\{ \|\partial_t \veps^{\pm}\|_{L^2(\oeps^{\pm})}+ \|D(\veps^{\pm})\|_{L^2(\oeps^{\pm})}  + \|\peps^{\pm}\|_{L^2(\oeps^{\pm})}  +  \|f_{\vareps}^{\pm}\|_{L^2(\oeps^{\pm})} \right\}
\\
&+ \vareps \|\partial_t \veps^M\|_{L^2(\oemf)} + \vareps\|D(\veps^M)\|_{L^2(\oemf)} \bigg] \|\peps^M\|_{L^2(\oemf)}.
\end{align*}
After integration with respect to time and using the estimates obtained above we obtain the result for $\peps^M$.

For the bounds of $\ceps^f$ and $\ceps^s$ we test equation $\eqref{eq_Var_Micro_cepsf}$ and $\eqref{eq_Var_Micro_cepss}$ with  $\ceps^f$ respectively $\ceps^s$ and add up both equations to obtain 
\begin{align}
\begin{aligned}\label{ineq:aux1_Lemma_apriori}
\frac12 \frac{d}{dt}& \|\ceps^f\|_{L^2(\oef)}^2 + \frac{1}{2\vareps} \frac{d}{dt} \|\ceps^s\|_{L^2(\oems)}^2 + D^f\|\nabla \ceps^f\|_{L^2(\oef)}^2 + D^s\vareps^{\gamma}\|\nabla \ceps^s\|_{L^2(\oems)}^2 
\\
&\le \int_{\oef} \ceps^f \veps \cdot \nabla \ceps^f dx + \int_{\geps} h(\ceps^f,\ceps^s) (\ceps^s - \ceps^f) d\sigma.
\end{aligned}
\end{align}
For the convective term  we use the trace inequality (here we use $\Omega \subset \R^3$)
\begin{align}
    \|\weps\|_{L^4(\partial_N \Omega)} \le C \sum_{\pm}\|\weps \|_{H^1(\oeps^{\pm})}
\end{align}
for all $\weps \in H^1(\oef)$ and a constant $C>0$ independent of $\vareps$ to obtain by integration by parts and the H\"older inequality
\begin{align*}
    \left| \int_{\oef} \veps \ceps^f \cdot \nabla \ceps^f dx \right| &= \frac{1}{2}\left|\int_{\oef} \veps \cdot \nabla |\ceps^f|^2 dx \right| = \frac{1}{2}\left| \int_{\partial_N \Omega} \veps \cdot \nu |\ceps^f|^2 d\sigma  \right|
    \\
    &\le \frac{1}{2} \|\veps\|_{L^4(\partial_N \Omega)} \| \ceps^f\|_{L^{\frac83}(\partial_N \Omega)}^2
    \\
    &\le C \sum_{\pm} \|\veps\|_{H^1(\oeps^{\pm})} \| \ceps^f\|_{L^{\frac83}(\partial_N \Omega)}^2.
\end{align*}
Now, for arbitrary $p \in [1,4)$ we use the compactness of the embedding $H^1(\oef) \hookrightarrow L^p(\partial_N \Omega)$ (again, this holds since $\Omega \subset \R^3$) to obtain for arbitrary $\theta >0 $ the existence of $C_{\theta}>0$, such that 
\begin{align*}
 \|w\|_{L^p(\partial_N \Omega)} \le C_{\theta}\|w\|_{L^2(\oef)} + \theta \|\nabla w \|_{L^2(\oef)}   
\end{align*}
for all $w \in H^1(\oef)$. Hence, we get
\begin{align*}
    \left| \int_{\oef} \veps \ceps^f \cdot \nabla \ceps^f dx \right| &\le C \sum_{\pm} \|\veps^{\pm}\|_{H^1(\oeps^{\pm})}\left(C_{\theta} \|\ceps^f\|_{L^2(\oef)}^2 + \theta \|\nabla \ceps^f\|_{L^2(\oef)}^2\right)
    \\
    &\le C_{\theta}\|\ceps^f\|_{L^2(\oef)}^2 + \theta \|\nabla \ceps^f\|_{L^2(\oef)}^2, 
    \end{align*}
where in the last inequality we used the \textit{a priori} estimate already obtained for $\veps$ (and possibly changed the values of $C_{\theta}$ and $\theta$).

For the nonlinear term on $\geps$ we use the growth condition of $h$ from assumption \ref{ass:h} and the trace inequalities in Lemma \ref{lem:Stand_scaled_trace_inequality} and \ref{lem:Trace_inequality_Bulk} (remember $|\geps|\le C$) to obtain for all $\theta>0$
\begin{align*}
\int_{\geps} h(\ceps^f,\ceps^s) (\ceps^s - \ceps^f) d\sigma &\le C \left(1 + \|\ceps^f\|_{L^2(\geps)}^2 + \|\ceps^s\|^2_{L^2(\geps)} \right)
\\
&\le C_{\theta}\left(1 + \foe \|\ceps^s\|_{L^2(\oems)}^2 +  \|\ceps^f\|_{L^2(\oeps^{\pm})}^2    \right) 
\\
&+ C_0 \vareps \|\nabla \ceps^f\|_{L^2(\oemf)}^2 + \theta \left( \vareps \|\nabla \ceps^s\|_{L^2(\oems)}^2 + \|\nabla \ceps^f\|_{L^2(\oeps^{\pm})}^2 \right).
\end{align*}
Hence, we obtain from $\eqref{ineq:aux1_Lemma_apriori}$ by choosing $\theta>0$ small enough for all $\vareps< \vareps_0$ for $\vareps_0>0 $ small enough
\begin{align*}
     \frac{d}{dt} \|\ceps^f\|_{L^2(\oef)}^2 + \frac{1}{\vareps} \frac{d}{dt} \|\ceps^s\|_{L^2(\oems)}^2 + \|&\nabla \ceps^f\|_{L^2(\oef)}^2 + \vareps^{\gamma}\|\nabla \ceps^s\|^2_{L^2(\oems)} 
\\
&\le C \left(1 + \foe \|\ceps^s\|_{L^2(\oems)}^2 +  \|\ceps^f\|_{L^2(\oeps^{\pm})}^2    \right).
\end{align*}
Now, the Gronwall inequality implies the $L^2$-bounds for $\ceps^f$ and $\ceps^s$ and their gradients.
\\

It remains to show the bounds for the time derivatives $\partial_t \ceps^f$ and $\partial_t \ceps^s$. First, for $\xieps^s \in H^1_{\#}(\oems)$ with $\|\xieps^s\|_{H_{\vareps,\gamma}(\oems)}\le 1$ we obtain from $\eqref{eq_Var_Micro_cepss}$ with the trace inequalities from Lemma \ref{lem:Stand_scaled_trace_inequality} and \ref{lem:Trace_inequality_Bulk}
\begin{align*}
\foe \langle \partial_t \ceps^s,\xieps^s\rangle_{H_{\vareps,\gamma}(\oems)} \le& C \vareps^{\gamma} \|\nabla \ceps^s\|_{L^2(\oems)} \|\nabla \xieps^s\|_{L^2(\oems)} 
\\
&+ C \left( 1 + \|\ceps^f\|_{L^2(\geps)} + \|\ceps^s\|_{L^2(\geps)} \right) \|\xieps^s\|_{L^2(\geps)}
\\
\le& C \bigg\{ 1 + \vareps^{\frac{\gamma}{2}} \|\nabla \ceps^s\|_{L^2(\oems)} +   \sqrt{\vareps} \|\nabla \ceps^f\|_{L^2(\oemf)} 
\\
&+ \|\ceps^f\|_{H^1(\oeps^{\pm})} + \frac{1}{\sqrt{\vareps} } \|\ceps^s\|_{L^2(\oems)} + \sqrt{\vareps} \|\nabla \ceps^s\|_{L^2(\oems)}  \bigg\},
\end{align*}
where at the end we also used the inequality
\begin{align*}
    \|\xieps^s\|_{L^2(\geps)} \le C \left( \frac{1}{\sqrt{\vareps}} \|\xieps\|_{L^2(\oems)} + \sqrt{\vareps} \|\nabla \xieps^s\|_{L^2(\oems)} \right) \le C \|\xieps^s\|_{H_{\vareps,\gamma}(\oems)},
\end{align*}
since $\gamma \in [-1,1]$. Now, taking the supremem, integrating with respect to time and using the estimates already obtained above we get 
\begin{align*}
\|\partial_t \ceps^s\|_{L^2((0,T),H_{\vareps,\gamma}(\oems)^{\ast})} \le C\vareps .
\end{align*}
Now, we estimate the time derivative of $\ceps^f$. For this we choose test-functions $\xieps^f \in H^1_{\#}(\oef)$ with $\|\xieps^f\|_{H^1(\oef)} \le 1$ in $\eqref{eq_Var_Micro_cepsf}$. The arguments are similar to the arguments for $\partial_t \ceps^s$, excepting the advective term, which is treated in the following way:
\begin{align*}
\int_{\oef} \veps \ceps^f \cdot \nabla \xieps^f dx = \int_{\oemf} \veps^M \ceps^f \cdot \nabla \xieps^f dx + \sum_{\pm } \int_{\oeps^{\pm}}\veps^{\pm} \ceps^f \cdot \nabla \xieps^f dx.
\end{align*}
For the bulk terms we use the embedding $H^1(\oeps^{\pm}) \hookrightarrow L^4(\oeps^{\pm})$ (with embedding constant independent of $\vareps$) to obtain
\begin{align*}
\left| \int_{\oeps^{\pm}} \veps^{\pm} \ceps^f \cdot \nabla \xieps^f dx \right| \le  \|\veps^{\pm}\|_{L^4(\oeps^{\pm})} \|\ceps^f\|_{L^4(\oeps^{\pm})} \|\nabla \xieps^f\|_{L^2(\oeps^{\pm})} \le C \|\veps^{\pm}\|_{H^1(\oeps^{\pm})} \|\ceps^f\|_{H^1(\oeps^{\pm})}.
\end{align*}
For the advective term in the thin layer we use Corollary \ref{cor:embedding_H1_L6_oef} and Lemma \ref{lem:embedding_thin_layer_H1_Lp} with $p=3$ and $G=\Gamma$ (remember $\veps = 0$ on $\geps$)
\begin{align*}
\left|\int_{\oemf} \veps^M \ceps^f \cdot \nabla \xieps^f dx \right| &\le \|\veps^M\|_{L^3(\oemf)} \|\ceps^f\|_{L^6(\oemf)} \|\nabla \xieps^f\|_{L^2(\oemf)}
\\
&\le C \sqrt{\vareps} \|\nabla \veps^M\|_{L^2(\oemf)} \|\ceps^f\|_{L^6(\oef)} 
\\
&\le C \|\ceps^f\|_{H^1(\oef)}.
\end{align*}
This implies the desired result for $\vareps < \vareps_0$. However, since there are only finitely many elements $\vareps \geq \vareps_0$ the result is valid for all $\vareps$ by possibly changing the generic constant $C$.

Finally, the estimates for $\ceps^f$ in $\oemf$ and on $\geps$ follow directly from Lemma \ref{lem:estimate_L2_membrane_Bulk_Gradient} and \ref{lem:Trace_inequality_Bulk}.

\end{proof}
For $\gamma = -1$ we can give a simple relation between the norm on $H_{\vareps,-1}(\oems)^{\ast}$ and the more common space $H^1(\oems)^{\ast}$. In fact, we have:
\begin{align*}
   \|\partial_t \ceps^s\|_{L^2((0,T),H_{\vareps,-1}(\oems)^{\ast})}= \sqrt{\vareps} \|\partial_t \ceps^s\|_{L^2((0,T),H^1(\oems)^{\ast})}.
\end{align*}

\begin{remark}
We emphasize that the results in Proposition \ref{prop:existence_apriori} remain valid in the case $|S_s^{\pm}| \neq 0$ when the solid phase touches the fluid bulk regions. In fact, in particular, we used estimates and results from the appendix, which are also valid in this case.
\end{remark}

\section{Compactness results for the micro-solution}
\label{sec:comptness_results}

In this section we derive the compactness results for sequence $(\veps,\peps,\ceps^f,\ceps^s)$ of microscopic solutions of $\eqref{MicroscopicModel}$. These results are based on the \textit{a priori} estimates from Section \ref{sec:existence_apriori_estimates} and do not use the explicit structure of the microscopic model, and therefore can be considered as general compactness results applicable also to other problems fulfilling the same $\vareps$-uniform bounds. The crucial point is to obtain strong (two-scale) compactness in the thin layer and on the surface $\geps$. The latter is necessary to pass to the limit in the nonlinear terms on $\geps$. For the case $\gamma \in (-1,1]$ we first prove an additional \textit{a priori} bound for differences of the shifts of $\ceps^s$.

\subsection{Convergence of fluid velocity and pressure}

For the fluid flow we can consider the bulk domains $\oeps^{\pm}$ and the thin fluid layer $\oemf$ separately, since we are only interested in weak convergence results (see also the convergence results for $\ceps^f$ below, were this is not possible anymore). Under the \textit{a priori} estimates in Proposition \ref{prop:existence_apriori}, the compactness results for $\veps^{\pm}$ and $\veps^M$ are almost classical, and we summarize them in the next Proposition. The crucial question is the coupling condition across $\Sigma$ for the limit functions $v_0^{\pm}$ of $\veps^{\pm}$. It turns out that from the continuity of the velocity across $S_{\vareps}^{\pm}$ we get that  $v_0^M = v_0^{\pm}$ on $\Sigma \times S^{\pm}$, where $v_0^M$ is the two-scale limit of $\veps^M$.  In particular, we obtain from this relation the continuity of the normal velocity for $v_0^+$ and $v_0^-$ across $\Sigma$.

\begin{proposition}\label{prop:compactness_fluid}
\begin{enumerate}[label = (\roman*)]
\item\label{prop:compactness_fluid_item1} There exist $v_0^{\pm} \in L^2((0,T), H^1_{\#}(\Omega^{\pm})^3 \cap H^1((0,T),L^2(\Omega^{\pm}))^3$ and $p_0^{\pm} \in L^2((0,T)\times \Omega^{\pm})$, such that up to a subsequence
\begin{align*}
\chi_{\oeps^{\pm}}\veps^{\pm} &\rightarrow v_0^{\pm} &\mbox{ in }& L^2((0,T)\times \Omega^{\pm}),
\\
\chi_{\oeps^{\pm}} \nabla \veps^{\pm} &\rightharpoonup \nabla v_0^{\pm} &\mbox{ in }& L^2((0,T)\times \Omega^{\pm}),
\\
\chi_{\oeps^{\pm}} \partial_t \veps^{\pm} &\rightharpoonup \partial_t v_0^{\pm} &\mbox{ in }& L^2((0,T)\times \Omega^{\pm}),
\\
\chi_{\oeps^{\pm}} \peps^{\pm} &\rightharpoonup p_0^{\pm} &\mbox{ in }& L^2((0,T)\times \Omega^{\pm}).
\end{align*}

\item\label{prop:compactness_fluid_item2} There exists $v_0^M \in L^2((0,T)\times \Sigma, H_{\#}^1(Z_f))^3$ with $v_0^M = 0$ on $\Gamma$ and 
\begin{align*}
    \nabla_y \cdot v_0^M = 0, \qquad \nabla_{\x} \cdot \int_{Z_f} v_0^M dy = 0,
\end{align*}
such that up to a subsequence it holds that 
\begin{align*}
\veps^M &\rats v_0^M,
\\
\vareps \nabla \veps^M &\rats \nabla_y v_0^M,
\\
\partial_t \veps^M &\rats 0,
\\
\peps^M &\rats 0.
\end{align*}

\item\label{prop:compactness_fluid_item3} We have
\begin{align*}
    v_0^{\pm} = v_0^M \qquad \mbox{on } (0,T)\times \Sigma \times S_f^{\pm}.
\end{align*}
In particular $v_0^M$ is constant with respect to $y$ on $S_f^{\pm}$.
\end{enumerate}
\end{proposition}
\begin{proof}
Throughout the proof we will use the \textit{a priori} estimates from Proposition \ref{prop:existence_apriori}. First of all,  we  easily obtain $v_0^{\pm} $  with the stated regularity and the weak convergences of $\chi_{\oeps^{\pm}} \veps^{\pm}$, $\chi_{\oeps^{\pm}} \nabla \veps^{\pm}$, and $\chi_{\oeps^{\pm}} \partial_t \veps^{\pm}$ in $L^2((0,T) \times \Omega^{\pm})$. Since for every $\delta >0$ we obtain with the Aubin-Lions Lemma the strong convergence of $\chi_{\oeps^{\pm}} \veps^{\pm}$ in $L^2((0,T)\times \Omega^{\pm}_{\delta})^3$, we get the strong convergence of this sequence in the whole domain $\Omega^{\pm}$. The convergence of $\peps^{\pm}$ is clear. In \ref{prop:compactness_fluid_item2}, the existence of $v_0^M$ with the desired properties follows by standard two-scale compactness results for divergence free vector fields, see for example \cite{Allaire_TwoScaleKonvergenz} for domains and \cite{MarusicMarusicPalokaTwoScaleConvergenceThinDomains,fabricius2020pressure}. It remains to check the interface condition $v_0^{\pm} = v_0^M$ in \ref{prop:compactness_fluid_item3}. First of all, we notice, that $\veps^{\pm}$ (as a function of $(t,\x)$) converges in the two-scale sense to $v_0^{\pm}|_{\Sigma}$ on $\Sigma$, and $\veps^M|_{S_{\vareps}^{\pm}}$ in the two-scale sense to $ v_0^M|_{S^{\pm}}$ on $S_{\vareps}^{\pm}$. Hence, we obtain for every $\phi \in C_0^{\infty}((0,T)\times \Sigma, C_{\#}^{\infty}(S^{\pm}))^3$ 
\begin{align*}
\int_0^T\int_{\Sigma} \int_{S^{\pm}} v_0^{\pm} \cdot \phi (\x,\y) d\y d\x dt &= \lim_{\vareps \to 0} \int_0^T \int_{S_{\vareps}^{\pm}} \veps^{\pm} \cdot \phi \left(\x, \frac{\x}{\vareps}\right) d\x dt 
\\
&= \lim_{\vareps \to 0} \int_0^T \int_{S_{\vareps}^{\pm}} \veps^M \cdot \phi \left(\x, \frac{\x}{\vareps}\right) d\x dt 
\\
&= \int_0^T \int_{\Sigma} \int_{S^{\pm}} v_0^M \cdot \phi (\x,\y) d\y d\x dt,
\end{align*}
which gives the desired result.

\end{proof}

\begin{remark}\label{rem:fluid_vel_zero_Sigma}
If $|S_s^{\pm}|>0$ then it holds that $v_0^{\pm} = 0$ on $\Sigma$. This follows by the same arguments as for $v_0^{\pm} = v_0^M$ in the proof above (just replace $S^{\pm}$ with $S_s^{\pm}$ and follow the same ideas) and using $\veps^{\pm} = 0$ on $S_{\vareps,s}^{\pm}$.
In particular we immediately obtain $v_0^M = 0$ on $S^{\pm}_f$. This is the crucial difference to the case $|S_s^{\pm}|=0$ and in this case the membrane acts as an impermeable barrier for the fluid flow.
\end{remark}

\begin{corollary}\label{cor:continuity_normal_component_v_0}
It holds that 
\begin{align*}
    [v_0^+]_3 = [v_0^-]_3 \qquad \mbox{ on } (0,T)\times \Sigma.
\end{align*}
\end{corollary}
\begin{proof}
Using $v_0^{\pm} = v_0^M $ on $(0,T)\times \Sigma \times S^{\pm}$ and $v_0^M = 0$ on $(0,T)\times \Sigma \times \Gamma$, we obtain almost everywhere in $(0,T)\times \Sigma$ (use that $v_0^M$ is divergence free with respect to $y$)
\begin{align*}
[v_0^+]_3 - [v_0^-]_3 = \int_{S^+} v_0^M \cdot \nu d\sigma_y + \int_{S^-} v_0^M \cdot \nu d\sigma_y 
= \int_{\partial Z_f} v_0^M \cdot \nu d\sigma_y = \int_{Z_f} \nabla_y \cdot v_0^M dy = 0.
\end{align*}
\end{proof}

\subsection{Convergence of the solute concentration in the fluid domain}
\label{sec:compactness_concentrations_fluid}
From the \textit{a priori} estimates in Proposition \ref{prop:existence_apriori} we immediately obtain weak convergences in the bulk domains. Also the strong convergence in $L^2$ will follow quite straightforward from these \textit{a priori} estimates. However, to pass to the limit in the nonlinear terms on the surface $\geps$ we need the two-scale convergence of $h(\ceps^f,\ceps^s)$ on $\geps$, for which we need the strong two-scale convergence of $\ceps^f$ on $\geps$. This will be the crucial part in the following arguments. To the best of our knowledge, such strong compactness results for the setting of our model and the \textit{a priori} estimates for $\ceps^f$ are new. Compared to \cite{GahnEffectiveTransmissionContinuous} we have another scaling for the time-derivative and less regularity for the time-derivative. The latter also makes it impossible to consider the bulk domains $\oeps^{\pm}$ and the thin layer $\oemf$ separately as in \cite{GahnNeussRaduKnabner2018a} (see also the weak convergence results for $\veps$). To overcome this problem  we extend $\ceps^f$ to the whole domain $\Omega$ with the extension operator from Lemma \ref{lem:Extension_operators}. The linearity of this operator and the \textit{a priori} bounds for $\ceps^f$ allow to control differences of shifts of the extension of $\ceps^f$ with respect to time. Then, we can apply the Simon-compactness theorem to obtain strong convergence of the extension of $\ceps^f$ in $\Omega$. Now, with the estimates from the appendix (see Lemma \ref{lem:estimate_L2_membrane_Bulk_Gradient}, \ref{lem:estimate_L2_norm_membrane_Hbeta} and \ref{lem:Trace_inequality_Bulk}) we can control $\ceps^f$ within  the thin fluid layer $\oemf$ by the norms in the bulk domains.

In the following we define the extension
\begin{align*}
    \tceps^f := E_{\vareps}^f \ceps^f \in L^{\infty}((0,T),L^2(\Omega)) , L^2((0,T), H^1_{\#}(\Omega))
\end{align*}
with the extension operator $E_{\vareps}^f$ from Lemma \ref{lem:Extension_operators}.
\begin{proposition}\label{prop:convergence_cepsf}
There exists 
\begin{align*}
c_0^f \in L^{\infty}((0,T),L^2(\Omega)) \cap L^2((0,T),H^1_{\#}(\Omega)) \cap H^1((0,T),H^1_{\#}(\Omega)^{\ast}),
\end{align*}
such that up to a subsequence for every $\beta \in (\frac12,1)$
\begin{align*}
    \tceps^f &\rightharpoonup c_0^f &\mbox{ in }& L^2((0,T),H^1(\Omega)),
    \\
    \tceps^f &\rightarrow c_0^f &\mbox{ in }& L^2((0,T), H^{\beta}(\Omega)),
    \\
    \partial_t (\chi_{\oef} \ceps^f) &\rightharpoonup \partial_t c_0^f &\mbox{ in }& L^2((0,T),H^1_{\#}(\Omega)^{\ast}).
\end{align*}
Further, it holds that $\tceps^f \rats c_0^f|_{\Sigma}$  on $\geps$.
\end{proposition}
\begin{proof}
From the properties of the extension operator in Lemma \ref{lem:Extension_operators} and the \textit{a priori} estimates for $\ceps^f$ in Proposition \ref{prop:existence_apriori} we obtain
\begin{align*}
\|\tceps^f\|_{L^{\infty}((0,T),L^2(\Omega))} + \|\tceps^f\|_{L^2((0,T),H^1(\Omega))} \le C.
\end{align*}
Hence, there exists $c_0^f \in L^{\infty}((0,T),L^2(\Omega)) \cap L^2((0,T),H^1_{\#}(\Omega)) $, such that up to a subsequence
\begin{align*}
    \tceps &\rightharpoonup c_0^f \qquad \mbox{in } L^2((0,T),H^1(\Omega)).
\end{align*}
For the strong convergence in $L^2((0,T), H^{\beta}(\Omega))$ we need some control with respect to the time-variable to apply a Kolmogorov-Simon-type compactness result (see \cite[Theorem 5]{Simon}. We emphasize that a direct application of the Aubin-Lions-Lemma is not possible, because we have no control of $\partial_t ( E_{\vareps}^f \ceps^f)$. Hence, we argue in the same way as in the proof of \cite[Lemma 10]{Gahn}. Using the linearity of $E_{\vareps}^f$ we obtain together with Lemma \ref{lem:Nikolskii_estimate} for every $0<h \ll 1$
\begin{align*}
\|\tceps^f(\cdot_t + h, \cdot_x) - \tceps^f&\|_{L^2((0,T-h) \times \Omega)} \le
C \|\ceps^f(\cdot_t + h, \cdot_x) - \ceps^f\|_{L^2((0,T-h) \times \Omega)}
\\
&\le h^{\frac14} \|\ceps^f(\cdot_t + h ,\cdot_x) - \ceps^f\|^{\frac12}_{L^2((0,T),H^1(\oef))} \|\partial_t \ceps^f\|^{\frac12}_{L^2((0,T),H^1_{\#}(\oef)^{\ast})}
\\
&\le Ch^{\frac14}.
\end{align*}
Since $\tceps$ is bounded in $L^2((0,T),H^1(\Omega))$ and the embedding $H^1(\Omega)\hookrightarrow H^{\beta}(\Omega)$ is compact, we obtain the strong convergence of $\tceps^f$ from \cite[Theorem 5]{Simon}.

It remains to check the weak convergence of the time derivative of the zero extension. Since $\partial_t (\chi_{\oef} \ceps^f)$ is bounded in $L^2((0,T),H^1_{\#}(\Omega)^{\ast})$, there exists a subsequence and a function $W\in L^2((0,T),H^1_{\#}(\Omega)^{\ast})$ such that $\partial_t (\chi_{\oef} \ceps^f )\rightharpoonup W$ in $L^2((0,T),H^1_{\#}(\Omega)^{\ast}).$ Hence, we have for all $\psi \in C_0^{\infty}((0,T))$ and $\phi \in H^1_{\#}(\Omega)$
\begin{align*}
\int_0^T   \langle W, \phi\rangle_{H^1_{\#}(\Omega)} \psi(t) dt &= \lim_{\vareps \to 0} \int_0^T \int_{\Omega} \chi_{\oef} \ceps^f \phi \psi'dx dt.  
\end{align*}
We split the integral on the right-hand side in the bulk terms and the membrane part $\oemf$ and calculate the limits:
\begin{align*}
\int_0^T\int_{\oeps^{\pm}} \ceps^f \phi \psi' dx dt = \int_0^T \int_{\Omega} \tceps^f \chi_{\oeps^{\pm}} \phi \psi' dx dt \rightarrow \int_0^T \int_{\Omega^{\pm}} c_0^f \phi \psi' dx dt.
\end{align*}
Further, we have with the \textit{a priori} estimates for $\ceps^f$ from Proposition \ref{prop:existence_apriori}
\begin{align*}
\left| \int_0^T \int_{\oemf} \ceps^f \phi \psi' dx dt \right| \le C \|\ceps^f\|_{L^2((0,T)\times \oemf)} \|\phi\|_{L^2(\oemf)} \le C \sqrt{\vareps} \|\phi\|_{L^2(\Omega)} \rightarrow 0.
\end{align*}
This implies  $W = \partial_t c_0^f$. Finally, the convergence $\tceps^f \rats c_0^f|_{\Sigma}$  on $\geps$ follows directly from Lemma \ref{lem:weak_two_scale_Geps} respectively Corollary \ref{cor:weak_ts_conv_geps}.
\end{proof}

\begin{lemma}\label{lem:aux_lemma}
For every $u\in H^1(\Omega)$ it holds that 
\begin{align*}
\teps^M u &\rightarrow u|_{\Sigma} &\mbox{ in }& L^2(\Sigma \times Z),
\\
\teps^M u|_{\Gamma}  &\rightarrow u|_{\Sigma} &\mbox{ in }& L^2(\Sigma \times \Gamma).
\end{align*}
\end{lemma}
\begin{proof}
We only prove the convergence in $L^2(\Sigma \times \Gamma)$. The other case follows by similar arguments (using Lemma \ref{lem:estimate_L2_membrane_Bulk_Gradient} instead of Lemma \ref{lem:Trace_inequality_Bulk}). Let $u_k \in C^{\infty}(\overline{\Omega})$ with $u_k \rightarrow u$ in $H^1(\Omega)$. With the trace inequality in Lemma \ref{lem:Trace_inequality_Bulk} and the elemental properties of the unfolding operator, we have
\begin{align*}
\|\teps^M u - u|_{\Sigma} \|_{L^2(\Sigma \times \Gamma)} &\le \|u_k - u\|_{L^2(\geps)} + \|\teps^M u_k  - u_k|_{\Sigma}\|_{L^2(\Sigma \times \Gamma)} + \|u_k|_{\Sigma} - u|_{\Sigma} \|_{L^2(\Sigma \times \Gamma)}
\\
&\le C \|u_k - u\|_{H^1(\Omega)} + \|\teps^M u_k - u_k \|_{L^2(\Sigma \times \Gamma)} + |\Gamma|^{\frac12} \|u_k - u\|_{L^2(\Sigma)}.
\end{align*}
The first and third term converge to zero for $k\to \infty$, due to the strong $H^1$-convergence of $u_k$ to $u$ and the trace inequality. The second term vanishes for $\vareps \to 0$ (for fixed $k$) by the dominated convergence theorem.
\end{proof}

Now, we are able to show the strong two-scale convergence of $\tceps^f$ in the membrane $\oeps^M$ and $\geps \cup S_{\vareps}^{\pm}$. We emphasize that here we have no good control of gradient compared to the case $\gamma=-1$ for $\ceps^s$ (see Section \ref{sec:conv_ceps_s} below).

\begin{proposition}
Up to a subsequence it holds that
\begin{align*}
   \tceps^f|_{\oeps^M} &\rasts c_0^f|_{\Sigma}  &\mbox{ strongly in the two-scale sense in }& \oeps^M,
   \\
   \ceps^f|_{\geps } &\rasts c_0^f|_{\Sigma} &\mbox{ strongly in the two-scale sense on }& \geps.
\end{align*}
In other words, we have $\teps^M \tceps^f \rightarrow c_0^f|_{\Sigma} $ in $L^2((0,T)\times \Sigma \times Z)$ and $\teps^M \ceps^f|_{\Gamma} \rightarrow c_0^f|_{\Sigma}$ in $L^2((0,T)\times \Sigma \times \Gamma )$.
\end{proposition}
\begin{proof}
Using the properties of the unfolding operator and Lemma \ref{lem:estimate_L2_norm_membrane_Hbeta} we obtain for $\beta \in (\frac12 ,1)$
\begin{align*}
\|\teps^M \tceps^f - c_0^f|_{\Sigma}\|_{L^2((0,T)\times \Sigma \times Z)} &\le \frac{1}{\sqrt{\vareps}} \| \tceps^f - c_0^f\|_{L^2((0,T)\times \oeps^M)} + \|\teps^M c_0^f - c_0^f|_{\Sigma}\|_{L^2((0,T)\times \Sigma \times Z)}
\\
&\le C \|\tceps^f - c_0^f\|_{L^2((0,T),H^{\beta}(\Omega))} +\|\teps^M c_0^f - c_0^f|_{\Sigma}\|_{L^2((0,T)\times \Sigma \times Z)}.
\end{align*}
The first term converges to zero for $\vareps \to 0$, because of the strong convergence of $\tceps^f$ to $c_0^f$ in $L^2((0,T),H^{\beta}(\Omega))$ from Proposition \ref{prop:convergence_cepsf}. The second term vanishes in the limit, due to Lemma \ref{lem:aux_lemma}. This gives the strong convergence of $\tceps^f|_{\oem}$.  The convergence of the traces follows by similar arguments using the trace inequality $\eqref{ineq:lem_trace_inequality_bulk_Hbeta}$ in Lemma \ref{lem:Trace_inequality_Bulk}.
\end{proof}

\begin{remark}
The Proposition remains valid if we replace $\geps$ by $\geps \cup S_{\vareps,f}^{\pm}$ (respectively $\Gamma$ with $\Gamma \cup S_f^{\pm}$), and the proof follows the same lines.
\end{remark}

\subsection{Convergence of the solute concentration in the solid domain}
\label{sec:conv_ceps_s}

Now, we give the compactness results for the concentration $\ceps^s$ in the solid phase of the thin layer $\oems$. Again, the only critical question is the strong two-scale convergence on $\geps$, necessary for the nonlinear boundary term. Hence, from a general perspective, we need strong two-scale compactness results for thin perforated layers under weak regularity conditions for the time-derivative (not in $L^2$). Such results were obtained for the full thin layer in \cite{GahnNeussRaduKnabner2018a}. For the case $\gamma = 1$ the proof is quite similar, and we formulate here a general two-scale compactness result of Kolmogorov-Simon-type based suitable \textit{a priori} bounds and additional estimates for the shifts with respect to the spatial variable. Such a result was formulated in \cite{GahnPop} for periodically perforated domains. For $\gamma \in [-1,1)$ we extend the function with the extension operator from Lemma \ref{lem:Extension_operators} to the whole layer $\oeps^M$ and now average this quantity with respect to $x_3$-variable. Compared to \cite{GahnNeussRaduKnabner2018a}, where a similar argument was used, we have no direct control of the time-variable for this sequence of averaged functions, because the time-derivative is not commutable with the extension operator (for low regularity of the time-derivative). However, we control shifts with respect to time in the same way as for $\ceps^f$.

\subsubsection{The case $\gamma = -1$}

For $\gamma = -1$ we obtain from Proposition \ref{prop:existence_apriori} the following \textit{a priori} estimates:
\begin{align*}
\frac{1}{\sqrt{\vareps}} \|\partial_t \ceps^s \|_{L^2((0,T),H^1(\oems)^{\ast})} + \frac{1}{\sqrt{\vareps}} \|\ceps^s\|_{L^2((0,T),H^1(\oems)} \le C.
\end{align*}
With the extension operator $E_{\vareps}^s$ from Lemma \ref{lem:Extension_operators} we define $\tceps^s:= E_{\vareps}^s \ceps^s$ and obtain 
\begin{align*}
    \frac{1}{\sqrt{\vareps}}\|\tceps^s\|_{L^2((0,T),H^1(\oem))} \le C.
\end{align*}
Again (as for $\ceps^f$), we have no control for $\partial_t \tceps^s$.  We argue in a similar way as in \cite{GahnNeussRaduKnabner2018a}, whereby a perforated layer $\oems$ must be taken into account here. However, we will overcome this problem by using the extension operator. We define the averaged function $\bceps^s$  by
\begin{align}\label{def:averaged_function}
    \bceps^s(\x):= \frac{1}{2\vareps} \int_{-\vareps}^{\vareps} \tceps^s (x) dx_3.
\end{align}

The following lemma was shown in the proof of  \cite[Proposition 4]{GahnNeussRaduKnabner2018a}. For the sake of completeness we formulate it here as an own result and give the proof for arbitrary $p \in (1,\infty)$.
\begin{lemma}\label{lem:general_strong_ts_convergence_averaged_function}
Let $p \in (1,\infty)$ and denote its dual exponent by $p'$. For an arbitrary function $\ueps \in L^p(\oeps^M)$ we define its average $\bueps \in L^p(\Sigma)$ by
\begin{align*}
\bueps(\x) : = \frac{1}{2\vareps} \int_{-\vareps}^{\vareps} \ueps (x) dx_3.
\end{align*}
Let $\ueps \in W^{1,p}(\oeps^M)$ be a sequence, such that $\bueps \rightarrow u_0$ in $L^p(\Sigma)$ for a function $u_0\in L^p(\Sigma)$ and assume that 
\begin{align*}
    \lim_{\vareps \to 0} \vareps^{\frac{1}{p'}} \|\partial_n \ueps\|_{L^p(\oeps^M)} = 0.
\end{align*}
Then $\ueps$ converges strongly in the two-scale sense in $L^p$ to $u_0$, i.e., the unfolded sequence $\teps^M \ueps $ converges strongly to $u_0$ in $L^p(\Sigma \times Z)$.
The result is also valid in the time-dependent case.
\end{lemma}
\begin{proof}
It is enough to show that 
\begin{align}\label{convergence:aux_lemma}
    \vareps^{-\frac{1}{p}} \| \ueps - u_0 \|_{L^p(\oeps^M)} \rightarrow 0
\end{align}
for $\vareps \to 0$. In fact, using the equivalent characterization of the strong two-scale convergence via the unfolding operator, we obtain
\begin{align*}
\|\teps^M \ueps - u_0 \|_{L^p(\Sigma \times Z)} &\le \|\teps^M \ueps - \teps^M u_0 \|_{L^p(\Sigma \times Z)} + \|\teps^M u_0 - u_0 \|_{L^p(\Sigma \times Z)}
\\
&= \vareps^{-\frac{1}{p}} \|\ueps - u_0 \|_{L^p(\oeps^M)} + \|\teps^M u_0 - u_0 \|_{L^p(\Sigma \times Z)}.
\end{align*}
Since the second term vanishes for $\vareps \to 0$, it follows that $\eqref{convergence:aux_lemma}$ implies the strong two-scale convergence of $\ueps$. To check $\eqref{convergence:aux_lemma}$ we use the simple estimate (see the proof of  \cite[Lemma 5.1]{GahnNeussRaduKnabner2018a} for more details in the case $p=2$)
\begin{align*}
\vareps^{-\frac{1}{p}} \|\ueps - \bueps \|_{L^p(\oeps^M)} \le C \vareps^{\frac{1}{p'}} \|\partial_n \ueps\|_{L^p(\oeps^M)},
\end{align*}
to obtain 
\begin{align*}
    \vareps^{-\frac{1}{p}} \|\ueps - u_0\|_{L^p(\oeps^M)} &\le \vareps^{-\frac{1}{p}} \|\ueps - \bueps\|_{L^p(\oeps^M)} + \vareps^{-\frac{1}{p}}  \| \bueps - u_0 \|_{L^p(\oeps^M)} 
    \\
    &\le C \vareps^{\frac{1}{p'}} \|\partial_n \ueps\|_{L^p(\oeps^M)} +  C \|\bueps - u_0\|_{L^p(\Sigma)} \overset{\vareps \to 0}{\longrightarrow} 0.
\end{align*}

\end{proof}
Now, we show the convergence results for $\tceps^s $, especially the strong convergence of $\tceps^s $ in the two-scale sense in $\oeps^M$ and on the surface $\geps$:
\begin{proposition}[The case $\gamma = -1$]\label{prop:strong_TS_convergence_cepss_gamma_-1}
Let $\gamma = -1$. There exist limit functions $c_0^s \in L^2((0,T),H^1_{\#}(\Sigma)) \cap H^1((0,T),H^1_{\#}(\Sigma)^{\ast})$ and $c_1^s \in L^2((0,T)\times \Sigma,H_{\#}^1(Z_s))$, such that up to a subsequence
\begin{align*}
\nabla \tceps^s \rats \nabla_{\x} c_0^s + \nabla_y c_1^s,\qquad 
    \tceps^s \rasts c_0^s, \qquad \ceps^s|_{\geps} \rasts c_0^s \quad \mbox{ on }\geps.
\end{align*}
Further, for all $\phieps(t,x) := \phi \left(t,\x,\fxe\right)$ with $\phi \in C_0^{\infty}((0,T), C_{\#}^{\infty}(\Sigma, C_{\#}^0(Z_s)))$ and all $\eta \in L^2((0,T),H^1_{\#}(\Sigma))$ it holds that (again up to a subsequence)
\begin{align*}
    \lim_{\vareps \to 0} \foe \int_0^T \langle \partial_t \ceps^s , \eta + \phieps \rangle_{H^1_{\#}(\oems)} dt = \int_0^T \int_{Z_s} \langle \partial_t c_0^s, \eta + \phi(\cdot_t,\cdot_{\x},y ) \rangle_{H^1_{\#}(\Sigma)} dy dt.
\end{align*}
\end{proposition}
\begin{proof}
The existence of functions $c_0^s$ and $c_1^s$ together with the weak two-scale convergence for $\tceps^s$ and $\nabla \tceps^s$ is  standard, see Lemma \ref{LemmaBasicTSCompactness}.
To establish the strong two-scale convergence in the thin layer $\oeps^M$, we use Lemma \ref{lem:general_strong_ts_convergence_averaged_function}. Hence, we have to check the properties of the averaged function $\bceps^s$. We argue in a similar way as in the proof of Proposition \ref{prop:convergence_cepsf}. We have for $0< h \ll 1$ (see also \cite[Lemma 5.1]{GahnNeussRaduKnabner2018a}) with Lemma \ref{lem:Nikolskii_estimate}
\begin{align*}
\int_0^{T-h } \|\bceps^s (t + h) - \bceps^s\|_{L^2(\Sigma)}^2 dt &\le \frac{C}{\vareps} \int_0^{T-h} \|\tceps^s (t + h ) - \tceps^s\|_{L^2(\oeps^M)}^2 dt 
\\
&\le \frac{C}{\vareps} \int_0^{T-h} \|\ceps^s(t+h) - \ceps\|_{L^2(\oeps^M)}^2 dt 
\\
&\le \frac{C\sqrt{h}}{\vareps} \|\partial_t \ceps^s\|_{L^2((0,T),H^1(\oems)^{\ast})}  \|\ceps^s\|_{L^2((0,T),H^1(\oems))} 
\\
&\le C\sqrt{h}.
\end{align*}
Further, we have (use again \cite[Lemma 5.1]{GahnNeussRaduKnabner2018a})
\begin{align*}
    \|\bceps^s\|_{L^2((0,T), H^1(\Sigma))} \le \frac{1}{\sqrt{2\vareps}}\|\tceps^s\|_{L^2((0,T),H^1(\oeps^M)} \le C.
\end{align*}
Hence, we can apply \cite[Theorem 1]{Simon} and obtain existence of $c_0^s \in L^2((0,T),H^1(\Sigma))$ such that up to a subsequence
\begin{align*}
    \bceps^s &\rightharpoonup c_0^s &\mbox{ in }& L^2((0,T),H^1(\Sigma)),
    \\
    \bceps^s &\rightarrow c_0^s &\mbox{ in }& L^2((0,T)\times \Sigma).
\end{align*}
Now, Lemma \ref{lem:general_strong_ts_convergence_averaged_function} implies the strong two-scale convergence of $\tceps$ in $\oeps^M$, and Lemma \ref{lem:strong_ts_convergence_boundary} gives the strong two-scale convergence on $\geps$. Finally, the existence of $\partial_t c_0^s$ and the associated convergence follows directly from Lemma \ref{lem:two_scale_conv_time_derivative} in the appendix.

\end{proof}

\subsubsection{The cases $\gamma \in (-1,1)$ and $\gamma = 1$}

Here, we treat the cases $\gamma \in (-1,1)$ and $\gamma = 1$. Although these two cases behave completely different in the limit, they have one significant similarity compared to the case $\gamma = -1$: The \textit{a priori} estimates obtained in Proposition \ref{prop:existence_apriori} are not enough to obtain strong two-scale convergence results in $\oeps^M$ and on $\geps$. From a mathematical point of view, compared to the proof of the strong two-scale convergence in the case $\gamma  =-1$, for an application of the Kolmogorov-Simon-type compactness result the uniform bounds for the gradient are not good enough.  What we need is more control of the spatial variable, more precisely we have to control the differences for spatial shifts of $\ceps^s$.

The situation gets even more complicated for $\gamma = 1$. In this case, the limit behavior changes completely compared to the cases $\gamma \in [-1,1)$. In fact, the two-scale limit depends on the macroscopic and microscopic variable $(\x,y) \in \Sigma \times Z_s$.  While in the cases $\gamma \in [-1,1)$ the (weak/strong) two-scale limit of $\ceps^s$ and the (weak/strong) limit of the averaged function $\bceps^s$ are equal (and both convergences are equivalent), this is no longer valid for the case $\gamma = 1$. Hence, in the latter we will argue with the unfolding operator. This leads to two additional difficulties for an application of a Kolmogorov-Simon-compactness result. First of all, the unfolded sequence is depending on an additional macroscopic variable $\x \in \Sigma$. We overcome this by controlling differences of the shifts for the unfolded sequence by differences of the shifts for the sequnce itself, see for example \cite{Gahn,GahnNeussRaduKnabner2018a,GahnNRP}. The second crucial point is the control of the time variable for the unfolded sequence, which is not so straightforward as in the cases $\gamma \in [-1,1)$ (obtained via Lemma \ref{lem:Nikolskii_estimate} and the properties of the extension operator). For this, we use a general compactness result (see Proposition \ref{prop:general_ts_compactness_gamma_1}) which depends in particular on estimates for the differences of the shifts and the bound for $\partial_t \ceps$ in $L^2((0,T),H_{\#}^1(\oems)^{\ast})$
\\

Let us start with the proof of some additional \textit{a priori} estimates for the differences of the shifts. 
For given $l \in \Z^{n-1} \times \{0\}$ we define (for a given function $\phi$) the difference of shifts by
\begin{align*}
    \delta \phi(x) := \phi(x + l\vareps) - \phi(x).
\end{align*}
We emphasize that in this notation we neglect the dependence of $\delta \phi$ on the specific shift $l\vareps$. This should be clear from the context. Otherwise, we also write $\delta_{l\vareps}$.  Due to the periodic boundary conditions on $\partial \Sigma$ for $\ceps^f$ and $\ceps^s$, the functions $\delta \ceps^f$ and $\delta \ceps^s$ are well-defined for all $l \in \Z^{n-1}\times \{0\}$.

\begin{lemma}\label{lem:apriori_shifts}
Let $l \in \Z^{n-1} \times \{0\}$. The microscopic solution $(\ceps^f,\ceps^s)$ from \eqref{eq_Var_Micro_cepsf}-\eqref{eq_Var_Micro_cepss} fulfills 
\begin{align*}
\frac{1}{\sqrt{\vareps}} \|\delta\ceps^s\|_{L^{\infty}((0,T),L^2(\oems))} +  &\vareps^{\frac{\gamma}{2}} \|\nabla \delta \ceps^s\|_{L^2((0,T)\times \oems)} 
\\
&\le C \left(\frac{1}{\sqrt{\vareps}} \|\delta c_{\vareps,\mathrm{in}}^s\|_{L^2(\oems)} + \|\delta \ceps^f\|_{L^2((0,T), H^{\beta}(\oef))}\right).
\end{align*}
\end{lemma}
\begin{proof}
Testing the weak formulation $\eqref{eq_Var_Micro_cepss}$ for $\ceps^s$ with $\delta \ceps^s$, we obtain with Lemma \ref{lem:Trace_inequality_Bulk} for $\delta \ceps^f$ and the trace inequality from Lemma \ref{lem:Stand_scaled_trace_inequality} for $\delta \ceps^s$ on $\geps$
\begin{align*}
\foe \langle \partial_t \delta \ceps^s ,\delta \ceps^s\rangle_{H^1_{\#}(\oems)} + \vareps^{\gamma} \|\nabla \delta \ceps^s\|_{L^2(\oems)}^2 &\le C \left( \|\delta \ceps^f\|_{L^2(\geps)}^2 + \|\delta \ceps^s\|_{L^2(\geps)}^2 \right)
\\
&\le C \|\delta \ceps^f\|_{H^{\beta}(\oef)}^2 + \frac{C_{\theta}}{\vareps} \|\delta \ceps^s\|_{L^2(\oems)}^2 + \theta \vareps \|\nabla \delta \ceps^s\|_{L^2(\oems)}^2
\end{align*}
with $\theta >0$ arbitrary. For $\theta $ small enough the last term on the right-hand side can be absorbed from the left-hand side. After an integration with respect to time and an application of Gronwall inequality, we obtain the desired result.
\end{proof}

\begin{remark}
Similar estimates can be found for example in \cite{GahnNeussRaduKnabner2018a} and \cite{GahnNRP}. Here the proof simplifies a lot, because we consider periodic boundary conditions on $\partial \Sigma$. The general case can be treated with the same arguments as in the aforementioned papers. However, we emphasize that here we argue in a slightly different way. We do not use both equations for $\ceps^f$ and $\ceps^s$, but estimate the shifts for $\ceps^s$ by shifts of $\ceps^f$ in $H^{\beta}(\oef)$, for which we already have strong compactness results.
\end{remark}

\begin{proposition}[The case $\gamma \in (-1,1)$]\label{prop:strong_ts_conv_gamma_-1_1}
Let $\gamma \in (-1,1)$. There exists a limit function $c_0^s \in L^2((0,T) \times \Sigma) \cap H^1((0,T),L^2(\Sigma))$, such that up to a subsequence
\begin{align*}
    \tceps^s \rasts c_0^s, \qquad \ceps^s|_{\geps} \rasts c_0^s \quad \mbox{ on }\geps.
\end{align*}
Further, for every $\phi \in L^2((0,T),H^1(\Sigma))$ it holds that 
\begin{align*}
    \foe \int_0^T \langle \partial_t \ceps^s,\phi \rangle_{H^1_{\#}(\oems)} dt \rightarrow |Z_s|\int_0^T \langle \partial_t c_0^s , \phi \rangle_{H^1_{\#}(\Sigma)} dt.
\end{align*}
\end{proposition}
\begin{proof}
The proof uses the same arguments as the proof of \cite[Theorem 7.3]{GahnNeussRaduKnabner2018a}. The only difference is that in a first step we have to extend the function $\ceps^s$ to the whole thin layer $\oeps^M$ to guarantee that the averaged function $\bceps^s$ is well defined. For the sake of completeness we sketch the main ideas. We first argue in a similar way as for the case $\gamma =-1$ and consider the averaged function (of the extension) $\bceps^s$ defined in $\eqref{def:averaged_function}$. By the same arguments as in the proof of Proposition \ref{prop:strong_TS_convergence_cepss_gamma_-1}, we obtain 
\begin{align*}
 \|\bceps^s (t + h) - \bceps^s\|_{L^2((0,T-h)\times\Sigma)}^2 &\le  \frac{C\sqrt{h}}{\vareps} \|\partial_t \ceps^s\|_{L^2((0,T),H_{\vareps,\gamma}(\oems)^{\ast})}  \|\ceps^s\|_{L^2((0,T),H_{\vareps,\gamma}(\oems))} 
\\
&\le C\sqrt{h}.
\end{align*}
The same argument holds if we consider the intervall $(h,T)$ instead of $(0,T-h)$. Next, we estimate shifts with respect to the spatial variable. For this we make use of the estimates for the differences of the shifts in Lemma \ref{lem:apriori_shifts}, since the estimates for the gradient of $\ceps^s$ are not enough.
Let $\hxi \in \R^{n-1}$ and $\xi:= (\hxi ,0)$ and $\xi_{\vareps} := \vareps \left( \left[\frac{\hxi}{\vareps}\right] , 0\right)$. We emphasize that for $x \in \oems$ there exists $k \in \Z^{n-1} \times \{0\}$ with $x + \xi_{\vareps} \in \vareps (k + Z_s)$, but in general we can have $x + \xi \notin \vareps(k+Z_s)$. 
We obtain 
\begin{align*}
\|&\bceps^s (\cdot, \cdot + \hxi) - \bceps^s \|_{L^2((0,T) \times \Sigma)} \le \frac{1}{2\sqrt{\vareps}} \| \tceps^s (\cdot , \cdot + \xi ) - \tceps^s \|_{L^2((0,T)\times \oeps^M)} 
\\
&\le \frac{1}{\sqrt{\vareps}} \| \tceps^s (\cdot , \cdot + \xi) - \tceps^s (\cdot , \cdot + \xi_{\vareps}) \|_{L^2((0,T)\times \oeps^M)} + \frac{1}{\sqrt{\vareps}} \| \tceps^s(\cdot , \cdot + \xi_{\vareps}) - \tceps^s\|_{L^2((0,T)\times \oeps^M)} 
\\
&=: A_{\vareps}^1 + A_{\vareps}^2.
\end{align*}
For the first term $A_{\vareps}^1$ we use the mean value theorem, the $\partial \Sigma$-periodicity of $\ceps^s$ (resp. $\tceps^s$), the properties of the extension operator from Lemma \ref{lem:Extension_operators}, the \textit{a priori} bound for $\nabla \ceps^s$ from Proposition \ref{prop:existence_apriori}, and the fact $|\xi - \xi_{\vareps}| \le C\vareps$ to get
\begin{align*}
A_{\vareps}^1 \le C\sqrt{\vareps} \|\nabla \tceps^s\|_{L^2((0,T)\times \oeps^M)} \le C\varepsilon^{\frac{1-\gamma}{2}}.
\end{align*}
For the second term $A_{\vareps}^2$ we use again the properties of the extension operator from Lemma \ref{lem:Extension_operators}  and also Remark \ref{rem:extension_operator} to obtain
\begin{align*}
A_{\vareps}^2 = \frac{1}{\sqrt{\vareps}} \| E_{\vareps}^s( \ceps^s(\cdot , \cdot + \xi_\vareps)) - E_{\vareps}^s \ceps^s \|_{L^2((0,T)\times \oeps^M)} \le \frac{C}{\sqrt{\vareps}} \|\ceps^s (\cdot , \cdot + \xi_{\vareps}) - \ceps^s \|_{L^2((0,T)\times \oems)}.
\end{align*}
Using the \textit{a priori} bounds for the differences of the shifts from Lemma \ref{lem:apriori_shifts}, and the strong convergence of $\tceps^f$ in $L^2((0,T),H^{\beta}(\Omega))$ for $\beta \in \left(\frac12, 1\right)$, we obtain that $A^2_{\vareps} \rightarrow 0$ for $\vareps,|\xi| \to 0$. Together with the uniform bound for $\bceps^s$ in $L^2((0,T)\times \Sigma)$, we can apply the Kolmogorov-Simon-compactness result and obtain the strong convergence of $\bceps^s$ in $L^2((0,T)\times \Sigma).$  We emphasize that for the shifts with respect to time it is enough to consider shifts on the interval $(0,T-h)$ and $(h,T)$, see for example the argument in the third step in the proof of \cite[Theorem 1]{Simon}.

This strong convergence implies the strong two-scale convergence of $\tceps^s$ in $\oeps^M$, see Lemma \ref{lem:general_strong_ts_convergence_averaged_function}. By Lemma \ref{lem:strong_ts_convergence_boundary} we obtain again the strong two-scale convergence on $\geps$. For the convergence of the time-derivative and the regularity for $\partial_t c_0^s$ we use Lemma \ref{lem:two_scale_conv_time_derivative}.
\end{proof}
The convergence result for the time-derivative is also valid for suitable oscillating test-functions, see Lemma \ref{lem:two_scale_conv_time_derivative}. However, for the derivation of the macroscopic model such test-functions are not necessary.
\begin{remark}
In contrast to \cite[Theorem 7.3]{GahnNeussRaduKnabner2018a} we obtain here the strong convergence of $\bceps^s$ in $L^2$, instead of $L^p$ for $p \in [1,2)$. Here we can improve the inegrability exponent, since we consider periodic boundary conditions.
\end{remark}

Finally, we have to consider the most critical case $\gamma = 1$. In this case it makes no sense to consider the averaged sequence $\bceps^s$ as above, since the two-scale limit of $\tceps^s$ is depending on $\x$ and $y$. Instead, we consider the unfolded sequence $\teps^M \ceps^s$ and can argue in the same way as in \cite[Theorem 7.5]{GahnNeussRaduKnabner2018a} (no perforations inside the thin layer). An extension to the whole layer is therefore not necessary. However, here we formulate a general strong two-scale (or equivalent formulation via unfolding operator) convergence results based on \textit{a priori} estimates for the sequence and the differences for the shifts. This result is analogous to \cite[Theorem 1]{GahnNRP}, where a perforated domain (not thin) was considered.

\begin{proposition}\label{prop:general_ts_compactness_gamma_1}
Let $\ueps \in L^2((0,T),H^1_{\#}(\oems))\cap H^1((0,T),H_{\vareps,1}(\oems))$ with:
\begin{enumerate}[label = (\roman*)]
\item It holds that
\begin{align*}
\|\ueps\|_{L^2((0,T),H_{\vareps,1}(\oems))} + \foe\|\partial_t \ueps\|_{L^2((0,T),H_{\vareps,1}(\oems)} \le  C.
\end{align*}

\item For $l_{\vareps} \in  \vareps \Z^{n-1} \times \{0\}$ with $l_{\vareps } \to 0$ for $\vareps \to 0$ it holds that (with $\delta = \delta_{l_{\vareps}}$)
\begin{align*}
\|\delta \ueps \|_{L^2((0,T)\times \oems)} + \vareps \|\nabla \delta \ueps\|_{L^2((0,T)\times \oems)} \overset{\vareps \to 0}{\longrightarrow} 0.
\end{align*}
\end{enumerate}
Then, there exists $u_0 \in L^2((0,T)\times \Sigma , H_{\#}^1(Z_s))$, such that for every $\beta \in \left(\frac12,1\right)$ it holds that
\begin{align*}
\teps^M \ueps \rightarrow u_0 \qquad \mbox{in } L^2((0,T)\times \Sigma,H^{\beta}(Z_s)).
\end{align*}
In particular $\ueps$ converges strongly in the two-scale sense to $u_0$ in $\oeps^M$ and on $\geps$.
\end{proposition}
\begin{proof}
The proof follows the same ideas as the proof of \cite[Theorem 1 and Remark 5]{GahnNRP}. The fact that we are dealing with thin domains has no influence on the proof. For the sake of completeness let us shortly describe the main ideas: We consider the unfolded sequence $\teps^M \ueps \in L^2(\Sigma,L^2((0,T),H^{\beta}(Z_s)))$ and apply the Komlogorov-Simon-compactness results from \cite{GahnNeussRaduKolmogorovCompactness} for Bochner spaces $L^2(\Sigma,B)$ with rectangular domains $\Sigma$ in $\R^n$ and the Banach space $B= L^2((0,T),H^{\beta}(Z_s))$.  The control of the shifts in $\Sigma$ is an immediate consequence of (ii). To control the range of $\teps^M \ueps$ we use the compact embedding $H^1(Z_s) \hookrightarrow H^{\beta}(Z_s)$, and the commutability between the unfolding operator $\teps^M$ and the time-derivative, see \cite[Proposition 7 and 8]{GahnNeussRaduKnabner2018a}.  We emphasize that this property is not obvious, because of the low regularity of the time-derivative $\partial_t \ceps^s$ which is only an element of  $\partial_t \ceps$ in $L^2((0,T),H_{\#}^1(\oems)^{\ast})$.
\end{proof}
Again, we obtain in contrast to the results in \cite{GahnNeussRaduKnabner2018a} and \cite{GahnNRP} the strong convergence in $L^2$ instead of $L^p$ for $p \in [1,2)$, which follows again by the periodic boundary conditions on $\partial \Sigma$. As an immediate consequence we obtain the strong two-scale convergence of $\ceps^s$ in $\oems$ and on $\geps$:

\begin{proposition}[The case $\gamma = 1$]\label{prop:strong_ts_conv_gamma_1}
Let $\gamma = 1$. There exists a limit function 
\begin{align*}
c_0^s \in L^2((0,T)\times \Sigma, H^1_{\#}(Z_s))\cap H^1((0,T),L^2(\Sigma, H_{\#}^1(Z_s))^{\ast}),
\end{align*}
such that up to a subsequence
\begin{align*}
\vareps \nabla \tceps^s \rats \nabla_y c_0^s,\qquad    \tceps^s \rasts c_0^s, \qquad \ceps^s|_{\geps} \rasts c_0^s \quad \mbox{ on }\geps.
\end{align*}
Further, for all $\phieps(t,x) := \phi \left(t,\x,\fxe\right)$ with $\phi \in C_0^{\infty}([0,T)\times \overline{\Sigma}, C_{\#}^0(\overline{Z_s}))$ it holds that (up to a subsequence)
\begin{align*}
    \lim_{\vareps \to 0} \foe \int_0^T \langle \partial_t \ceps^s , \phieps \rangle_{H^1(\oems)} dt = \int_0^T \langle \partial_t c_0^s , \phi \rangle_{L^2(\Sigma,H_{\#}^1(Z_s))} dt.
\end{align*}
\end{proposition}
\begin{proof}
Follows from Proposition \ref{prop:general_ts_compactness_gamma_1} and the \textit{a priori} estimates in Proposition \ref{prop:existence_apriori} and Lemma \ref{lem:apriori_shifts}. The time regularity and the convergence of the time-derivative follows from Lemma \ref{lem:two_scale_conv_time_derivative}. The two-scale convergence of the gradient is classical, see Lemma \ref{LemmaBasicTSCompactness}.
\end{proof}
The strong two-scale convergence results of $\ceps^f$ and $\ceps^s$ on $\geps$ immediately imply the strong two-scale convergence of the sequence of nonlinear functions:
\begin{corollary}\label{cor:ts_conv_nonlinear_term}
Let $c_0^f$ be the limit function from Propostion \ref{prop:convergence_cepsf} and $c_0^s$ the limit function from Proposition \ref{prop:strong_TS_convergence_cepss_gamma_-1} (for $\gamma = -1$) resp. Proposition \ref{prop:strong_ts_conv_gamma_-1_1} (for $\gamma \in (-1,1)$) or Proposition \ref{prop:strong_ts_conv_gamma_1} (for $\gamma  =1$). Then, up to a subsequence, it holds that
\begin{align*}
    h(\ceps^f,\ceps^s) \rasts h(c_0^f|_{\Sigma},c_0^s) \quad\mbox{ on }\geps.
\end{align*}
\end{corollary}

\section{Derivation of the macroscopic model}
\label{sec:derivation_macro_model}

With the compactness results obtained for the microscopic solution $(\veps,\peps,\ceps^f,\ceps^s)$ from Section \ref{sec:comptness_results} we are able to pass to the limit in the respective microscopic equations. For the fluid flow we have to identify the interface conditions across $\Sigma$. For the transport problems we can follow classical approaches to pass to the limit $\vareps\to 0$, see for example \cite{GahnEffectiveTransmissionContinuous,GahnNeussRaduKnabner2018a}. However, since the scaling for the time-derivative in the thin layer is different, the only contribution from the thin layer (in the limit the interface) on the concentration of the solute in the fluid phase is the nonlinear coupling term. In the following we treat the fluid equations and transport equations separately.

\subsection{Effective interface law for the fluid flow}
\label{sec:derivation_fluid_model}

We derive the macroscopic equation for the fluid problem. While we can expect to obtain Stokes-equations in the bulk domains $\Omega^{\pm}$, the interface conditions across $\Sigma$ for the macroscopic fluid velocity and pressure are not obvious. For the derivation we choose suitable test-functions in the microscopic equation $\eqref{eq:Var_Micro_veps}$ adapted to the structure of the limit function $v_0$. In particular, in the thin layer we have to choose oscillating test-functions. Hence, to pass to the limit for the terms in the thin layer we use similar arguments as for Stokes flow in thin perforated layers. However, the coupling to the bulk domains has significant contributions to the limit equations, especially for the cell problems and the Darcy velocity at $\Sigma$. Further, we have to construct a special divergence operator including coupling conditions across $\Sigma$ to generate a pressure $p_1^M$ associated to the fluid velocity $v_0^M$.
\\

\noindent\textit{Proof of Theorem \ref{MainThm:Fluid}:} The convergence results for $(\veps,\peps)$ were already established in Proposition \ref{prop:compactness_fluid} and the continiuty of the normal velocity across $\Sigma$ in Corollary \ref{cor:continuity_normal_component_v_0}. It remains to show that $(v_0,p_0)$ is the unique weak solution of the macro model $\eqref{def:Macro_Stokes_model_strong}$. Let us choose test-functions $\phieps$ in the variational equation $\eqref{eq:Var_Micro_veps}$ of the form
\begin{align*}
\phieps(x):= \begin{cases}
\phi^{\pm} \left( t,x \mp \vareps e_n\right) &\mbox{ for } x \in \oeps^{\pm},
\\
\phi^M\left(t,\x,\fxe\right) &\mbox{ for } x \in \oemf,
\\
0 &\mbox{ for } x \in \oems,
\end{cases}
\end{align*}
for functions $\phi^{\pm} \in  C_0^{\infty}\left((0,T),C_{\#}^{\infty}\left( \overline{\Omega^{\pm}}\right)\right)^3$ fulfilling $\phi^+_3 = \phi^-_3$ on $(0,T)\times \Sigma$, and $\phi^M \in C_0^{\infty}\left((0,T),C_{\#}^{\infty}\left(\Sigma, H_{\#}^1(Z_f,\Gamma)\right)\right)^3$ with $\nabla_y \cdot \phi^M = 0$ and $\phi^M = \phi^{\pm}(\x)$ on $(0, T) \times \Sigma \times S^{\pm}$. The last condition implies that $\phi^M$ is constant on $S^{\pm}$ with respect to $y$.
The motivation for this choice of test functions is that it fits to the structure of the limit function $(v_0^{\pm}, v_0^M)$ from Proposition \ref{prop:compactness_fluid}, see also Corollary \ref{cor:continuity_normal_component_v_0}. Hence, we obtain almost everywhere in $(0,T)$
\begin{align*}
\sum_{\pm}& \bigg\{ \int_{\oeps^{\pm}} \partial_t \veps^{\pm}\cdot  \phi^{\pm}(x \mp \vareps e_n)  dx +  \int_{\oeps^{\pm}} D(\veps^{\pm}): D(\phi^{\pm})(x \mp \vareps e_n) dx   -\int_{\oeps^{\pm}} \peps^{\pm} \nabla \cdot \phi^{\pm}(x \mp \vareps e_n) dx \bigg\}
\\
+& \int_{\oemf} \partial_t \veps^M \cdot \phi^M  \bxfxe dx  + \foe\int_{\oemf} \vareps D(\veps^M) : \left[ \vareps D_{\x}(\phi^M) + D_y(\phi^M)\right]\bxfxe  dx
\\
-&  \foe \int_{\oemf} \peps^M \nabla_{\x} \cdot \phi^M \bxfxe dx 
= \sum_{\pm} \int_{\oeps^{\pm}} f_{\vareps}^{\pm} \cdot \phi^{\pm}(x \mp \vareps e_n) dx .
\end{align*}
Due to Proposition \ref{prop:existence_apriori}, the term including the time-derivative $\partial_t \veps^M$ is of order $\sqrt{\vareps}$ (after integration with respect to time). Integration with respect to time and using the compactness results from Proposition \ref{prop:compactness_fluid}, we obtain for $\vareps \to 0$
\begin{align}
\begin{aligned}\label{eq:aux_Darcy_law}
\sum_{\pm} & \left\{ \int_0^T \int_{\Omega^{\pm}} \partial_t v_0^{\pm} \cdot \phi^{\pm} dx dt + \int_0^T \int_{\Omega^{\pm}} D(v_0^{\pm}) : D(\phi^{\pm}) dxdt - \int_0^T  \int_{\Omega^{\pm}} p_0^{\pm} \nabla \cdot \phi^{\pm} dx  dt \right\} 
\\
&+ \int_0^T \int_{\Sigma} \int_{Z_f} D_y(v_0^M) : D_y(\phi^M) dy d\x dt  = \sum_{\pm} \int_0^T \int_{\Omega^{\pm}} f^{\pm}_0 \cdot \phi^{\pm} dx dt .
\end{aligned}
\end{align}
By density this result is valid for test functions $\phi = (\phi^+,\phi^M , \phi^-) \in L^2((0,T),H)$ with $\nabla_y \cdot \phi^M =0$, where the space $H$ is defined via
\begin{align*}
H:= \big\{ \phi \in H^1_{\#}(\Omega^+)^3 \times L^2(\Sigma,H_{\#}^1(Z_f,\Gamma))^3 \times &H^1_{\#}(\Omega^-)^3 \, : 
\\
&\phi^{\pm} = \phi^M \mbox{ on } \Sigma \times S^{\pm} , \phi^+_3 = \phi^-_3 \mbox{ on } \Sigma \big\}.
\end{align*}
We emphasize that $v_0 = (v_0^+,v_0^M,v_0^-) \in L^2((0,T),H)$. To construct a pressure in the cell $Z_f$ we consider the following divergence operator:
\begin{align*}
    Div_H : H \rightarrow L^2(\Sigma \times Z_f), \qquad \Div_H \phi = \nabla_y \cdot \phi^M.
\end{align*}
\begin{lemma}\label{lem:Bogovskii}
The range of the operator $\Div_H$ is equal to $L^2(\Sigma,L_0^2(Z_f))$ and is in particular closed. Further, for every $F \in H^{\ast}$ with $F(\phi) = 0$ for all $\phi \in H$ with $\Div_H \phi = 0$ there exists  a unique $p^M \in L^2(\Sigma, L_0^2(Z_f))$ with 
\begin{align*}
    F(\phi) = (p^M , \nabla_y \cdot \phi^M )_{L^2(\Sigma \times Z_f)} \qquad \mbox{for all } \phi \in H.
\end{align*}
\end{lemma}
\begin{proof}
For every $\phi \in H$ we have almost everywhere in $\Sigma$ by integration by parts (using $\phi^M = 0 $ on $\Gamma$)
\begin{align*}
\int_{Z_f} \Div_H \phi dy = \int_{Z_f} \nabla_y \cdot \phi^M dy = \int_{S^+} \phi^M_3 d\sigma_y - \int_{S^-} \phi^M_3 d\sigma_y = \phi^+_3 - \phi^-_3 = 0.
\end{align*}
Now, let $f \in L^2(\Sigma ,L_0^2(Z_f))$. By the classical Bogovskii theory there exists $\phi^M \in L^2(\Sigma,H_{\#}^1(Z_f,\Gamma \cup S^+ \cup S^-))^3$ such that $\nabla_y \cdot \phi^M = f$. Extending $\phi^M$ by zero to the bulk domains $\Omega^{\pm}$ we obtain a function $\phi \in H$ fulfilling $\Div_H \phi = f$, which implies $R(\Div_H) = L^2(Z,L_0^2(Z_f))$.

The second statement is just the closed range theorem. In fact, by assumption we have $F \in N(\Div_H)^{\perp} = R(\Div_H^{\ast})$ (the last equality follows from the closed range theorem). Hence, there exists $p^M \in L^2(\Sigma,L_0^2(Z_f))$ with $\Div_H^{\ast} p^M = F$. In other words, for all $\phi \in H$ it holds that
\begin{align*}
F(\phi ) = \langle \Div_H^{\ast} p ,\phi \rangle_H = (p,\Div_H \phi)_{L^2(\Sigma \times Z_f)} = (p, \nabla_y \cdot \phi^M)_{L^2(\Sigma \times Z_f)}.
\end{align*}
Uniqueness is clear.
\end{proof}
From equation $\eqref{eq:aux_Darcy_law}$ and Lemma \ref{lem:Bogovskii} we obtain the existence of a unique $p_1^M \in L^2((0,T)\times \Sigma,L_0^2(Z_f))$ such that for all $\phi = (\phi^+,\phi^M , \phi^-) \in L^2((0,T),H)$ and almost everywhere in $(0,T)$ it holds that
\begin{align}
\begin{aligned}\label{eq:two_scale_model_fluid}
\sum_{\pm} & \left\{ \int_{\Omega^{\pm}} \partial_t v_0^{\pm} \cdot \phi^{\pm} dx +\int_{\Omega^{\pm}} D(v_0^{\pm}) : D(\phi^{\pm}) dx -  \int_{\Omega^{\pm}} p_0^{\pm} \nabla \cdot \phi^{\pm} dx   \right\} 
\\
&+ \int_{\Sigma} \int_{Z_f} D_y(v_0^M) : D_y(\phi^M) dy d\x - \int_{\Sigma} \int_{Z_f} p_1^M \nabla_y \cdot \phi^M dy d\x = \sum_{\pm} \int_{\Omega^{\pm}} f^{\pm}_0 \cdot \phi^{\pm} dx  . 
\end{aligned}
\end{align}
We summarize the previous results in the following Proposition:
\begin{proposition}\label{prop:two_scale_model_fluid}
The limit function $(v_0,p_0)$ with $v_0 = (v_0^+,v_0^M ,v_0^-)$ and $p_0 = (p_0^+,p_0^-)$ from Proposition \ref{prop:compactness_fluid} fulfills $v_0 \in L^2((0,T),H)$ and is the unique weak solution of 
\begin{align*}
\partial_t v_0^{\pm} - \nabla\cdot D(v_0^{\pm}) + \nabla p_0^{\pm} &= f_0^{\pm} &\mbox{ in }& (0,T)\times \Omega^{\pm},
\\
\nabla \cdot v_0^{\pm} &= 0 &\mbox{ in }& (0,T)\times \Omega^{\pm},
\\
[v_0^+]_3 &= [v_0^-]_3 &\mbox{ on }& (0,T)\times \Sigma,
\\
-[D(v_0^{\pm}) - p_0^{\pm}I]\nu &= 0 &\mbox{ on }& (0,T)\times \Sigma \times \{\pm H\},
\\
-\nabla_y \cdot D_y(v_0^M) + \nabla_y p_1^M &= 0 &\mbox{ in }& (0,T)\times \Sigma \times Z_f,
\\
\nabla_y \cdot v_0^M &= 0 &\mbox{ in }& (0,T)\times \Sigma \times Z_f,
\\
v_0^M &= 0 &\mbox{ on }& (0,T)\times \Sigma \times \Gamma,
\\
v_0^{\pm} &= v_0^M &\mbox{ on }& (0,T)\times \Sigma \times S^{\pm},
\\
v_0^{\pm}(0) &= 0 &\mbox{ in }& \Omega^{\pm},
\\
v_0^{\pm} \,\, \Sigma\mbox{-periodic},\,\, v_0^M \,\, Y\mbox{-periodic}.
\end{align*}
A weak solution of this problem are functions $v_0\in L^2((0,T),H)$ with $\partial_t v_0^{\pm} \in L^2((0,T)\times \Omega^{\pm})$ and $\nabla \cdot v_0^{\pm} = 0$ resp. $\nabla_y \cdot v_0^M = 0$, together with a pressure $(p_0^+,p_1^M,p_0^-) \in L^2(\Omega^+)\times L^2(\Sigma,L_0^2(Z_f)) \times L^2(\Omega^-)$, such that $\eqref{eq:two_scale_model_fluid}$ is valid for all $\phi \in H$ almost everywhere in $(0,T)$.
\end{proposition}
\begin{proof}
Uniqueness is obvious and the initial conditions can be obtained by similar arugments as above with test-functions not vanishing in $t= 0$.
\end{proof}
In a next step, we will formulate the problem on the bulk domains $\Omega^+$ and $\Omega^-$ with interface conditions across $\Sigma$. From Proposition \ref{prop:two_scale_model_fluid} we obtain that $(v_0^M,p_1^M)$ is the unique weak solution of the problem 
\begin{align*}
-\nabla_y \cdot D_y(v_0^M) + \nabla_y p_1^M &= 0 &\mbox{ in }& (0,T)\times \Sigma \times Z_f,
\\
\nabla_y \cdot v_0^M &= 0 &\mbox{ in }& (0,T)\times \Sigma \times Z_f,
\\
v_0^M &= 0 &\mbox{ on }& (0,T)\times \Sigma \times \Gamma,
\\
v_0^M &= v_0^{\pm} &\mbox{ on }& (0,T)\times \Sigma \times S^{\pm},
\\
v_0^M,\, p_1^M \,\, Y\mbox{-periodic}.
\end{align*}
We emphasize (again), that this problem has a weak solution, since $[v_0^+]_3 = [v_0^-]_3$ on $\Sigma$. We shortly write 
\begin{align*}
    [v_0]_3:= [v_0^{\pm}]_3 \qquad \mbox{ on } \Sigma.
\end{align*}
By linearity, we immediately obtain the following representations for $v_0^M$ and $p_1^M$ for almost every $(t,\x,y) \in (0,T)\times \Sigma \times Z_f$:
\begin{align}
\begin{aligned}\label{id:representation_v0M_p1M}
 v_0^M(t,\x,y) &= \sum_{\pm} \sum_{i=1}^2 [v_0^{\pm}]_i(t,\x,0)  q_i^{\pm}(y) + [v_0]_3 q_3(y),
    \\
    p_1^M(t,\x,y) &= \sum_{\pm}\sum_{i=1}^2 [v_0^{\pm}]_i(t,\x,0) \pi_i^{\pm}(y) + [v_0]_3 \pi_3(y),
\end{aligned}
\end{align}
where here $v_0^{\pm}(t,\x,0)$ denotes the trace of $v_0^{\pm}$ on $\Sigma$, and for $i=1,2$ the tuple $(q_i^{\pm},\pi_i^{\pm}) \in H^1_{\#}(Z_f,\Gamma)^3 \times L_0^2(Z_f)$ is the unique weak solution of the cell problem
\begin{align}
\begin{aligned}\label{Cell_problems_fluid_12}
    -\nabla_y \cdot D_y(q_i^{\pm}) +\nabla_y \pi_i^{\pm} &= 0 &\mbox{ in }& Z_f,
    \\
    \nabla_y \cdot q_i^{\pm} &= 0 &\mbox{ in }& Z_f,
    \\
    q_i^{\pm} &= 0 &\mbox{ on }& \Gamma \cup S^{\mp},
    \\
    q_i^{\pm} &= e_i &\mbox{ on }& S^{\pm},
    \\
    q_i^{\pm},\, \pi_i^{\pm} &\mbox{ is } Y\mbox{-periodic},
\end{aligned}
\end{align}
and $(q_3,\pi_3) \in H^1_{\#}(Z_f,\Gamma)^3 \times L_0^2(Z_f)$  is the unique solution of  
\begin{align}
\begin{aligned}\label{Cell_problems_fluid_3}
    -\nabla_y \cdot D_y(q_3) +\nabla_y \pi_3 &= 0 &\mbox{ in }& Z_f,
    \\
    \nabla_y \cdot q_3 &= 0 &\mbox{ in }& Z_f,
    \\
    q_3 &= 0 &\mbox{ on }& \Gamma ,
    \\
    q_3 &= e_3 &\mbox{ on }& S^+ \cup S^-,
    \\
    q_3,\, \pi_3 &\mbox{ is } Y\mbox{-periodic}.
\end{aligned}
\end{align}
Here we have to point out the difference between the cell problems for $(q_i^{\pm},\pi_i^{\pm})$ with $i=1,2$ and $(q_3,\pi_3)$. While for the first tuple we have the zero boundary condition on $S^-$ resp. $S^+$, this is not possible for $(q_3,\pi_3)$, since we need $\int_{\partial Z_f} q_3 \cdot \nu dy = 0$.

Next, we derive some kind of permeability coefficients across the interface $\Sigma$, which allows to replace the terms in $\eqref{eq:two_scale_model_fluid}$ including integrals over $Z_f$ by an integral over $\Sigma$.  In fact we will see that these integrals are not depending on the choice of $\phi^M$, but only on its values $\phi^+$ and $\phi^-$ on $S^+$ and $S^-$, respectively. Let $z^{\pm} \in \R^3$ with $z^+_3 = z^-_3 =: z_3$ and choose 
\begin{align}\label{def:aux_phiM}
    \phi^M := \sum_{\alpha \in \{\pm\}} \sum_{j=1}^2 z_j^{\alpha} q_j^{\alpha} + z_3 q_3.
\end{align}
Hence, $\phi^M$ fulfills $\nabla_y \cdot \phi^M = 0$ and $\phi^M = z^{\pm}$ on $S^{\pm}$. Using the representation above for $v_0^M$, we get almost everywhere in $(0,T)\times \Sigma$
\begin{align*}
\int_{Z_f}& D_y(v_0^M) : D_y(\phi^M) - p_1^M \nabla_y \cdot \phi^M dy 
\\
=&\sum_{\alpha,\beta \in \{\pm\}} \sum_{i,j=1}^2 [v_0^{\alpha}]_i z_j^{\beta} \int_{Z_f} D_y(q_i^{\alpha} ) : D_y (q_j^{\beta}) dy + \sum_{\alpha\in\{\pm\}} \sum_{i=1}^2 [v_0^{\alpha}]_i z_3 \int_{Z_f} D_y(q_i^{\alpha}) : D_y(q_3) dy 
\\
&+ \sum_{\beta \in \{\pm\}} \sum_{j=1}^2 [v_0]_3 z_j^{\beta} \int_{Z_f} D_y(q_3) : D_y(q_j^{\beta}) dy + [v_0]_3 z_3 \int_{Z_f} D_y(q_3) : D_y(q_3) dy.
\end{align*}
This gives rise to the definition of the following effective coefficients for $\alpha ,\beta \in \{\pm\}$ and $i,j=1,2$:
\begin{align*}
    G_{ij}^{\alpha \beta} &:= \int_{Z_f} D_y(q_i^{\alpha}) : D_y(q_j^{\beta}) dy ,
    \\
   G_{3i}^{\alpha \alpha }:=  G_{i3}^{\alpha \alpha } &:= \int_{Z_f} D_y(q_i^{\alpha}) : D_y(q_3) dy ,
    \\
    G_{33}^{\alpha \alpha} &:= \frac12 \int_{Z_f} D_y(q_3) : D_y(q_3) dy.
\end{align*}
Obviously, the last integral is not depending on $\alpha$, and also the second is just depending on $\alpha $ and not $\beta$. However, in the following this will simplify the notation. The tensor $G^{\alpha \beta}_{ij}$ is symmetric with respect to $\alpha,\beta$ and $i,j$.
Hence, we obtain 
\begin{align*}
\int_{Z_f}& D_y(v_0^M) : D_y(\phi^M) - p_1^M \nabla_y \cdot \phi^M dy 
\\
=&\sum_{\alpha,\beta \in \{\pm\}} \sum_{i,j=1}^2 G_{ij}^{\alpha \beta} [v_0^{\alpha}]_i z_j^\beta + \sum_{\alpha \in \{\pm\}}\sum_{i=1}^2 G_{i3}^{\alpha \alpha} \left\{ [v_0^{\alpha}]_i z_3 + [v_0]_3 z_i^{\alpha}\right\} + G_{33}^{\alpha \alpha} z_3 [v_0]_3.
\end{align*}
We define the tensors $K^{\alpha} \in \R^{3\times 3}$ for $\alpha \in \{\pm\}$ and $M\in \R^{3\times 3}$ by ($i,j=1,2,3$)
\begin{align}\label{def:Kpm}
    K^{\alpha}_{ij}:= G^{\alpha \alpha}_{ij}
\end{align}
and 
\begin{align}\label{def:M}
    M_{ij}:= G^{+-}_{ij} \quad \mbox{ for } i,j\in \{1,2\},\quad M_{i3} = M_{3i} = 0 \quad \mbox{ for } i=1,2,3.
\end{align}
Since $G^{+-}_{ij} = G^{-+}_{ij}$ the matrix $M$ is not depending on $\pm$.  Altogether, we obtain
\begin{align}\label{eq:aux_membrane}
\int_{Z_f}& D_y(v_0^M) : D_y(\phi^M) - p_1^M \nabla_y \cdot \phi^M dy = \sum_{\pm} K^{\pm} v_0^{\pm} \cdot z^{\pm} + Mv_0^+  \cdot z^- + M v_0^- \cdot z^+.
\end{align}
Plugging in this identity into $\eqref{eq:two_scale_model_fluid}$ we obtain that $(v_0,p_0)$ is a weak solution of the macroscopic fluid problem $\eqref{def:Macro_Stokes_model_strong}$. We emphasize that in $\eqref{eq:two_scale_model_fluid}$ we have to choose test-functions $(\phi^+,\phi^M,\phi^-) \in H$, even if the equation is not depending anymore on $\phi^M$. However, using for example the construction in $\eqref{def:aux_phiM}$, we obtain for every $(\phi^+,\phi^-) \in H^1_{\#}(\Omega^+)^3 \times H^1(\Omega^-)^3_{\#}$ with $\phi^+_3 = \phi^-_3$ on $\Sigma$ an element $\phi^M$, such that $(\phi^+,\phi^M,\phi^-)\in H$.

Next, we  prove that $(v_0,p_0)$ is the unique solution of the macroscopic fluid model. We show the the effective coefficients $K^{\pm}$ and $M$ fulfill the following coercivity property:
\begin{lemma}
There exists a constant $c_0>0$ such that for every $\xi^{\pm} \in \R^3$ with $\xi_3^+ = \xi_3^-$ it holds that 
\begin{align}\label{ineq:coercivity_macro_fluid}
M\xi^+ \cdot \xi^- + M \xi^- \cdot \xi^+ + \sum_{\pm} K^{\pm} \xi^{\pm} \cdot \xi^{\pm} \geq c_0 \sum_{\pm} |\xi^{\pm}|^2.
\end{align}
\end{lemma}
\begin{proof}
Let $\xi^{\pm} \in \R^3$ with $\xi_3^+ = \xi_3^- =: \xi_3$ and define
\begin{align}\label{def:aux_coercivity_macro_fluid}
    \xi^M:= \sum_{\pm} \sum_{i=1}^2 \xi^{\pm}_i q_i^{\pm} + \xi_3 q_3.
\end{align}
By an elemental calculation similar to the derivation of $\eqref{eq:aux_membrane}$ we obtain 
\begin{align}\label{eq:aux_coercivity_macro_fluid}
    M\xi^+ \cdot \xi^- + M \xi^- \cdot \xi^+ + \sum_{\pm} K^{\pm} \xi^{\pm} \cdot \xi^{\pm}  = \|D_y(\xi^M)\|_{L^2(Z_f)}^2 \geq 0.
\end{align}
Now, we assume that $\eqref{ineq:coercivity_macro_fluid}$ is not valid. Hence, there exists a sequence $(\xi_n^{\pm})_{n\in \N} \subset \R^3$ (here the index $n$ is not the component) with $[\xi_n^+]_3 = [\xi_n^-]^3=: [\xi_n]_3$ with $|\xi_n^+|^2 + |\xi_n^-|^2 = 1$ and 
\begin{align*}
    M\xi_n^+ \cdot \xi_n^- + M \xi_n^- \cdot \xi_n^+ + \sum_{\pm} K^{\pm} \xi_n^{\pm} \cdot \xi_n^{\pm} < \frac{1}{n}.
\end{align*}
We define $\xi_n^M \in H^1(Z_f,\Gamma)^3$ as in $\eqref{def:aux_coercivity_macro_fluid}$, and obtain from $\eqref{eq:aux_coercivity_macro_fluid}$  and the Korn inequality that 
\begin{align*}
  \|\xi^M_n\|_{H^1(Z_f)} \le C   \|D_y(\xi^M_n)\|_{L^2(Z_f)}^2 < \frac{C}{n}.
\end{align*}
Hence, up to a subsequence we get $\xi^M_n \rightarrow 0$ in $H^1(Z_f)$. From the trace inequaliy we obtain
\begin{align*}
    1 = |\xi_n^+|^2 + |\xi_n^-|^2 = \sum_{\pm} \|\xi_n^M\|_{L^2(S^{\pm})}^2 \le C \|\xi_n^M\|^2_{H^1(Z_f)} \overset{n \to \infty }{\longrightarrow} 0,
\end{align*}
leading to a contradiction and we obtain the desired result.
\end{proof}
As a direct consequence we obtain the following result:
\begin{corollary}
    The solution of $\eqref{eq:var_macro_fluid}$ is unique. In particular, all the convergence results for $\veps$ are valid for the whole sequence.
\end{corollary}

It remains to show the characterization for the Darcy-velocity $\bar{v}_0^M$ in $\eqref{Darcy_velocity_main_result}$ which was defined via
\begin{align*}
    \bar{v}_0^M(t,\x) := \frac{1}{|Z_f|} \int_{Z_f} v_0^M dy.
\end{align*}
Now, we use the representation of $v_0^M$ from $\eqref{id:representation_v0M_p1M}$. We define $q_3^{\pm}:= \frac12 q_3$ to obtain  for $k =1,2,3$
\begin{align*}
    \left[\bar{v}_0^M\right]_k = \sum_{\pm} \sum_{j=1}^3 [v_0^{\pm}]_j  \frac{1}{|Z_f|} \int_{Z_f} [q_j^{\pm}]_k dy.
\end{align*}
This gives rise for the definition for $A^{\pm}\in \R^{3\times 3}$ for $j,k=1,2,3$
\begin{align*}
A_{kj}^{\pm}:= \frac{1}{|Z_f|} \int_{Z_f} [q_j^{\pm}]_k dy. 
\end{align*}
For the last row of $A^{\pm}$ we obtain using $\nabla_y \cdot q_j^{\pm} = 0$ and $[q_j^{\pm}]_3 = \frac12 \delta_{j3}$ on $S^{\pm}$ ($\nu^M = \pm e_3$ on $S^{\pm}$)
\begin{align*}
  |Z_f|  A_{3j}^{\pm}  = \int_{Z_f} [q_j^{\pm}] \cdot e_3 dy = \int_{Z_f} [q_j^{\pm}] \cdot \nabla y_3 dy = \sum_{\pm} \pm  \int_{S^{\pm}} [q_j^{\pm}]_3 y_3 d\sigma = \frac12 \delta_{j3}  \left(|S^+| + |S^-|\right) = \delta_{j3}.
\end{align*}
Now, we define $Q^{\pm} \in \R^{3\times 3}$ by 
\begin{align}\label{MatrixQbar}
    Q^{\pm}:= A^{\pm}  -  \frac{1}{|Z_f|} e_3 \otimes e_3.
\end{align}
In particular, the last row of $Q^{\pm}$ is zero and we obtain
\begin{align}\label{DarcyVel_aux}
    \bar{v}_0^M = \sum_{\pm} Q^{\pm}v_0^{\pm}|_{\Sigma} + \frac{|Z|}{|Z_f|}[v_0]_3 e_3.
\end{align}
We emphasize that $|Z|= 2$, but here we wrote $|Z|$ to illustrate that the factor in the second term includes the ratio between the whole reference cell $Z$ and its fluid part.

\begin{remark}
\begin{enumerate}[label = (\roman*)]
\item We point out the following properties of the representation for the Darcy-velocity in $\eqref{DarcyVel_aux}$: The third component in the first term is equal to zero, since the last row of $Q^{\pm}$ vanishes. Hence, the first term describes the tangential velocity at $\Sigma$, and the second term gives the vertical velocity (the flux) through the interface $\Sigma$.

\item  For the right-hand side in the tangential stress boundary condition $\eqref{Macro_Model_Fluid_Stress_tangential}$ we have due to the properties of  $M$ (all components in third row and column are zero)
\begin{align*}
    Mv_0^{\pm} \cdot \nu^{\mp} = 0.
\end{align*}
This means that $Mv_0^{\pm}$ is in fact tangential on $\Sigma$.
\end{enumerate}
\end{remark}

\subsection{The macroscopic transport equations}

To finish the proof of Theorem \ref{Main_theorem_transport}, we have to derive the macroscopic equations, since the convergence results were already established in Section \ref{sec:compactness_concentrations_fluid} and \ref{sec:conv_ceps_s}. So, we pass to the limit $\vareps \to 0$ in the transport equations for $\ceps^f$ and $\ceps^s$. \\

\noindent \textit{Proof of Theorem \ref{Main_theorem_transport}}:
We start with the equation for $\ceps^f$. The only dependence on the parameter $\gamma$ comes from the boundary term and the limit function $c_0^s$ obtained in Section \ref{sec:conv_ceps_s} for $\gamma = -1$, $\gamma \in (-1,1)$, and $\gamma =1$. However, for the derivation of the macroscopic equation for $c_0^f$ the specific form of $c_0^s$ has no influence, so we can treat all the cases simultaneously. In the variational equation $\eqref{eq_Var_Micro_cepsf}$ we choose test functions $\xieps^f:= \xi^f \in C_0^{\infty}([0,T),C_{\#}^{\infty}(\overline{\Omega}))$ and obtain after integration with respect to time
\begin{align}
\begin{aligned}\label{eq:var_micro_deriv_c0f}
   \int_0^T &\langle \partial_t \ceps^f , \xi^f \rangle_{H^1(\oef)} dt  + \sum_{\pm} \int_0^T \int_{\oeps^{\pm}} [D^f \nabla \ceps^f - \veps \ceps^f] \cdot \nabla \xi^f dx dt 
   \\
   &+ \int_0^T \int_{\oemf} [D^f \nabla \ceps^f - \veps^M \ceps^f]\cdot \nabla \xi^f dx dt = -\int_0^T \int_{\geps} h(\ceps^f,\ceps^s)\xi^f d\sigma dt.
\end{aligned}
\end{align}
Using the compactness results for the time derivative in Proposition \ref{prop:compactness_fluid} we obtain
\begin{align}\label{conv:aux_deriv_macro_c0f}
     \int_0^T \langle \partial_t \ceps^f , \xi^f \rangle_{H^1(\oef)} dt = \int_0^T \langle \partial_t (\chi_{\oef} \ceps^f) , \xi^f \rangle_{H^1(\Omega)} dt \overset{\vareps \to 0}{\longrightarrow} \int_0^T \langle \partial_t c_0^f , \xi^f \rangle_{H^1(\Omega)} dt.
\end{align}
Further, using the \textit{a priori} estimates from Proposition \ref{prop:existence_apriori}, we obtain for the integral in the layer
\begin{align*}
\bigg| \int_0^T &\int_{\oemf} [D^f \nabla \ceps^f - \veps^M \ceps^f]\cdot \nabla \xi^f dx dt \bigg|
\\
&\le C \sqrt{\vareps} \|\nabla \ceps^f \|_{L^2((0,T)\times \oemf)} + C \| \veps^M\|_{L^2((0,T)\times \oemf)} \|\ceps^f\|_{L^2((0,T)\times \oemf)} 
\\
&\le C (\sqrt{\vareps} + \vareps ), 
\end{align*}
and therefore this term vanishes for $\vareps \to 0$. Now, from the convergence results in Proposition \ref{prop:compactness_fluid} and \ref{prop:convergence_cepsf} and Corollary \ref{cor:ts_conv_nonlinear_term} we obtain for $\vareps \to 0$ almost everywhere in $(0,T)$ and (by density) for all $\xi^f \in H_{\#}^1(\Omega)$ (here we use the notation $v_0 =v_0^{\pm} $ in $\Omega^{\pm}$)
\begin{align*}
    \langle \partial_t c_0^f , \xi^f \rangle_{H^1(\Omega)} +  \int_{\Omega} [D^f \nabla c_0^f - v_0 c_0^f] \cdot \nabla \xi^f dx = -\int_{\Sigma} \int_{\Gamma} h(c_0^f,c_0^s) d\sigma_y \xi^f d\x.
\end{align*}
We emphasize that for  $\gamma \in [-1,1)$ the limit function $c_0^s$ is independent of $y$ and we can write for the last integral
\begin{align*}
    -\int_{\Sigma} \int_{\Gamma} h(c_0^f,c_0^s) d\sigma_y \xi^f d\x = -|\Gamma|\int_{\Sigma} h(c_0^f,c_0^s) \xi^f d\x.
\end{align*}
Finally, let us derive the initial condition for $c_0^f$. Even if it is standard, for the sake of completeness, we sketch the procedure. 
Using $\eqref{conv:aux_deriv_macro_c0f}$, integration by parts in time, and the assumption \ref{ass:initial_values_ceps_f}, we get
\begin{align*}
\int_0^T \langle \partial_t c_0^f , \xi^f \rangle_{H^1(\Omega)} dt &= \lim_{\vareps \to 0}\int_0^T \langle\partial_t \ceps^f , \xi^f\rangle_{H^1(\oef)} dt 
\\
&= \lim_{\vareps \to 0}\left\{ -\int_0^T \int_{\oef} \ceps^f \partial_t \xi^f dx dt - \int_{\oef} c_{\vareps,\mathrm{in}}^f \xi^f(0,x) dx \right\}
\\
&= -\int_0^T \int_{\Omega} c_0^f \partial_t \xi^f dx dt - \int_{\Omega} c_{\mathrm{in}}^f \xi^f(0,x) dx.
\end{align*}
Integration by parts gives $c_0^f(0) = c_{\mathrm{in}}^f$ in $L^2(\Omega)$. Altogether, we showed that $c_0^f$ is a weak solution of the macroscopic equation $\eqref{eq:Macro_c0f_strong}$.

For the transport equation in $\oems$ we have to distinguish the three cases $\gamma = -1$, $\gamma \in (-1,1)$, and $\gamma = 1$. 

\noindent\textit{\underline{The case $\gamma = - 1$}}: As a test-function in the variational equation $\eqref{eq_Var_Micro_cepss}$ we choose 
\begin{align*}
    \xieps^s(t,x):= \xi_0^s(t,\x) + \vareps \xi_1^s\left(t,\x,\fxe\right)
\end{align*}
with  $\xi_0^s \in C_0^{\infty}((0,T),C_{\#}^{\infty}(\Sigma))$ and $\xi_1^s \in C_0^{\infty}((0,T)\times \Sigma, C_{\#}^{\infty}(\overline{Z_s}))$ to obtain almost everywhere in $(0,T)$
\begin{align}
\begin{aligned}\label{eq:micro_deriv_macro_c0s}
\foe &\left\langle \partial_t \ceps^s,  \xi_0^s + \vareps \xi_1^s\left(t,\x,\fxe\right)\right\rangle_{H^1(\oems)} 
\\
&+ \foe \int_{\oems} D^s \nabla \ceps^s \cdot \left[\nabla_{\x} \xi_0^s + \nabla_y \xi_1^s\left(t,\x,\fxe\right) + \vareps \nabla_{\x} \xi_1^s\left(t,\x,\fxe\right) \right] dx
\\
=& \int_{\geps } h(\ceps^f,\ceps^s) \left[ \xi_0^s + \vareps \xi_1^s \left(t,\x,\fxe\right)\right] d\sigma.
\end{aligned}
\end{align}
Using the convergence results from Proposition \ref{prop:strong_TS_convergence_cepss_gamma_-1} and Corollary \ref{cor:ts_conv_nonlinear_term}, we obtain for $\vareps \to 0$ (after an integration with respect to time) almost everywhere in $(0,T)$:
\begin{align*}
|Z_s|\langle \partial_t c_0^s, \xi_0^s\rangle_{H^1(\Sigma)} + \int_{\Sigma} \int_{Z_s} D^s \left[\nabla_{\x} c_0^s + \nabla_y c_1^s\right] \cdot \left[\nabla_{\x} \xi_0^s + \nabla_y \xi_1^s \right] dy d\x 
= |\Gamma| \int_{\Sigma}  h(c_0^f,c_0^s) \xi_0^s d\sigma.
\end{align*}
By density this equation is valid for all $\xi_0^s \in H^1_{\#}(\Sigma)$ and $\xi_1^s \in L^2(\Sigma, H_{\#}^1(Z_s))$. Choosing first $\xi_0^s=0$ we obtain that $c_1^s$ solves almost everywhere in $\Sigma$
\begin{align*}
    \int_{Z_s} D^s[\nabla_{\x} c_0^s + \nabla_y c_1^s] \cdot \nabla_y \xi_1^s dy = 0.
\end{align*}
Hence, by standard arguments we obtain for almost every $(t,\x,y) \in (0,T)\times \Sigma \times Z_f$
\begin{align*}
    c_1^s(t,\x,y)  = \sum_{i=1}^2 \partial_{x_i} c_0^s(t,\x) \eta_i(y) 
\end{align*}
with $\eta_i \in H^1(Z_s)/\R$ the unique weak solution of the cell problem 
\begin{align}
\begin{aligned}\label{cell_problem_c_1_s}
-\nabla_y \cdot \left( D^s[e_i + \nabla_y \eta_i ]\right) &= 0 &\mbox{ in }& Z_s,
\\
-D^s [e_i + \nabla_y \eta_i]\cdot \nu &= 0 &\mbox{ on }& \Gamma,
\\
\eta_i \,\, Y\mbox{-periodic}, \,\,\, \int_{Z_s} \eta_i dy = 0.
\end{aligned}
\end{align}
Next, we choose in the equation above $\xi_1^s= 0$ and obtain after an elemental calculation
\begin{align}\label{eq:var_equation_c0s}
 |Z_s|\langle \partial_t c_0^s, \xi_0^s\rangle_{H^1(\Sigma)} + \int_{\Sigma} D_{\ast}^s \nabla_{\x} c_0^s \cdot \nabla_{\x} \xi_0^s d\x = |\Gamma|\int_{\Sigma} h(c_0^f,c_0^s) \xi_0^s d\sigma
\end{align}
with the homogenized diffusion coefficient $D^s \in \R^{2\times 2}$ defined for $i,j=1,2$ by 
\begin{align}\label{hom_diff_coefficient_Ds}
(D^s_{\ast})_{ij}:= \int_{Z_s} D^s [e_i + \nabla_y \eta_i ] \cdot [e_j + \nabla_y \eta_j] dy.
\end{align}
Finally, using Corollary \ref{cor:initial_conditions_ts_limit} and  assumption \ref{ass:initial_values_ceps_s} we obtain for $c_0^s$ the initial condition $c_0^s(0) = \bar{c}_{\mathrm{in}}^s$.  Altogether we showed that $c_0^s$ is a weak solution of the macroscopic problem \ref{eq:Macro_c0s_gamma-1_strong}.

\noindent\textit{\underline{The case $\gamma \in (-1,1)$}}: Again the limit function $c_0^s$ obtained in Proposition \ref{prop:strong_ts_conv_gamma_-1_1} is independent of the microscopic variable $y \in Z_s$. However, compared to the case $\gamma = -1$ we loose spatial regularity and only have $L^2$-regularity for $c_0^s$ with respect to $\x$. As a test function in the variation equation $\eqref{eq_Var_Micro_cepss}$ we now choose simply $\xieps^s(t,x):= \xi_0^s \in C_0^{\infty}((0,T),C_{\#}^{\infty}(\Sigma))$ and argue in the same way as in the case $\gamma = -1$. We only emphasize that the flux term including $\nabla \ceps^s$ is now of order $\vareps^{\frac{1 + \gamma}{2}}$ and therefore vanishes for $\vareps \to 0$. We obtain in the limit almost everywhere in $(0,T)$ and for all $\xi_0^s \in H_{\#}^1(\Sigma)$
\begin{align*}
   |Z_s| \langle \partial_t c_0^s, \xi_0^s\rangle_{H^1(\Sigma)} = |\Gamma| \int_{\Sigma} h(c_0^f,c_0^s) \xi_0^s d\sigma
\end{align*}
Using again Corollary \ref{cor:initial_conditions_ts_limit}, we get $c_0^s(0) = \bar{c}_{\mathrm{in}}^s$.
Hence, $c_0^s$ is a weak solution of the macroscopic equation $\eqref{eq:Macro_c0s_gamma-11_strong}$.

\noindent\textit{\underline{The case $\gamma = 1$}}: Finally, for $\gamma = 1$ we choose in the variational equation $\eqref{eq_Var_Micro_cepss}$ test-functions $\xieps^s(t,x) = \xi_1^s\left(t,\x,\fxe\right)$ with $\xi_1^s \in C_0^{\infty}((0,T)\times \Sigma,C_{\#}^{\infty}(\overline{Z_s}))$ and obtain for $\vareps\to 0$ using Proposition \ref{prop:strong_ts_conv_gamma_1} and Corollary \ref{cor:ts_conv_nonlinear_term} almost everywhere in $(0,T)$
\begin{align*}
    \langle \partial_t c_0^s,\xi_1^s\rangle_{L^2(\Sigma,H_{\#}^1(Z_s))} + \int_{\Sigma} \int_{Z_s} D^s \nabla_y c_0^s \cdot \nabla_y \xi_1^s dyd\x = \int_{\Sigma} \int_{\Gamma} h(c_0^f,c_0^s) \xi_1^s d\sigma d\x.
\end{align*}
By density, this equation is also valid for all $\xi_1^s \in L^2(\Sigma,H_{\#}^1(Z_s))$. Further, by Corollary \ref{cor:initial_conditions_ts_limit} we have the initial condition $c_0^s(0) = c_{\mathrm{in}}^s$. In other words, $c_0^s$ is a weak solution of the macroscopic problem $\eqref{eq:Macro_c0s_gamma1_strong}$.

\begin{remark}
We emphasize that for the derivation of the initial condition for the limit function $c_0^s$ we only used the compactness results in Lemma \ref{lem:two_scale_conv_time_derivative} and Corollary \ref{cor:initial_conditions_ts_limit}, but not the microscopic equation and the macroscopic equation. The latter is often used in the literature, when the time-derivative of the microscopic solution is only a functional but not an element in $L^2$. Of course, this idea is also valid in our case.
\end{remark}

It remains to show the uniqueness of the whole macroscopic model.  For the fluid model (which is not influenced by the concentration of the solutes) we already have uniqueness. Hence, the problem of uniqueness reduces to the question of uniqueness for a reaction-diffusion-advection equation with nonlinear right-hand sides. For the pure diffusive case we refer to \cite{GahnNeussRaduKnabner2018a}, where similar problems where considered without advection and a slightly differnt form of the nonlinearity. The treatment of the advective term is classical, so we obtain the desired result and the whole Theorem \ref{Main_theorem_transport} is proved.

\subsection{The case $|S_s^{\pm}| \neq 0$}
\label{sec:Case_Sspm_not_empty}

Finally, we investigate the case $|S_s^{\pm}| \neq 0$ (valid for both $S^+_s$ and $S_s^-$) when the solid phase touches the bulk domains. First of all, we have to slightly modify the microscopic problem. Now, the interface conditions between the bulk domains and the thin layer for the fluid problem in $\eqref{ContinuityVelocity}$ and $\eqref{ContinuityNormalStress}$ are valid on $S_{\vareps,f}^{\pm}$ instead of $S_{\vareps}^{\pm}$ (see Section \ref{sec:micro_geometry} for the definition). Additionally, we need an interface condition for the transport equation of $\ceps^s$ between the solid phase and the bulk regions, where we assume again the nonlinear flux condition
\begin{align*}
-(D^f \nabla \ceps^f - \veps \ceps^f) \cdot \nu = - \vareps^{\gamma} D^s \nabla \ceps^s \cdot \nu &= h(\ceps^f,\ceps^s) &\mbox{ on }& (0,T)\times S_{\vareps,s}^{\pm}.
\end{align*}
The weak formulation remains the same as in the previous case and is given in Definition \ref{def:Weak_Micro_Model}. Hence, the \textit{a priori} estimates in Proposition \ref{prop:existence_apriori} and the estimates for the differences of the shifts in Lemma \ref{lem:apriori_shifts} are still valid. We emphasize that all the results used for the derivation of the \textit{a priori} estimates in the appendix are still valid for $|S_s^{\pm}|\neq 0$.

The crucial difference now occurs in Proposition \ref{prop:compactness_fluid}, where we immediately obtain $v_0^{\pm} = 0$ on $(0,T)\times \Sigma$ and therefore the interface $\Sigma$ acts as an impermeable barrier to the fluid flow, see Remark \ref{rem:fluid_vel_zero_Sigma}.
Now, we can follow the arguments in Section \ref{sec:derivation_fluid_model} and obtain in Proposition \ref{prop:two_scale_model_fluid} the interface conditions
\begin{align*}
    v_0^{\pm} \quad \mbox{on } (0,T)\times \Sigma, \qquad v_0^M = 0\quad \mbox{on } (0,T)\times \Sigma \times S^{\pm}_f.
\end{align*}
In particular, the tuple $(v_0^M,p_1^M)$ solves the cell problem
\begin{align*}
-\nabla_y \cdot D_y(v_0^M) + \nabla_y p_1^M &= 0 &\mbox{ in }& (0,T)\times \Sigma \times Z_f,
\\
\nabla_y \cdot v_0^M &= 0 &\mbox{ in }& (0,T)\times \Sigma \times Z_f,
\\
v_0^M &= 0 &\mbox{ on }& (0,T)\times \Sigma \times (\Gamma \cup S_f^{\pm}),
\\
v_0^M,\, p_1^M \,\, Y\mbox{-periodic}.
\end{align*}
This problem only has the trivial solution $v_0^M = 0$ and $p_1^M = 0$. Summarizing, we obtain that the macroscopic fluid velocity $v_0 =(v_0^+,v_0^-)$ (and $v_0^M =0$) solves the problem 
\begin{align*}
    \partial_t v_0^{\pm} - \nabla\cdot D(v_0^{\pm}) + \nabla p_0^{\pm} &= f_0^{\pm} &\mbox{ in }& (0,T)\times \Omega^{\pm},
\\
\nabla \cdot v_0^{\pm} &= 0 &\mbox{ in }& (0,T)\times \Omega^{\pm},
\\
v_0^{\pm} &= 0 &\mbox{ on }& (0,T)\times \Sigma,
\\
v_0^{\pm}(0) &= 0 &\mbox{ in }& \Omega^{\pm},
\\
v_0^{\pm} \,\, \Sigma\mbox{-periodic},
\end{align*}
Hence, in the case $|S_s^{\pm}|\neq 0$, in the limit $\vareps \to 0$, the fluid flow in the two bulk domains $\Omega^{\pm}$ is modelled by two decoupled Stokes systems.

Now, we have to consider the transport equations, where we will see that the macroscopic transport problem is still given by  $\eqref{eq:Macro_c0f_strong}$ and (\ref{Macro_Model_Transport}$\ast$) with $\ast \in\{\mbox{b,c,d}\}$ for the different choices of $\gamma$. This can be shown by following the  arguments from  Section \ref{sec:comptness_results} and \ref{sec:derivation_macro_model}. However, we can also use the following simple trick. Replace the reference element $Z$ by the bigger cell $\tZ := Y \times (-2,2)$ and define $\tZ_f:= \mathrm{int} \left(\tZ \setminus Z_s\right)$. The associated thin perforated layer is denoted by $\widetilde{\Omega}_{\vareps}^{M,f}$. Since the transport equation for $\ceps^f$ has the same scaling in the bulk domains and in the thin layer, we can just replace $\oemf$ by $\widetilde{\Omega}_{\vareps}^{M,f}$ (respectively $\oeps^{\pm}$ by $\Omega_{2\vareps}^{\pm}$ for the convergence results of the fluid velocity and pressure in the bulk domains, see Proposition \ref{prop:compactness_fluid}) and obtain the same limit problem $\eqref{Macro_Model_Transport}$ and (\ref{Macro_Model_Transport}$\ast$) with $\ast \in\{\mbox{b,c,d}\}$ for the different choices of $\gamma$. Of course, the volume of $\tZ_f$ is different from the volume of $Z_f$, however, this has no influence since in the macroscopic equation $\eqref{Macro_Model_Transport}$ this quantity does not occur.

\begin{remark}
The situation changes if we consider the case $|S_s^+| = 0$ and $|S_s^-|\neq 0$ (only the bottom of the solid phase touches the bulk domain).  In this case, the membrane acts as a rough boundary and we have $[v_0^+]_3 = [v_0^-]_3  = 0$. In this case, the vertical component of the Darcy-velocity $\bar{v}_0^M$ vanishes. The proof follows similar ideas as above. The same holds if we interchange $+$ and $-$.
\end{remark}

\begin{appendix}
\normalsize
\section{Auxiliary results and inequalities}

In this section we give some auxiliary results and estimates which are necessary for our homogenization and dimension reduction. Some of them are standard, and we refer to the literature. However, several results seem to be new and are of particular importance on their own, since they allow to treat more complex applications.

In the following, we consider arbitrary dimensions $n\geq 3$. This means that we can use all the definitions from Section \ref{sec:micro_geometry} by just replacing the dimension $3$ by $n$. We emphasize that the results which are only formulated for $\oef$, $\oemf$, and $\oems$ are also valid for $n=2$, as long as all these domains are connected. 
We emphasize that in the whole appendix we consider the general case that $|S_s^{\pm}|>0$, meaning that the solid part of the thin perforated layer touches the bulk regions.

\begin{lemma}[Korn and Poincar\'e inequality]\label{lem:Korn_inequality}
There exists a constant $C>0$, such that for all $\ueps \in H^1(\oemf)^n$ with $\ueps = 0$ on $\geps$ it holds that
\begin{align*}
    \|\ueps\|_{L^2(\oemf)} \le C\vareps \|\nabla \ueps \|_{L^2(\oemf)} \le C\vareps \|D(\ueps)\|_{L^2(\oemf)}.
\end{align*}
\end{lemma}

\begin{lemma}[Standard scaled trace inequality]\label{lem:Stand_scaled_trace_inequality}
For all $\theta >0$ there exists $C_\theta>0$  independent of $\vareps$, such that for all $\phieps \in H^1(\Omega_\vareps^{M,\ast}), \ast \in \{s,f\}$, it holds that
\begin{align*}
    \|\phieps\|_{L^2(\geps)} \le \frac{C_\theta}{\sqrt{\vareps}} \|\phieps\|_{L^2(\oe^{M,\ast})}
    + \theta \sqrt{\vareps}\|\nabla \phieps \|_{L^2(\oe^{M,\ast})}.
\end{align*}

\end{lemma}

In the following lemma we estimate the scaled $L^2$-norm in the fluid part of the membrane by the $H^1$-norm in the whole fluid domain. This allows to obtain better estimates (with respect to $\vareps$) in the membrane. Later, we will improve this estimate by replacing $H^1$ with $H^\beta$, which allows us to transfer strong convergence results from the bulk domains to the thin layer. Even if the result is quite simple and easy to prove, it is crucial to obtain the necessary \textit{a priori} estimates for the microscopic solution. It also builds the basis for several further estimates obtained below, for example the trace inequality on $\geps$ with respect to the bulk domains (see Lemma \ref{lem:Trace_inequality_Bulk}).
\begin{lemma}\label{lem:estimate_L2_membrane_Bulk_Gradient}
For every $\phieps \in H^1(\oemf)$ it holds that
\begin{align}\label{ineq:aux_trace_inequality}
        \frac{1}{\sqrt{\vareps}} \|\phieps\|_{L^2(\oemf)} \le C \left( \sqrt{\vareps} \|\nabla \phieps\|_{L^2(\oemf)} + \|\phieps\|_{L^2(S_{\vareps,f}^{\pm})}\right)
    \end{align}
In particular, for every $\phieps \in H^1(\oef)$ we obtain
\begin{align*}
\frac{1}{\sqrt{\vareps}} \|\phieps\|_{L^2(\oemf)} \le C \|\phieps\|_{H^1(\oef)}.
\end{align*}
\end{lemma}
\begin{proof}
A similar estimate as $\eqref{ineq:aux_trace_inequality}$ was shown in \cite{conca1985application} and we use the same ideas.
   Due to the compactness of $H^1(Z_f)$ in $L^2(Z_f)$ we have the following inequality (see also the proof of Lemma \ref{lem:estimate_L2_norm_membrane_Hbeta} below for some more details)
    \begin{align}\label{ineq:bound_gradient_surface}
        \|\phi\|_{L^2(Z_f)}  \le C \left( \|\nabla \phi\|_{L^2(Z_f)}  + \|\phi \|_{L^2(S_f^{\pm})}\right)
    \end{align}
    for every $\phi \in H^1(Z_f)$. Now, $\eqref{ineq:aux_trace_inequality}$ follows by a standard scaling argument. Obviously, we have
    \begin{align}
     \|\phieps\|_{L^2(S_{\vareps,f}^{\pm})} \le C \|\phieps\|_{H^1(\oeps^{\pm})}
    \end{align}
    for a constant $C>0$ independent of $\vareps$, which finishes the proof.

\end{proof}
In the following lemma we improve the inequality from the previous lemma by considering the $H^\beta$-norm for $\beta \in (\frac12 ,1)$ instead of the $H^1$-norm. We define for an open set $U\subset \R^n$ the $H^\beta$-seminorm by
\begin{align*}
    |u|_{\beta,U}^2:= \int_U \int_U \frac{|u(x) - u(y)|^2}{|x-y|^{n + 2\beta}} dy dx.
\end{align*}
\begin{lemma}\label{lem:estimate_L2_norm_membrane_Hbeta}
Let $\beta \in (\frac12,1)$. Then, for every $\phieps \in H^{\beta}(\oef)$ it holds that
\begin{align*}
    \frac{1}{\sqrt{\vareps}} \|\phieps\|_{L^2(\oemf)} \le C \|\phieps\|_{H^{\beta}(\oef)}.
\end{align*}
This inequality is also valid if we replace $Z_f$ resp. $\oemf$  by $Z$ resp. $\oeps^M$, and $\oef$ by $\Omega$ 
\end{lemma}
\begin{proof}
We argue in the same way as in the proof of Lemma \ref{lem:estimate_L2_membrane_Bulk_Gradient}, where we just replace the semi-norm $\|\nabla u\|_{L^2}$  by the $H^{\beta}$-semi-norm. First of all, we notice that for every $\phi \in H^{\beta}(Z_f)$ it holds that
\begin{align}\label{ineq:aux_estimate_L2_norm_Hbeta}
    \|\phi\|_{L^2(Z_f)} \le C \left(|\phi|_{\beta,Z_f} + \|\phi\|_{L^2(S_f^{\pm})} \right).
\end{align}
This follows by a standard contradiction argument and the compact embedding $H^{\beta}(Z_f) \hookrightarrow L^2(Z_f)$. For the sake of completeness we give some details. If the inequality is not valid, we find a sequence $\phi_n \in H^{\beta}(Z_f)$ such that
\begin{align*}
    1 = \|\phi_n\|_{L^2(Z_f)} > n \left( |\phi_n|_{\beta,Z_f} + \|\phi_n\|_{L^2(S_f^{\pm})} \right).
\end{align*}
Hence, we obtain a subsequence and a function $\phi \in H^{\beta}(Z_f)$, such that $\phi_n \rightharpoonup \phi $ in $H^{\beta}(Z_f)$ and $\phi_n \rightarrow \phi$ in $L^2(Z_f).$ Further, it holds that $|\phi|_{\beta,Z_f} = 0$ and $\|\phi\|_{L^2(S_f^{\pm})} = 0$. The first identity gives $\phi$ is constant, and the second equality that $\phi = 0, $ what gives us a contradiction. 

Now, we consider $\phieps\in H^{\beta}(\oef)$ and obtain by a standard scaling argument  and inequality $\eqref{ineq:aux_estimate_L2_norm_Hbeta}$ (see also the proof of \cite[Lemma A.6]{GahnDissertation} for more details)
\begin{align*}
\foe \|\phieps\|_{L^2(\oemf)}^2 &= \sum_{k \in K_{\vareps}} \vareps^{n-1} \int_{Z_f} |\phi_{\vareps}(\vareps(y+ k))|^2 dy 
\\
&\le C \vareps^{n-1} \sum_{k \in K_{\vareps}} \left( |\phieps(\vareps (\cdot + k))|_{\beta,Z_f}^2 + \|\phieps(\vareps( \cdot + k))\|_{L^2(S_f^{\pm})} \right)
\\
&= C \vareps^{n-1} \sum_{k\in K_{\vareps}} \left(\vareps^{2\beta - n}|\phieps|_{\beta,\vareps(Z_f + k)}^2 + \vareps^{1- n} \|\phieps\|^2_{L^2(\vareps(S_f^{\pm} + k))} \right)
\\
&\le C \left( \vareps^{2\beta - 1} |\phieps|^2_{\beta,\oemf} + \|\phieps\|_{L^2(S_{\vareps,f}^{\pm})} \right).
\end{align*}
Then the inequality $\|\phieps\|_{L^2(S_{\vareps,f}^{\pm})} \le C \|\phieps\|_{H^{\beta}(\oeps^{\pm})}$ gives the desired result.
\end{proof}

\begin{lemma}[Trace inequality]\label{lem:Trace_inequality_Bulk}
    There exists a constant $C_0>0$ with the following property: For every  $\theta >0$ there exists a constant $ C_{\theta}>0$ independent of $\vareps$,  such that for every $\phieps \in H^1(\oef)$ it holds that
    \begin{align}\label{ineq:lem_trace_inequality_bulk_scaled}
        \|\phieps\|_{L^2(\geps)} \le C_0 \sqrt{\vareps } \|\nabla \phieps\|_{L^2(\oemf)} + C_{\theta} \|\phieps\|_{L^2(\oeps^{\pm})} + \theta \|\nabla \phieps\|_{L^2(\oeps^{\pm})}.
    \end{align}
    In particular it holds that
    \begin{align}\label{ineq:lem_trace_inequality_bulk_H1}
        \|\phieps\|_{L^2(\geps)} \le C \|\phieps\|_{H^1(\oef)}.
    \end{align}

Further, for every $\phieps \in H^{\beta}(\oef)$ with $\beta \in \left(\frac12 , 1\right)$ it holds that 
\begin{align}\label{ineq:lem_trace_inequality_bulk_Hbeta}
\|\phieps\|_{L^2(\geps )} \le C \| \phieps \|_{H^{\beta}(\oef)}.
\end{align}
All the results are valid if we replace $\geps$ by $\geps \cup S_{\vareps,f}^{\pm}$.
\end{lemma}

A similar estimate as $\eqref{ineq:lem_trace_inequality_bulk_H1}$ was shown in \cite[Proposition 2]{donato2019asymptotic} for a rough oscillating interface given as a graph, see also \cite[Lemma 2]{freudenberg2024homogenization} for slight generalizations. Here, we give a result for the case when $\geps$ is only a Lipschitz surface. Further, we improve  these results by using the $H^{\beta}$-norm in the trace inequality, see $\eqref{ineq:lem_trace_inequality_bulk_Hbeta}$. Additionally, we give in $\eqref{ineq:lem_trace_inequality_bulk_scaled}$ a scaled trace inequality (with respect to $\theta>0$), which is important for the treatment of the nonlinear boundary terms in the derivation of the \textit{a priori} estimates.
\begin{proof}
With Lemma \ref{lem:estimate_L2_membrane_Bulk_Gradient} and the scaled trace inequality from Lemma \ref{lem:Stand_scaled_trace_inequality} we obtain for every $\phieps \in H^1(\oef)$ and a constant $C>0$ resp. $C_0>0$ independent of $\vareps$
    \begin{align*}
        \|\phieps\|_{L^2(\geps)} &\le C \left( \frac{1}{\sqrt{\vareps}} \|\phieps\|_{L^2(\oemf)} + \sqrt{\vareps} \|\nabla \phieps \|_{L^2(\oemf)} \right)
        \\
        &\le C_0 \left( \sqrt{\vareps} \|\nabla \phieps \|_{L^2(\oemf)} + \|\phieps\|_{L^2(S_{\vareps,f}^{\pm})} \right)
        \\
        &\le  C_0 \sqrt{\vareps } \|\nabla \phieps\|_{L^2(\oemf)} + C_{\theta} \|\phieps\|_{L^2(\oeps^{\pm})} + \theta \|\nabla \phieps\|_{L^2(\oeps^{\pm})},
    \end{align*}
    where in the last inequality we used the trace estimate (for arbitrary $\theta>0$ and a constant $C_{\theta}>0$ depending on $\theta $ but not $\vareps$)
    \begin{align*}
        \|\phieps\|_{L^2(S_{\vareps,f}^{\pm})} \le C_{\theta} \|\phieps\|_{L^2(\oeps^{\pm})} + \theta \|\nabla \phieps\|_{L^2(\oeps^{\pm})}.
    \end{align*}

The inequality for $\phieps \in H^{\beta}(\oef)$ follows by the same argument by using Lemma \ref{lem:estimate_L2_norm_membrane_Hbeta} and the trace inequality (see \cite[Lemma A.6]{GahnDissertation})
\begin{align*}
\|\phieps\|_{L^2(\geps)} \le C \left(\frac{1}{\sqrt{\vareps}} \|\phieps\|_{L^2(\oemf)} + \vareps^{\beta - \frac12} |\phieps|_{\beta,\oemf}\right).
\end{align*}
For $\geps \cup S_{\vareps,f}^{\pm}$ instead of $\geps$ the proof follows the same lines.
\end{proof}

The existence of extension operators preserving \textit{a priori} bounds of the extended function in Sobolev spaces for periodic domains was studied in detail in \cite{Acerbi1992} (see also \cite{CioranescuSJPaulin} for the case of strict inclusions). These results are mainly based on suitable local extension operators. Therefore, they can be easily transferred to our microscopic geometry. In the following we formulate the necessary extension results for our problem. In contrast to existing literature we also need the periodicity of the extended function, which also follows by the construction of the global extension operator in \cite{Acerbi1992} and \cite{CioranescuSJPaulin}. For the sake of completeness we give the proof.

\begin{lemma}[Extension operators]\label{lem:Extension_operators}
\begin{enumerate}[label = (\roman*)]
\item There exists an extension operator $E_{\vareps}^s: H^1(\oems) \rightarrow H^1(\oeps^M)$ such that for all $\phieps \in H^1(\oems)$ it holds that
\begin{align*}
    \|E_{\vareps}^s \phieps\|_{L^2(\oeps^M)} \le C \|\phieps\|_{L^2(\oems)},\qquad \|\nabla E_{\vareps}^s \phieps \|_{L^2(\oeps^M)} \le C\|\nabla \phieps\|_{L^2(\oems)}, 
\end{align*}
for a constant $C>0$ independent of $\vareps$. The extension operator $E_{\vareps}^s$ preserves the $\Sigma$-periodicity, more precisely, we have $E_{\vareps}^s:H^1_{\#}(\oems) \rightarrow H^1_{\#}(\oem)$. 

\item  There exists an extension operator $E_{\vareps}^f: H^1(\oef) \rightarrow H^1(\Omega)$ such that for all $\phieps \in H^1(\oef)$ it holds that
\begin{align*}
    \|E_{\vareps}^f\phieps\|_{L^2(\Omega)} \le C \|\phieps\|_{L^2(\oef)},\qquad \|\nabla E_{\vareps}^f \phieps \|_{L^2(\Omega)} \le C\|\nabla \phieps\|_{L^2(\oef)}, 
\end{align*}
for a constant $C>0$ independent of $\vareps$. The extension operator $E_{\vareps}^f$ preserves the $\Sigma$-periodicity, more precisely, we have $E_{\vareps}^f:H^1_{\#}(\oef) \rightarrow H^1_{\#}(\Omega)$.
\end{enumerate}
\end{lemma}
\begin{proof}
    The results follow from the local results in \cite{Acerbi1992}, see also \cite{GahnJaegerTwoScaleTools} for thin perforated layers in the context of the symmetric gradient. We only give some additional comments about the $\Sigma$-periodicity of the extension, where we focus on the operator $E_{\vareps}^s$ (for the operator $E_{\vareps}^f$ we can argue in the same way). We also assume for the sake of simplicity that $\oem \setminus \oems$ is disconnected (for the general more technical case we refer to the construction in \cite{Acerbi1992} and \cite{GahnJaegerTwoScaleTools}, but the main idea is the same). We first show that for every $\phieps \in H_{\#}^1(\oems)$ and $l\in \Z^{n-1} \times \{0\}$ it holds that
\begin{align}\label{eq:commuting_prop_extension_shifts}
   ( E_{\vareps}^s\phieps )(\cdot + l\vareps ) = E_{\vareps}^s(\phieps(\cdot + l\vareps)).
\end{align}   
In fact, let $\tau: H^1(Z_s)\rightarrow H^1(Z)$ be the local extension operator from \cite{Acerbi1992} fulfilling for all $\phi \in H^1(Z_s)$
\begin{align*}
    \|\tau \phi \|_{L^2(Z)} \le C \|\phi\|_{L^2(Z_s)} , \qquad \|\nabla \tau \phi \|_{L^2(Z)} \le C \| \nabla \phi\|_{L^2(Z_s)}.
\end{align*}
Then the global extension operator $E_{\vareps}^s$ is defined in the following way. First of all, we define  $\widehat{\Omega}_{\vareps}^{M,s} := \bigcup_{k \in \Z^{n-1} \times \{0\}} \vareps (Z_s + k)$ and extend $\phieps \in H^1_{\#}(\oems)$ by periodicity to the whole domain $\widehat{\Omega}_{\vareps}^{M,s}$ to obtain a function (same notation) $\phieps \in H^1_{\mathrm{loc}}(\widehat{\Omega}_{\vareps}^{M,s})$. Now, for almost every $x \in \vareps (k + Z_s)$ with $k \in \Z^{n-1} \times \{0\}$ the extension operator $E_{\vareps}^s$ is defined by
\begin{align*}
    E_{\vareps}^s \phieps(x) = \tau \phieps^k \left(\fxe - k\right)
\end{align*}
for $\phieps^k := \phieps(\vareps( k +\cdot))$. Hence, for $l \in \Z^{n-1}\times \{0\}$ we get, since $x + l\vareps \in \vareps (k + l + Z_s)$,
\begin{align*}
(E_{\vareps}^s\phieps)(x + l\vareps ) = \tau \phieps^{k+l}\left(\frac{x + l\vareps}{\vareps} - (k +l)\right) = \tau \phieps^{k+l}\left(\frac{x}{\vareps} - k\right) = (E_{\vareps}^s (\phieps(\cdot + l\vareps))) (x).
\end{align*}
This implies $\eqref{eq:commuting_prop_extension_shifts}$. Now, choosing $l= (b-a,0) \in \Z^{n-1}\times 0$ we get the $\Sigma$-periodicity of $E_{\vareps}^s \phieps$, because $\phieps(\cdot + l\vareps) = \phieps$.

\end{proof}

\begin{remark}\label{rem:extension_operator}
The proof of Lemma \ref{lem:Extension_operators} shows that the extension operators commute with the shift operator, more precisely $\eqref{eq:commuting_prop_extension_shifts} $ holds.

\end{remark}
The existence of an extension operator for the perforated fluid domain $\oef$ immediately implies the following Sobolev embedding with embedding constant independently of $\vareps$:
\begin{corollary}\label{cor:embedding_H1_L6_oef}
Let $n \geq 3$ and $2^{\ast} = \frac{2n}{n-2}$ (the case $n=2$ is clear). There exists a constant $C>0$ independent of $\vareps$ such that for every $\xieps^f\in H^1(\oef)$ it holds that
\begin{align*}
\|\xieps^f\|_{L^{2^{\ast}}(\oef)} \le C \|\xieps^f\|_{H^1(\oef)}.
\end{align*}
\end{corollary}
\begin{proof}
With the extension operator $E_{\vareps}^f$ from Lemma \ref{lem:Extension_operators} we obtain from the continuity of the embedding $H^1(\Omega)\hookrightarrow L^{2^{\ast}}(\Omega)$
\begin{align*}
    \|\xieps^f\|_{L^{2^{\ast}}(\oef)} \le \|E_{\vareps}^f \xieps^f\|_{L^{2^{\ast}}(\Omega)} \le C \|E_{\vareps}^f \xieps^f\|_{H^1(\Omega)} \le C \|\xieps^f\|_{H^1(\oef)}.
\end{align*}
\end{proof}
To estimate the convective term in the transport equation in the thin layer and obtain $\vareps$-uniform \textit{a priori} bounds for the time-derivative of $\ceps^f$, we make use of the following embedding result with explicit $\vareps$-dependence of the embedding constant. It is applied to the fluid velocity $\veps$ which is equal to zero on $\geps$. However, to formulate the result in a more general framework we consider arbitrary functions in $H^1(\oemf)$, which leads to an additional term on a part of the boundary $\partial \oemf$.
\begin{lemma}\label{lem:embedding_thin_layer_H1_Lp}
Let $p \in [2, 2^{\ast})$ with $2^{\ast} := \frac{2n}{n-2}$ for $n\geq 3$ ($n=2$ is clear). Let $G\subset \partial Z_f$ and $|G|>0$. Further, we denote by $G_{\vareps}$ the associated microscopic surface (for example $G_{\vareps} = \geps$ for $G = \Gamma$). Then, for every $\xieps \in H^1(\oemf)$ it holds that 
\begin{align*}
    \|\xieps\|_{L^p(\oemf)} \le  C \left( \vareps^{n \left(\frac{1}{p} - \frac12 \right) + \frac12} \|\xieps\|_{L^2(G_{\vareps})} + \vareps^{n \left(\frac{1}{p} - \frac12 \right) + 1 } \|\nabla \xieps\|_{L^2(\oemf)} \right)
\end{align*}
for a constant $C>0$ independent of $\vareps$.
\end{lemma}
\begin{proof}
We have with $\xieps^k:= \xieps(\vareps ( \cdot + k))$ and the same argument as for $\eqref{ineq:bound_gradient_surface}$
\begin{align*}
    \|\xieps\|_{L^p(\oemf)}^p &= \vareps^n \sum_{k\in K_{\vareps}} \|\xieps^k\|_{L^p(Z_f)}^p  \le C \vareps^n \sum_{k\in K_{\vareps}} \left( \|\xieps^k\|_{L^2(G)} + \|\nabla \xieps^k\|_{L^2(Z_f)} \right)^p
    \\
    &\le C \vareps^n \sum_{k\in K_{\vareps}} \left( \|\xieps^k\|_{L^2(G)}^p + \vareps^p \|(\nabla \xieps)^k \|_{L^2(Z_f)}^p \right)
    \\
    &= C \vareps^n \sum_{k \in K_{\vareps}} \left( \vareps^{p\frac{1-n}{2}} \|\xieps \|_{L^2(\vareps(G + k))}^p + \vareps^{p\left( 1 - \frac{n}{2}\right)} \|\nabla \xieps\|_{L^p(\vareps(Z_f + k)}^p \right).
\end{align*}
Taking this inequality with the power $\frac{1}{p}$, we get (equivalence of $p$-norms)
\begin{align*}
\|\xieps\|_{L^p(\oemf)} &\le C \vareps^{\frac{n}{p}} \left\{ \sum_{k\in K_{\vareps}} \left( \vareps^{p\frac{1-n}{2}} \|\xieps \|_{L^2(\vareps(G + k))}^p + \vareps^{p\left( 1 - \frac{n}{2}\right)} \|\nabla \xieps\|_{L^2(\vareps(Z_f + k)}^p \right) \right\}^{\frac{1}{p}}
\\
&\le C \vareps^{\frac{n}{p}} \left\{ \sum_{k\in K_{\vareps}} \left( \vareps^{1-n} \|\xieps\|^2_{L^2((\vareps(G + k))} + \vareps^{2-n} \|\nabla\xieps\|^2_{L^2(\vareps(Z_f + k))} \right) \right\}^{\frac12}
\\
&\le  C \left( \vareps^{n \left(\frac{1}{p} - \frac12 \right) + \frac12} \|\xieps\|_{L^2(G_{\vareps})} + \vareps^{n \left(\frac{1}{p} - \frac12 \right) + 1 } \|\nabla \xieps\|_{L^2(\oemf)} \right).
\end{align*}
\end{proof}
\begin{remark}
    As a special case we have for $p=2$ the inequality $\eqref{ineq:aux_trace_inequality}$.
\end{remark}
Finally, we give the following well-known lemma giving a pointwise bound for differences of shifts in the intermediate space in a Gelfand-triple.
\begin{lemma}
\label{lem:Nikolskii_estimate}
Let $V$ and $H$ be Hilbert spaces   and we assume that $(V,H,V^{\ast})$ is a Gelfand triple. Let $v \in L^2((0,T),V) \cap H^1((0,T),V^{\ast})$. Then, for every $\phi \in V$ and almost every $t \in (0,T)$, $h\in(-T,T)$, such that $t + h \in (0,T)$, we have
\begin{align*}
\big|(v(t+h) - v(t), \phi)_H\big| \le \sqrt{|h|} \|\phi\|_V \|\partial_t v \|_{L^2((t,t+h),V^{\ast})}.
\end{align*}
In particular, it holds that
\begin{align*}
\big\|v(t+h) - v(t) \big\|^2_H \le \sqrt{|h|} \big\|v(t+h) - v(t)\big\|_V \|\partial_t v\|_{L^2((t,t+h),V^{\ast})}.
\end{align*}
\end{lemma}
\begin{proof}
    For a proof see \cite[Lemma 7.2]{GahnNeussRaduKnabner2018a} respectively \cite[Lemma 9]{Gahn}
\end{proof}

\section{Two-scale convergence and the unfolding method}
\label{sec:two_scale_convergence_unfolding}
In this section we first briefly repeat the definition of the two-scale convergence  for thin (perforated) layers, see \cite{BhattacharyaGahnNeussRadu,GahnEffectiveTransmissionContinuous,NeussJaeger_EffectiveTransmission}, and recall some known compactness results used in this paper. Further, we give some new two-scale compactness results for the time-derivative as well as some compactness results for suitable bounded sequences in $H^1(\oef)$ together with an interface condition for the limit function.
\begin{definition}\
\begin{enumerate}
[label = (\roman*)]
\item\,[Two-scale convergence in the thin layer $\oem$] We say the sequence $\weps \in L^2((0,T)\times \oem)$ converges (weakly) in the two-scale sense to a limit function $w_0 \in L^2((0,T)\times  \Sigma \times Z)$ if 
\begin{align*}
\lim_{\vareps\to 0} \foe \int_0^T \int_{\oem} \weps(t,x) \phi \tbxfxe dxdt = \int_0^T\int_{\Sigma} \int_Z w_0(t,\x,y) \psi(t,\x,y) dy d\x dt
\end{align*}
for all $\phi \in L^2((0,T)\times \Sigma,C_{\#}^0(\overline{Z}))$. We write $\weps \rats w_0$.
If additionally it holds that 
\begin{align*}
  \lim_{\vareps\to 0}  \frac{1}{\sqrt{\vareps}} \|\weps\|_{L^2((0,T)\times \oem)}  = \|w_0\|_{L^2((0,T)\times \Sigma \times Z)}
\end{align*}
we say that the sequence $\weps$ converges strongly in the two-scale sense to $w_0$ and we write $\weps \rasts w_0$.
\item\,[Two-scale convergence on the oscillating surface $\geps$] We say the sequence $\weps \in L^2((0,T)\times \geps)$ converges (weakly) in the two-scale sense to a limit function $w_0 \in L^2((0,T)\times  \Sigma \times \Gamma)$ if 
\begin{align*}
\lim_{\vareps\to 0}  \int_0^T\int_{\geps} \weps(t,x) \phi \tbxfxe dx dt = \int_0^T \int_{\Sigma} \int_\Gamma w_0(t,\x,y) \psi(t,\x,y) dy d\x dt
\end{align*}
for all $\phi \in C^0([0,T]\times \overline{\Sigma},C_{\#}^0(\Gamma))$. We write  
$\weps \rats w_0$ on $\geps$.
If additionally it holds that
\begin{align*}
\lim_{\vareps \to 0} \|\weps\|_{L^2((0,T)\times \geps)} = \|w_0\|_{L^2((0,T)\times \Sigma \times \Gamma)},
\end{align*}
we say that the sequence $\weps$ converges strongly in the two-scale sense to $w_0$ on $\geps$ and we write $\weps \rasts w_0$ on $\geps$.
\end{enumerate}
\end{definition}
\begin{remark}
The definition of the two-scale convergence on the surface $\geps$ can also be formulated on $\geps\cup S_{\vareps,f}^{\pm}$ and the following results are also valid in this case.
\end{remark}
Next, we give some basic compactness results for the two-scale convergence in thin layers. 
\begin{lemma}\label{LemmaBasicTSCompactness}\
\begin{enumerate}
[label = (\roman*)]
\item Every sequence $\weps \in L^2((0,T)\times \oem)$ with 
\begin{align*}
    \frac{1}{\sqrt{\vareps}} \|\weps\|_{L^2((0,T)\times \oem)} \le C
\end{align*}
has a (weakly) two-scale convergent subsequence.
\item  Every sequence $\weps \in L^2((0,T)\times \geps)$ with 
\begin{align*}
     \|\weps\|_{L^2((0,T)\times \geps)} \le C
\end{align*}
has a (weakly) two-scale convergent subsequence on $\geps$.
\item Let $\gamma \in [-1,1]$ and $\weps \in L^2((0,T), H^1(\oem))$ be a sequence with
\begin{align*}
\frac{1}{\sqrt{\vareps}}\Vert \weps \Vert_{L^2((0,T)\times \oem)} + \vareps^{\gamma}\Vert \nabla \weps \Vert_{L^2((0,T)\times \oem)}  \le C.
\end{align*}
Then we have the following cases:
\begin{enumerate}
[label = (\arabic*)]
\item For $\gamma = 1$ there exists $w_0 \in  L^2( (0,T)\times \Sigma, H_{\#}^1(Z)/\R)^3$, such that up to a subsequence
\begin{align*}
\weps  \rats w_0,\qquad 
\vareps \nabla \weps \rats \nabla_y w_0
\end{align*}

\item For $\gamma \in (-1,1)$ there exists $w_0 \in L^2((0,T)\times \Sigma)$, such that up to a subsequence  $\weps \rats w_0$.
\item For $\gamma = -1$ there exist $w_0 \in L^2((0,T),H^1(\Sigma))$ and $w_1 \in L^2((0,T)\times \Sigma , H_{\#}^1(Z))$ such that up to a subsequence
\begin{align*}
    \weps \rats w_0, \qquad \nabla \weps \rats \nabla_{\x} w_0 + \nabla_y w_1.
\end{align*}
If additionally it holds that $\weps \in L^2((0,T),H_{\#}^1(\oem))$, then we have $w_0\in L^2((0,T),H^1_{\#}(\Sigma))$.
\end{enumerate}

\end{enumerate}
\end{lemma}
\begin{proof}
See \cite{GahnEffectiveTransmissionContinuous} for the different cases $\gamma$ and \cite{BhattacharyaGahnNeussRadu} for the surface results. The periodicity in the case $\gamma = -1$ is obvious.
\end{proof}

Next, we repeat the definition of the unfolding operator for thin layers, see \cite{NeussJaeger_EffectiveTransmission}. We also refer to \cite{CioranescuGrisoDamlamian2018} for a general overview of the unfolding method.
\begin{definition}
Let $Z_{\ast} \subset Z$ (usually we have $Z_s$, $Z_f$, or $Z$ for $Z_{\ast}$) open and $\oeps^{M,\ast}$ is the associated thin microscopic layer ($\oems$, $\oemf$, or $\oem$). Then, for $p\in [1,\infty]$ we define the unfolding operator 
\begin{align*}
\teps^M: L^p((0,T)\times \oeps^{M,\ast})& \rightarrow L^p((0,T)\times \Sigma \times Z_{\ast}),
\\
\teps^M\ueps (t,\x,y) &= \ueps\left(t,\vareps \left[\frac{\x}{\vareps}\right] + \vareps \y , \vareps y_n\right).
\end{align*}
\end{definition}
Let us summarize some well-known properties of the unfolding operator:
\begin{lemma}
Let $p \in [1,\infty)$. 
\begin{enumerate} [label = (\roman*)]
\item For $\weps \in L^p((0,T)\times \oem)$ it holds that
\begin{align*}
    \|\teps^M \weps\|_{L^p((0,T)\times \Sigma \times Z)} = \vareps^{-\frac{1}{p}} \|\weps\|_{L^p((0,T)\times \oem)}.
\end{align*}
\item For $\weps \in L^p((0,T),W^{1,p}(\oem))$ we have $\teps^M \weps \in L^p((0,T)\times \Sigma,W^{1,p}(Z))$ and $\nabla_y \teps \weps = \vareps \teps^M \nabla \weps$.
\end{enumerate}
\end{lemma}
We have the following important relation between the two-scale convergence and the convergence of the associated unfolded sequence:
\begin{lemma}
A sequence $\weps \in L^2((0,T)\times \oem)$ converges weakly/strongly in the two-scale sense to a limit function $w_0 \in L^2((0,T)\times \Sigma \times Z)$, if and only if $\teps^M \weps$ converges weakly/strongly in $L^2((0,T)\times \Sigma \times Z)$. The same result is valid for sequences $\veps \in L^2((0,T)\times \geps)$.
\end{lemma}

The following lemma is somehow well-known in the literature and often used to derive the strong two-scale convergence on the surface $\geps$ by using the strong two-scale convergence in $\oeps^M$ (or in a general perforated domain). For the sake of completeness we formulate it as a general lemma and give the proof.
\begin{lemma}\label{lem:strong_ts_convergence_boundary}
Let $\ueps \in H^1(\oeps^M)$ and $u_0 \in L^2(\Sigma)$, such that for $\gamma \in [-1,1)$ it holds that
\begin{align*}
    \frac{1}{\sqrt{\vareps}} \|\ueps\|_{L^2(\oeps^M)} + \vareps^{\frac{\gamma}{2}} \|\nabla \ueps \|_{L^2(\oeps^M)} \le C,
\end{align*}
and $\ueps \rasts u_0$. Then for every $\beta \in (\frac12,1)$ it holds that
\begin{align*}
\teps^M \ueps &\rightarrow u_0  &\mbox{ in }& L^2( \Sigma ,H^{\beta}(Z))),
\\
\teps^M \ueps|_{\Gamma} &\rightarrow u_0 &\mbox{ in }& L^2(\Sigma, H^{\beta -\frac12} (\Gamma )).
\end{align*}
In particular, we obtain the strong two-scale convergence of $\ueps$ to $u_0$ on $\geps $. The result is also true if we replace $Z$ by the perforated cell $Z_s$. Of course, we can also replace $\Gamma$ by $\Gamma \cup S_s^{\pm}$.
\end{lemma}
\begin{proof}
First of all, the strong two-scale convergence of $\ueps$ implies $\teps^M \ueps \rightarrow u_0$ in $L^2(\Sigma \times Z)$.
From the continuity of the embedding   $H^1(Z)\rightharpoonup H^{\beta}(Z)$ and the properties of the unfolding operator we obtain (using that $u_0$ is independent of $y$)
\begin{align*}
\|\teps^M \ueps - u_0 \|_{L^2(\Sigma , H^{\beta}(Z))} &\le C \left( \|\teps^M \ueps - u_0 \|_{L^2(\Sigma \times Z)} + \|\nabla_y  \teps^M  \ueps \|_{L^2(\Sigma \times Z)} \right)
\\
&\le C \left( \|\teps^M \ueps - u_0 \|_{L^2(\Sigma \times Z)} + \sqrt{\vareps} \|\nabla \ueps\|_{L^2(\oeps^M)} \right)
\end{align*}
The first term vanishes for $\vareps \to 0$, as well as the second term which is of order $\vareps^{\frac{1 - \gamma}{2}}$. The convergence of the traces follows from the continuous embedding $H^{\beta}(Z) \hookrightarrow H^{\beta-\frac12} (\Gamma )$.
\end{proof}

\subsection{Two-scale convergence of the time-derivative}
\label{sec:Ts_compactness_time_derivative}

The following lemma gives convergence results for the time derivative in a two-scale sense and also regularity of the limit function, just based on the \textit{a priori} estimates for the sequence.

\begin{lemma}\label{lem:two_scale_conv_time_derivative}
Let $\gamma \in [-1,1]$ and  $\ueps \in L^2((0,T) , H^1(\oems)) \cap H^1((0,T),H^1(\oems)^{\ast})$ with
\begin{align*}
\foe \|\partial_t \ueps\|_{L^2((0,T),H_{\vareps,\gamma}(\oems)^{\ast})} +  \|\ueps \|_{L^2((0,T),H_{\vareps,\gamma}(\oems)) } \le C.
\end{align*}
\begin{enumerate}[label = (\roman*)]
\item\label{enumerate_1} Let $\gamma = 1$. There exists $u_0 \in L^2((0,T)\times \Sigma,H^1_{\#}(Z_s)) \cap H^1((0,T),L^2(\Sigma,H_{\#}^1(Z_s))^{\ast})$, such that up to a subsequence $\ueps \rats u_0$. Further, for all $\phieps(t,x) := \phi \left(t,\x,\fxe\right)$ with $\phi \in C_0^{\infty}([0,T)\times \overline{\Sigma}, C_{\#}^0(Z_s))$
it holds that (up to a subsequence)
\begin{align*}
    \lim_{\vareps \to 0} \foe \int_0^T \langle \partial_t \ueps , \phieps \rangle_{H^1(\oems)} dt = \int_0^T \langle \partial_t u_0 , \phi \rangle_{L^2(\Sigma,H_{\#}^1(Z_s))} dt.
\end{align*}

\item\label{enumerate_2} Let $\gamma \in (-1,1)$. There exists $u_0 \in L^2((0,T)\times \Sigma) \cap H^1((0,T),L^2(\Sigma))$, such that up to a subsequence $\ueps \rats u_0$. Further, for all $\phieps(t,x) := \phi \left(t,\x,\fxe\right)$ with $\phi \in C_0^{\infty}((0,T)\times \overline{\Sigma}, C_{\#}^0(Z_s))$ and all $\eta \in L^2((0,T),H^1(\Sigma))$
it holds that (up to a subsequence)
\begin{align*}
    \lim_{\vareps \to 0} \foe \int_0^T \langle \partial_t \ueps , \eta + \phieps \rangle_{H^1(\oems)} dt = \int_0^T \int_{\Sigma} \partial_t u_0 \left( \eta +\int_{Z_s} \phi dy \right) d\x  dt.
\end{align*}

\item\label{enumerate_3} Let $\gamma = -1$. There exists $u_0 \in L^2((0,T),H^1(\Sigma))\cap H^1((0,T),H^1(\Sigma)^{\ast})$ such that up to a subsequence
$\ueps \rats u_0$ and with $\eta$ and $\phieps$ as in \ref{enumerate_2}
\begin{align*}
     \lim_{\vareps \to 0} \foe \int_0^T \langle \partial_t \ueps , \eta + \phieps \rangle_{H^1(\oems)} dt  = \int_0^T \int_{Z_s} \langle \partial_t u_0 , \eta + \phi(\cdot_t,\cdot_{\x},y ) \rangle_{H^1(\Sigma)} dy dt.
\end{align*}
\end{enumerate}
All results are also valid for $\Sigma$-periodic functions, i.e., under the assumption  $\ueps \in L^2((0,T),H^1_{\#}(\oems))\cap H^1((0,T),H^1_{\#}(\oems)^{\ast})$ with obvious modifications.
\end{lemma}

\begin{proof}
We use similar arguments as in \cite[Proposition 2.10]{GahnDissertation} (case $\gamma \in [-1,1)$), and \cite[Proposition 4]{GahnNRP} (case $\gamma = -1$ in perforated domains (not thin)), so we refer to these references regarding some elemental calculations. However, here we slightly simplify the proof, get better convergence results, and obtain more regularity in the case $\gamma \in (-1,1)$. In fact, in this case, the time derivative is an element of $L^2((0,T)\times \Sigma)$. We denote for $0<h\ll 1$ the difference quotient with respect to time by $\partial_t^h$. We first show the existence of the time-derivative of the limit function for \ref{enumerate_1} - \ref{enumerate_3}, and finally give the argument for the convergence result.

 To prove \ref{enumerate_1} let $\gamma = 1$ and we obtain the existence of $u_0 \in L^2((0,T)\times \Sigma,H_{\#}^1(Z_s))$ such that up to a subsequence $\chi_{\oems} \ueps \rats \chi_{Z_s} u_0$. We emphasize that we also have the two-scale convergence of the gradients (as in the other cases $\gamma \in [-1,1)$  below), but this is not necessary to establish the regularity of $\partial_t u_0$, and we only need the regularity of the limit function with respect to space at the end of the proof to show the convergence results. Now, for  $\phi \in C_0^{\infty}((0,T)\times \Sigma , H_{\#}^1(Z_s))$  we define $\phieps(t,x):= \phi\left(t,\x,\fxe\right)$ and obtain
\begin{align*}
\langle \partial_t^h u_0 , \phi \rangle_{L^2((0,T-h),L^2(\Sigma,H_{\#}^1(Z_s))}
&= \lim_{\vareps \to 0} \foe \int_0^{T-h} \langle \partial_t^h \ueps , \phieps\rangle_{H_{\vareps,1}(\oems)} dt 
\\
&\le \lim_{\vareps \to 0} \foe \|\partial_t^h \ueps\|_{L^2((0,T),H_{\vareps,1}(\oems)^{\ast})} \|\phieps\|_{L^2((0,T),H_{\vareps,1}(\oems))} 
\\
&\le C \|\phi\|_{L^2(\Sigma,H^1(Z_s))},
\end{align*}
since  $\|\phieps\|_{H_{\vareps,1}(\oems)} \rightarrow \|\phi\|_{L^2(\Sigma,H^1(Z_s))}$ for $\vareps \to 0$ and every $t \in (0,T)$. Hence, $\partial_t^h u_0$ is bounded in $L^2((0,T-h),L^2(\Sigma,H_{\#}^1(Z_s))^{\ast})$ uniformly with respect to $h$, from which we obtain $\partial_t u_0 \in L^2((0,T),L^2(\Sigma,H_{\#}^1(Z_s))^{\ast})$.

Next, we show the regularity for $\partial_t u_0$ in \ref{enumerate_2} and \ref{enumerate_3}, so let $\gamma \in [-1,1)$. First of all, we obtain the existence of $u_0 \in L^2((0,T)\times \Sigma)$ with $\nabla_{\x} u_0 \in L^2((0,T)\times \Sigma)$ for $\gamma =-1$, such that up to a subsequence $\chi_{\oems} \ueps \rats \chi_{Z_s} u_0$.
We choose $\phi \in L^2((0,T),H^1(\Sigma))$, to obtain
\begin{align*}
\langle \partial_t^h u_0 , \phi \rangle_{L^2((0,T-h),H^1(\Sigma))} &= \lim_{\vareps\to 0} \frac{1}{\vareps |Z_s|} \int_0^{T-h} \langle \partial_t^h \ueps , \phi \rangle_{H_{\vareps,\gamma}(\oems)} dt 
\\
&\le \begin{cases}
       C \|\phi\|_{L^2(\Sigma)} &\mbox{ for } \gamma \in (-1,1),
        \\
       C \|\phi\|_{H^1(\Sigma)} &\mbox{ for } \gamma = -1.
    \end{cases},
\end{align*}
where at the end we used
\begin{align*}
    \|\phi\|_{H_{\vareps,\gamma}(\oems)} \overset{\vareps\to 0 }{\longrightarrow} \begin{cases}
        \|\phi\|_{L^2(\Sigma)} &\mbox{ for } \gamma \in (-1,1),
        \\
        \|\phi\|_{H^1(\Sigma)} &\mbox{ for } \gamma = -1.
    \end{cases}
\end{align*}
This implies $\partial_t u_0 \in L^2((0,T),H^1(\Sigma)^{\ast})$. Moreover, for $\gamma \in (-1,1)$ we obtain for all $\phi \in H^1(\Sigma)$ and almost every $t \in (0,T)$
\begin{align*}
    |\langle \partial_t u_0 , \phi\rangle_{H^1(\Sigma)} | \le C \|\phi\|_{L^2(\Sigma)}.
\end{align*}
In other words, $\partial_t u_0$ is a linear functional on  $H^1(\Sigma)$ and bounded (continuous) in $L^2(\Sigma)$. Hence, $\partial_t u_0 $ can be extended to a linear functional on $L^2(\Sigma)$ and by the Riesz-representation theorem we get $\partial_t u_0 \in L^2((0,T)\times \Sigma)$. 

Finally, we have to check the convergence results for test-functions $\phieps$ formulated in the statement. 
For $\gamma = 1$, we first choose  $\phi \in C_0^{\infty}((0,T)\times \overline{\Sigma}, C_{\#}^0(Z_s))$. Now, the result follows by a simple integration by parts in time and using the two-scale convergence for $\ueps$. We emphasize again, that also spatial regularity is necessary for $\ueps$ and $u_0$, such that a integration by parts formula for the generalized time-derivative is valid. Now, the general case $\phi \in C_0^{\infty}([0,T)\times \overline{\Sigma}, C_{\#}^0(Z_s))$ follows by an approximation argument using a cut off function with respect to time near $0$. The case $\gamma \in [-1,1)$ can be shown by similar arguments. However, we emphasize that our approximation argument above fails, so we can only consider the case $\phi(0) = 0$. 
\end{proof}

\begin{corollary}\label{cor:initial_conditions_ts_limit}
Let the assumptions and notations from Lemma \ref{lem:two_scale_conv_time_derivative} hold. Further, we assume that there exists $u_{\mathrm{in}} \in L^2(\Sigma \times Z_s)$, such that $\ueps(0) \rats u_{\mathrm{in}}$. Then
\begin{enumerate}[label = (\roman*)]
 \item for $\gamma = 1$ it holds that $u_0(0) = u_{\mathrm{in}}$ in $L^2(\Sigma \times Z_s)$,
 \item for $\gamma \in [-1,1)$ it holds that $u_0(0) = \frac{1}{|Z_s|} \int_{Z_s} u_{\mathrm{in}} dy$ in $L^2(\Sigma)$.
\end{enumerate}
\end{corollary}
\begin{proof}
We only give the proof for $\gamma = -1$. The other cases follow by similar arguments. Let $\eta \in C_0^{\infty} ([0,T)\times \Sigma)$. Then, by Lemma \ref{lem:two_scale_conv_time_derivative} we have with integration by parts and the two-scale convergence of $\ueps$
\begin{align*}
|Z_s| \int_0^T \langle \partial_t u_0 , \eta \rangle_{H^1(\Sigma)} dt &=\lim_{\vareps\to 0} \foe \int_0^T \langle \partial_t \ueps , \eta \rangle_{H^1(\oems)} dt
\\
&= \lim_{\vareps\to 0} \foe \left\{-\int_0^T \int_{\oems} \ueps \partial_t \eta dx dt - \int_{\oems} \ueps(0) \eta(0) dx \right\}
\\
&= -|Z_s| \int_0^T \int_{\Sigma} u_0 \partial_t \eta d\x dt - \int_{\Sigma} \int_{Z_s} u_{\mathrm{in}} \eta(0) d\x.
\end{align*}
Now, the desired result follows by using again an integration by parts with respect to time.
\end{proof}

\subsection{Interface condition for two-scale sequences in the layer coupled to bulk domains}

In the following lemma we show that the mean value over $S^\pm_f$ of the weak two-scale limit in the layer is equal to the trace of the weak limit in the bulk domains.
We formulate the statement in a more general framework then necessary for our application, i.e., we consider the case when the two-scale limit is depending on the microscopic variable $y$.

\begin{lemma}\label{lem:weak_two_scale_Geps}
Let $\ceps \in H^1(\oef)$ with 
\begin{align*}
\|\ceps \|_{H^1(\oeps^{\pm})} +  \frac{1}{\sqrt{\vareps}} \|\ceps\|_{L^2(\oemf)} + \sqrt{\vareps} \|\nabla \ceps\|_{L^2(\oemf)}   \le C.
\end{align*}
We define $\ceps^{\pm}:= \ceps|_{\oeps^{\pm}}$ and $\ceps^M := \ceps|_{\oemf}$.
Then there exist $c_0^{\pm} \in H^1(\Omega^{\pm})$ and $c_0^M \in L^2(\Sigma , H_{\#}^1(Z_f))$ such that up to a subsequence 
\begin{align*}
    \chi_{\oeps^{\pm}} \ceps^{\pm} &\rightharpoonup c_0^{\pm} &\mbox{ in }& L^2(\Omega^{\pm}),
    \\
    \chi_{\oeps^{\pm}} \nabla \ceps^{\pm} &\rightharpoonup \nabla c_0^{\pm} &\mbox{ in }& L^2(\Omega^{\pm})^n,
    \\
    \chi_{\oemf} \ceps^M &\rats \chi_{Z_f} c_0^M ,
    \\
    \vareps \chi_{\oemf} \nabla \ceps^M &\rats \chi_{Z_f} \nabla_y c_0^M,
    \\
    \ceps^M|_{\geps} &\rats c_0^M &\mbox{ on }& \geps.
\end{align*}
Further it holds that 
\begin{align*}
    c_0^{\pm} = \frac{1}{|S_{f}^{\pm}|} \int_{S_f^{\pm}} c_0^M d\sigma_y \qquad \mbox{on } \Sigma.
\end{align*}
\end{lemma}
\begin{proof}
The convergence results are clear (see Lemma \ref{LemmaBasicTSCompactness}) and we only give the proof for the equation of the traces on the interface $\Sigma$.
Let $\phi^{\pm} \in C_0^{\infty}(\Omega^{\pm} \cup \Sigma)^n$ and $\phi^M \in C_0^{\infty}(\Sigma, C_{\#}^{\infty}(Z))^n$ with $\phi^M(\x,y) = \phi^{\pm}(\x)$ for all $y \in S^{\pm}$, in particular $\phi^M$ is constant with respect to $y$ on $S^{\pm}$. Now, we define
\begin{align*}
    \phieps (x):= \begin{cases}
     \phi^+(x - \vareps e_n) &\mbox{ for } x \in \oeps^+,
     \\
     \phi^M\left(\x, \fxe \right) &\mbox{ for } x \in \oemf,
     \\
     \phi^-(x + \vareps e_n) &\mbox{ for } x \in \oeps^-.
    \end{cases}
\end{align*}
We obtain (in the following we indicate with respect to which domain the outer unit normal $\nu$ is understood)
\begin{align*}
\int_{\Sigma} \int_{\Gamma} c_0^M& \phi^M \cdot \nu_{Z_f} d\sigma_y d\x \overset{\vareps\to 0}{\longleftarrow} \int_{\geps} \ceps^M \phi^M\left(\x,\fxe\right) \cdot \nu_{\oemf} d\sigma 
\\
=& \int_{\oef} \nabla \cdot (\ceps \phieps) dx - \int_{S_{\vareps,s}^{\pm}} \ceps \phi^M\left(\x,\fxe\right) \cdot \nu_{\oeps^{\pm}} d\sigma 
\\
=& \sum_{\pm} \left\{ \int_{\oeps^{\pm}} \nabla \ceps^{\pm} \phi^{\pm} (x \mp \vareps e_n) dx + \int_{\oeps^{\pm} } \ceps^{\pm} \nabla \cdot \phi^{\pm}(x \mp \vareps e_n) dx \right\}
\\
&+ \foe \int_{\oemf} \ceps \left[\vareps \nabla_{\x} \cdot \phi^M\left(\x,\fxe\right) + \nabla_y \cdot \phi^M \left(\x,\fxe\right)\right] dx
\\
&+ \int_{\oemf} \nabla \ceps^M \cdot \phi^M\left(\x,\fxe\right) dx - \int_{S_{\vareps,s}^{\pm}} \ceps \phi^{\pm} \cdot \nu_{\oeps^{\pm}} d\sigma 
\\
\overset{\vareps\to 0}{\longrightarrow}& \sum_{\pm} \left\{\int_{\Omega^{\pm}}  \nabla c_0^{\pm} \cdot \phi^{\pm} dx + \int_{\Omega^{\pm}} c_0^{\pm} \nabla \cdot \phi^{\pm} dx \right\} 
\\
&+ \int_{\Sigma}\int_{Z_f} \nabla_y c_0^M \cdot \phi^M dy d\x + \int_{\Sigma} \int_{Z_f} c_0^M \nabla_y \cdot \phi^M dy d\x - |S_s^{\pm}|\int_{\Sigma}  c_0^{\pm}|_{\Sigma} \phi^{\pm} \cdot \nu_{\Omega^{\pm}} d\x,
\end{align*}
where for the last term we used the two-scale convergence of $\chi_{S_{\vareps,s}^{\pm}} \ceps(x \mp \vareps e_n)$ to $\chi_{S_s^{\pm}} c_0|_{\Sigma}$ on $\Sigma$. Using integration by parts we obtain with $1 - |S_s^{\pm}| = |S_f^{\pm}|$ and $\nu_{\Omega}^{\pm} = \mp \nu_{Z_f}$ on $\Sigma \times S^{\pm}_f$
\begin{align*}
0 &= \sum_{\pm}|S_f^{\pm}| \int_{\Sigma} c_0^{\pm} \phi^{\pm} \cdot \nu_{\Omega^{\pm}} d\x + \sum_{\pm} \int_{\Sigma} \int_{S_f^{\pm}} c_0^M  \phi^M  \cdot \nu_{Z_f} d\sigma_y d\x
\\
&= \sum_{\pm} \int_{\Sigma} \left[ |S_f^{\pm}| c_0^{\pm} - \int_{S_f^{\pm}} c_0^M d\sigma_y \right] \phi^{\pm} \cdot \nu_{\Omega^{\pm}} d\x.
\end{align*}
This gives the desired result.
\end{proof}

\begin{corollary}\label{cor:weak_ts_conv_geps}
Let $\ceps \in H^1(\oef)$ with 
\begin{align*}
\|\ceps \|_{H^1(\oef)} \le C.
\end{align*}
With the notations from Lemma \ref{lem:weak_two_scale_Geps} we obtain  $c_0^+ = c_0^- =: c_0$ on $\Sigma$ and 
\begin{align*}
    \ceps^M \rats c_0 \qquad\mbox{on } \geps.
\end{align*}
\end{corollary}
\begin{proof}
From inequality $\eqref{ineq:aux_trace_inequality}$ we obtain 
\begin{align*}
    \frac{1}{\sqrt{\vareps}} \|\ceps\|_{L^2(\oemf)} + \|\nabla \ceps \|_{L^2(\oemf)} \le C.
\end{align*}
In particular the assumptions of Lemma \ref{lem:weak_two_scale_Geps} are fulfilled and (with the notations from Lemma \ref{lem:weak_two_scale_Geps}) we obtain $c_0^M(\x,y) = c_0^M(\x) \in L^2(\Sigma)$, see Lemma \ref{LemmaBasicTSCompactness}. Further, we obtain
\begin{align*}
    c_0^{\pm} = \frac{1}{|S_f^{\pm}|} \int_{S_f^{\pm}} c_0^M d\sigma_y = c_0^M.
\end{align*}
\end{proof}
\begin{remark}
The corollary above immediately implies the continuity of the limit function $c_0$ across $\Sigma$. This is also obtained from the extension operator in Lemma \ref{lem:Extension_operators}.
\end{remark}

\begin{remark}
All the results can be easily transferred to the time-dependent case.
\end{remark}

\end{appendix}

\subsection*{Acknowledgements}
We acknowledge the contribution of Jonas Knoch (FAU Erlangen-N\"urnberg) to Figure \ref{fig:FigureMicroDomain} developed from \cite{fabricius2023homogenization}.

\bibliographystyle{abbrv}
\bibliography{literature} 

\end{document}